\documentclass[11pt]{article}

\usepackage[utf8]{inputenc}

\usepackage[titletoc,title]{appendix}
\usepackage{authblk}
\usepackage{caption}
\usepackage{url}
\usepackage{amsthm}
\usepackage{float}
\usepackage{hhline}
\usepackage{multicol}
\usepackage{subfiles}
\usepackage{enumerate}
\usepackage[shortlabels]{enumitem}
\usepackage{empheq}
\usepackage{tikz-cd}
\usepackage{stmaryrd}

\usepackage{thmtools}
\usepackage{thm-restate}

\usepackage{ulem}
\usepackage{mathrsfs}

\newcommand*{\charles}[1]{{\textcolor{red}{[\textbf{CA:} #1]}}}

\newcommand*{\david}[1]{{\textcolor{green}{[\textbf{DCS:} #1]}}}

\usepackage{amsmath,amsfonts,amssymb,mathtools}

\usepackage{graphicx,float}

\usepackage{geometry}
\usepackage{comment}

\usepackage{hyperref}
\usepackage[capitalize]{cleveref}

\hypersetup{
pdftitle={Critical points of the distance function to a generic submanifold},
citecolor=red,
colorlinks=true,
linkcolor=blue
}

\author[1]{Charles Arnal$^*$}
\author[2]{David Cohen-Steiner$^*$}
\author[3]{Vincent Divol$^*$}
\affil[1]{Université Paris-Saclay, CNRS, Inria, Laboratoire de Mathématiques d'Orsay}
\affil[2]{Université Côte d’Azur, Inria Sophia Antipolis-Mediterranée}
\affil[3]{CEREMADE, Université Paris-Dauphine, Université PSL}
\newcommand{\eps}{\varepsilon}
\newcommand{\dotp}[1]{\langle #1 \rangle}
\newcommand{\op}[1]{\left\| #1  \right\|_{\mathrm{op}}}
\newcommand{\p}[1]{\left(#1 \right)}

\newcommand{\Med}{\mathrm{Med}}
\newcommand{\Emb}{\mathrm{Emb}}
\newcommand{\diam}{\mathrm{diam}}
\newcommand{\Conv}{\mathrm{Conv}}
\newcommand{\Vol}{\mathrm{Vol}}

\newcommand{\dd}{\mathrm{d}}
\newcommand{\id}{\mathrm{id}}
\renewcommand{\Im}{\mathrm{Im}}

\newcommand{\defeq}{\vcentcolon=}
\newcommand{\eqdef}{=\vcentcolon}

\newcommand{\Z}{{\mathbb Z}}

\newcommand{\R}{{\mathbb R}}
\newcommand{\Q}{{\mathbb Q}}

\newcommand{\A}{{\mathsf A}}

\newcommand{\M}{{\mathcal M}}
\newcommand{\N}{{\mathbb N}}

\newcommand{\U}{{\mathcal U}}

\renewcommand{\S}{{\mathsf S}}
\renewcommand{\M}{{\mathsf M}}

\newcommand{\cA}{\mathcal{A}}

\newcommand{\cJ}{\mathcal{J}}

\newcommand{\cU}{\mathcal{U}}

\declaretheorem[name=Theorem,numberwithin=section]{thm}
\newtheorem{cor}[thm]{Corollary}

\newtheorem{definition}[thm]{Definition}

\newtheorem{remark}[thm]{Remark}
\newtheorem{example}[thm]{Example}
\newtheorem{proposition}[thm]{Proposition}
\newtheorem{lemma}[thm]{Lemma}

\newtheorem{claim}{Claim}

\newcommand{\Pall}{\hyperref[P1]{(P1-4)}}

\begin{document}
\title{Critical points of the distance function to a generic submanifold}
\maketitle

\def\thefootnote{*}\footnotetext{The authors contributed equally to this work.}\def\thefootnote{\arabic{footnote}}





\begin{abstract}
In general, the critical points of the distance function $d_{\mathsf M}$ to a compact submanifold $\mathsf M \subset \mathbb{R} ^D$ can be poorly behaved.
In this article,  we show that this is generically not the case by listing regularity conditions on the critical and $\mu$-critical points of a submanifold and by proving that they are generically satisfied and stable with respect to small $C^2$ perturbations.  
More specifically, for any compact abstract manifold $M$, the set of embeddings $i:M\rightarrow \mathbb R^D$ such that the submanifold $i(M)$ satisfies those conditions is open and dense in the Whitney $C^2$-topology.
When those regularity conditions are fulfilled, we prove  that the distance function to $i(M)$ satisfies Morse-like conditions and that  the critical points of the distance function to an $\varepsilon$-dense subset of the submanifold (e.g., obtained via some sampling process) are well-behaved. 
 We also provide many examples that showcase how the absence of these conditions allows for pathological situations. 
\end{abstract}

\tableofcontents

\section*{Acknowledgements}
We thank Dominique Attali, Frédéric Chazal, Herbert Edelsbrunner, Jisu Kim and Mathijs Wintraecken for helpful discussions. We are particularly grateful to André Lieutier for his instrumental suggestions.

\section{Introduction}

Questions regarding the distance function $d_\M : x \mapsto \inf\{d(x,y) : \; y\in \M \}$ to a submanifold $\M \subset \R^ D$ stand at the intersection of a variety of domains; they feature preeminently in computational geometry \cite{Chazal_CS_Lieutier_OGarticle}, statistics on manifolds \cite{niyogi2008finding, aamari2019estimating, aamari2023optimal} and topological data analysis \cite{edelsbrunner2022computational}, and their study frequently involves tools from classical differential geometry as well \cite{yomdin1981local, mather1983distance}.
In this article, we combine notions and methods from computational geometry and transversality theorems \textit{à la} Thom to understand the behaviour of the critical points of the distance function to a generic submanifold of $\R^D$, and we show that they behave very nicely. We also explore some consequences in terms of sampling of the submanifold and Morse-like behaviour of $d_\M$.

\medskip

\textbf{The distance function and its critical points:} 
In 2004, A. Lieutier \cite{lieutier2004any} proposed a generalized notion of gradient for the distance function $d_\S$ to a compact set $\S\subset \R^D$. Let $x\in \R^D$ and $\pi_\S(x)$ be the set of all projections of $x$ on $\S$,  that is 
\begin{equation}
   \forall y\in \S,\quad  y\in \pi_\S(x) \Longleftrightarrow d_\S(x) = d(x,y).
\end{equation}
 The generalized gradient of $d_\S$ at $x\not\in \S$ is given by
\begin{equation}
    \nabla d_\S(x) = \frac{x-m(\pi_\S(x))}{d_\S(x)},
\end{equation}
where $m(\sigma)$ is the center of the smallest enclosing ball of a non-empty bounded set $\sigma$. This is illustrated in Figure \ref{fig:def_generalized_gradient}. The gradient is extended by $0$ on $\S$, so that $\nabla d_\S$ is defined on $\R^D$, and with this definition the norm $\|\nabla d_{S} \|$ of the generalized gradient is lower semicontinuous (see \cite{lieutier2004any}).
The \textit{medial axis} $\Med (\S)$ is the set of points $z\in \R^ D$ such that $\pi_\S(z)$ is of cardinality at least $2$, i.e.~the points that do not have a unique projection on $\S$. Equivalently, $z\in \Med(\S)$ if and only if $\|\nabla d_{S}(z) \| < 1$ and $z\not \in \S$.

\begin{figure}
    \centering
    \includegraphics[width=0.6\textwidth]{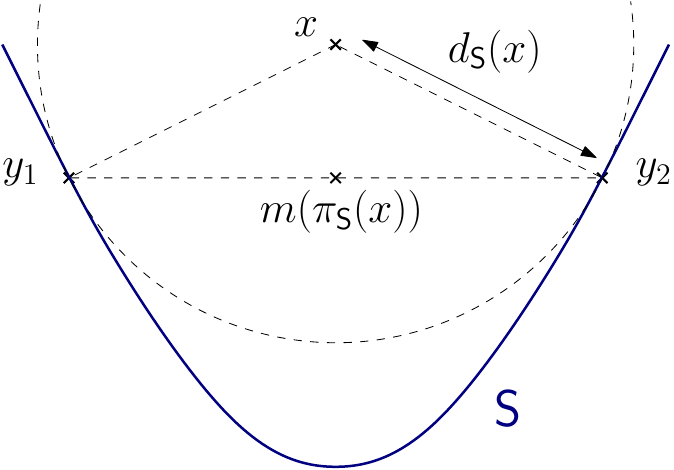}
    \caption{A point $x\in \Med(\S)$ with set of projections $\pi_{\S}(x)=\{y_1,y_2\}$.}
    \label{fig:def_generalized_gradient}
\end{figure}

A \textit{critical point} for the distance function is a point $z\in \R^D\backslash \S$ with $\nabla d_\S(z)=0$, in which case $ d_\S(z)$ is called a \textit{singular value}; we let $Z(\S)$ be the set of critical points for the distance function to $\S$, which we simply call critical points of $\S$. These will be our main object of study.
This definition is generalized as follows: for $\mu \in [0,1]$, one defines the set of \textit{$\mu$-critical} points as
\begin{equation}
    Z_\mu(\S) := \{z\in \R^D\backslash \S : \; \|\nabla d_\S(z)\|\leq \mu\}.
\end{equation} 

The generalized gradient and the associated notions of criticality have many uses in computational geometry and data analysis. Among other applications, they lead to Morse-theoretic statements --   it is for example shown in \cite{GroveCriticalPoints} that two offsets  $\S^{a} = \{x\in\R^D : \; d_\S(x)\leq a\}$ and $\S^b$ are isotopic whenever the segment $[a,b]$  does not contain any critical values of the distance function or $0$.
In topological data analysis, the family of offsets $(\S^t)_{t\geq 0}$ allows one to compute the  \v Cech persistence diagram of $\S$ \cite{chazal2014persistence}, which is routinely used in computations to analyze the topology of the set $\S$, with numerous applications in machine learning (see e.g. \cite{turkes2022effectiveness} or the review \cite{otter2017roadmap} and references therein).
The set of critical values of the distance function corresponds exactly to the coordinates of the points in the \v Cech persistence diagram, and understanding their behavior in different contexts is a central research topic.
Among other uses, the generalized gradient is also employed to build homotopies between an open bounded set and its medial axis \cite{lieutier2004any}, to give guarantees on the homology groups of an approximation of a given shape \cite{chazal2005weak}, or to control the convexity defect function of a set \cite{attali2011vietoris}.

Despite its importance, the set of critical points $Z(\S) $ is still poorly understood.
This is partly because without further assumptions, the set $Z(\S) $  can be almost arbitrarily wild: consider any compact set $\A\subset \R^D$, and let $\S:=\A\times \{-1\} \cup \A\times \{1\} \subset \R^{D+1}$. Then $Z(\S)\cap (\R^D\times [-1/2,1/2]) = \left(\A \cup Z(\A) \right)\times \{0\}$, which is extremely irregular for well-chosen\footnote{One can for example let $A$ be a Cantor set.
} sets $A$.
Another difficulty is the lack of stability of $Z(\S)$ with respect to small perturbations of $\S$, even under strong regularity assumptions on $\S$. Indeed, the map $\S \mapsto Z(\S)$ is not continuous for the Hausdorff distance $d_H$, as is illustrated in Figure \ref{fig:Z_not_continuous}.
This lack of stability can be particularly disastrous in practical situations, where the shape of interest $\S$ (e.g. a submanifold or a  polyhedron) is often only accessible through a finite (possibly noisy) sample $\A$. The set $Z(\A)$ can then be very different from $Z(\S)$, which can lead, in turn, to a poor estimation of the features of interest of $\S$.

\begin{figure}
    \centering
    \includegraphics[width=0.7\textwidth]{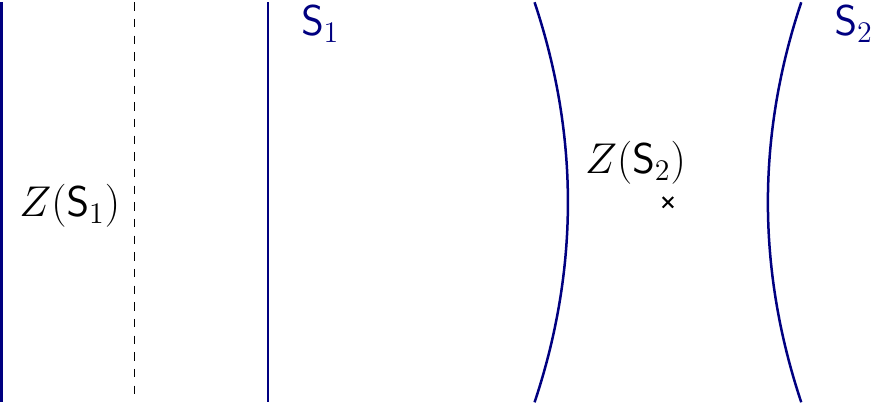}
    \caption{The Hausdorff distance between $\S_1$ and $\S_2$ can be made arbitrarily small while keeping the Hausdorff distance between $Z(\S_1)$ and $Z(\S_2)$ constant.}
    \label{fig:Z_not_continuous}
\end{figure}


The $\mu$-critical points are more stable: as shown by 
F. Chazal, D. Cohen-Steiner and A. Lieutier  in \cite{Chazal_CS_Lieutier_OGarticle}, if $\S$ and $\A$ are two compact sets with $d_H(\S,\A)\leq \eps$ and if $z\in\R^D$, then
\begin{equation}\label{eq:sampling_OG_theorem}
   \text{$z\in Z_\mu(\A) \implies$ $z$ is at distance at most $2\sqrt{\eps d_\S(z)}$ from $Z_{\mu'}(\S)$ for $\mu'= \mu+2\sqrt{\eps/ d_\S(z)}$.  }
\end{equation}
Nonetheless, this remains unsatisfactory when one is interested in actual critical points only.

The main goal of this article is to show that unlike what can be expected in general, the set of critical points
of a \textit{generic} compact $C^2$ submanifold $\M\subset \R^D$ is well-behaved, and remains stable with respect to taking a dense enough sample of $\M$.
Furthermore, both these properties are stable with respect to small $C^2$ perturbations of $\M$.
\medskip

\textbf{Regularity conditions:}
To formalize these statements, let us first clarify what we mean by ``well-behaved".
Here are four conditions that we want our submanifolds to satisfy:

\begin{enumerate}[start=1, label={(P\arabic*)}]
    \item For every $z_0\in Z(\M)$, the projections $\pi_\M(z_0)$ are the vertices of a non-degenerate simplex of $\R^D$. In particular, $z_0$ has at most  $D+1$ projections. Moreover, the point $z_0$ belongs to the relative interior
    of the convex hull of $\pi_\M(z_0)$.\label{P1}
    \item The set $Z(\M)$ is finite. \label{P2}
    \item \label{P3} For every $z_0\in Z(\M)$ and every $x_0\in \pi_\M(z_0)$, the sphere $S(z_0,d_\M(z_0))$ is non-osculating $\M$ at $x_0$, in the sense that there exist $\delta >0$ and $\alpha>0$ such that for all $y\in \M \cap B(x_0,\delta)$,
        \begin{equation}\label{eq:P3}
            \|y-z_0\|^2 \geq   \|x_0-z_0\|^2+\alpha\|y-x_0\|^2.
        \end{equation} 
    \item There exist constants $C>0$ and $\mu_0\in (0,1)$ such that for every $\mu\in [0,\mu_0)$, the set $Z_\mu(\M)$ is included in a tubular neighborhood of size $C\mu$ of $Z(\M)$, that is every point of $Z_\mu(\M)$ is at distance less than $C\mu$ from $Z(\M)$. \label{P4}
    \end{enumerate}

Properties \ref{P1} and \ref{P2} are quite self-explanatory; they control the discreteness of the set of the critical points and of their projections, and the genericity of their relative positions. Property \ref{P3} accepts many equivalent formulations, as shown in Section \ref{sec:osculation}, and is essentially a property of the curvature of $\M$ at the projection $x_0$. It is crucial to the stability of the projections of the critical points, as will become apparent later.
Finally, Property \ref{P4} controls the position of the $\mu$-critical points relative to the critical points.  It is equivalent to a local condition on the critical points and their projections, as shown in Section \ref{sec:BSP}; somewhat surprisingly, it also plays a role in the stability of the critical points themselves, as can be seen in Section \ref{sec:C2_stability}. 

\medskip

\textbf{Genericity:}
We also need definitions that appropriately capture the idea of genericity for a manifold.
Consider an abstract $C^r$ manifold $M$ for some $r\in \N\cup \{\infty\}$, and let $\Emb^k(M,\R^D)$ be the set of $C^k$ embeddings from $M$ to $\R^D$. For $0\leq l\leq k$, we can endow $\Emb^k(M,\R^D)$ with the Whitney $C^l$-topology, see \Cref{sec:notations_definitions}.
Intuitively, two embeddings $i,i'\rightarrow \R ^D$ are close for the Whitney $C^l$-topology if their first $l$ differentials (including their ``0"th differentials, i.e.~their values) are uniformly close over $M$ when expressed in the coordinates of a fixed, finite atlas of $M$.
In what follows, we are interested in showing that certain properties are such that for any $C^k$ compact manifold $M$, the set of embeddings $i $ such that $i(M)$ verifies the property is dense and open in $\Emb^k(M,\R^D)$ for the Whitney $C^l$-topology (for some $l\leq k \in \N$).
Openness captures the idea that the property should be stable with respect to small perturbations; density captures the idea that most submanifolds should satisfy the condition, as any submanifold can be made to satisfy it with an arbitrarily small modification.
The conjunction of openness and density further reinforces this idea; by contrast, density alone can result in counter-intuitive situations, as $\Q$ is dense in $\R$, yet few would claim that most real numbers are rational.\footnote{This nuance can be particularly important for practical applications, as by nature computers can only manipulate rational numbers: one could imagine some algorithm that is guaranteed to perform well on a dense set of irrational inputs, yet fails for any rational input.}

\medskip

\textbf{The Genericity and Stability Theorems:}
We can now formulate our first main result, i.e.~that generic $C^2$ submanifolds have well-behaved critical points and $\mu$-critical points in the sense defined above.

\begin{thm}[The Genericity Theorem]\label{thm:generic}
    Let $M$ be a compact $C^k$ manifold for some $k\in \N\cup\{\infty\}$ with $k\geq 2$. Then the set of $C^k$ embeddings $i:M\to \R^D$ such that $i(M)$ satisfies Conditions \Pall\ is dense for the Whitney $C^k$-topology and open for the Whitney $C^2$-topology in $\Emb^k(M,\R^D)$.
\end{thm}

\begin{remark}
As the Whitney $C^{l_1}$-topology is finer than the Whitney $C^{l_2}$-topology for any $l_1\geq l_2$, the statement implies that for any $2\leq l\leq k$, the set of appropriate embeddings is dense and open in $\Emb^k(M,\R^D)$ for the Whitney $C^l$-topology.
\end{remark}

\begin{remark}\label{rmk:C1_also_ok}
The density part of the statement easily extends to $C^1$ submanifolds: if $M$ is a compact $C^1$ manifold, then the set of $C^1$ embeddings $i:M\to \R^D$ such that $i(M)$ satisfies Conditions \Pall\ is dense  in $\Emb^1(M,\R^D)$ for the Whitney $C^1$-topology.
We justify this remark in the second part of the proof of  Theorem \ref{thm:generic}, at the end of Section \ref{sec:density}.
\end{remark}

While openness is proved through arguments similar to those used in \cite{Chazal_CS_Lieutier_OGarticle} or \cite{lieutier2004any}, we use a more abstract method for the density part of Theorem \ref{thm:generic}.
Indeed, while it is relatively easy to perturb any local non-generic situation into a generic one, ensuring that the submanifold $i(M)$ satisfies \Pall\ globally is more challenging.
Inspired by Y. Yomdin's work in \cite{yomdin1981local}, we proceed as follows: we define various subsets $W$ of the space of multijets such that embeddings whose associated multijet is transverse to $W$ satisfy good global geometric properties, and we apply variants of Thom's transversality theorem to show that the set of transverse embeddings is dense in the appropriate topology (all these notions are properly introduced in Section \ref{sec:density}).

Various examples suggest that the openness part of the statement cannot be easily extended to $C^1$ submanifolds; in fact, it is even surprising that the theorem  holds for $C^2$ manifolds, as most straightforward arguments seem to force us to resort to the Whitney $C^3$-topology (see e.g., Remark \ref{rmk:openness_alternative_proof}).
Showing that the property is true for the $C^2$-topology greatly increases the difficulty of the proof.
On the other hand, this rather convoluted proof yields as a byproduct a stronger stability result than the openness part of the Genericity Theorem \ref{thm:generic}, which we call the $C^2$ Stability Theorem.
Let $\M$ be a submanifold that satisfies Properties \Pall; then the theorem states that not only the properties themselves, but the numbers and positions of the critical points of $\M$ and of their projections are stable with respect to small $C^2$ perturbations, in the sense detailed below.

\begin{restatable}[The $C^2$ Stability Theorem]{thm}{stabilitythm}\label{thm:C2_stab}
Let $i\in \Emb^2(M,\R^D)$ be an embedding such that $\M=i(M)$ satisfies \Pall. Then, for all $\eps>0$ small enough, there exists a neighborhood $\cU$ of $i$ in $\Emb^2(M,\R^D)$ equipped with the Whitney $C^2$-topology  such that for every $i'\in \cU$, with $\M'=i'(M)$:
    \begin{enumerate}
        \item $\M'$ satisfies  \Pall. 
        \item There is a bijection $\Psi: Z(\M) \rightarrow Z(\M')$.
        \item Let $z_0\in Z(\M)$; then $\Psi(z_0)\in B(z_0,\eps)$.
        \item Let $z_0\in Z(\M)$ with  $\pi_{\M}(z_0)=\{x_1,\dots,x_s\}$. Then $\pi_{\M'}(\Psi(z_0))$ is included in the union of balls of radius $\eps$ centered at the points $x_j$, with exactly one point in each ball. 
    \end{enumerate}
\end{restatable}

\begin{remark}
    We only ask that $\eps$ be small enough to ensure that the balls of point 4. do not intersect. 
\end{remark}

The $C^2$ Stability Theorem shows that Properties \Pall, which we know to be generic in the sense of Theorem \ref{thm:generic}, are desirable not only for themselves, but also for the stability of the critical points and of their projections with respect to $C^2$ perturbations that they guarantee - various examples throughout this paper illustrate how removing some of the properties can lead to somewhat pathological instabilities.

\medskip

\textbf{Sampling theory on generic submanifolds:}
The following theorem states that Properties \Pall\ also induce another type of stability, this time with respect to taking an $\eps$-dense subset $\A$ of the submanifold $\M$ - a typical example would be for $\A$ to be some point cloud sampled on $\M$.



\begin{restatable}[Stability Theorem for Subsets]{thm}{subsetthm}\label{thm:crit_points}
    Let $\M \subset \R^D$ be a compact $C^2$ submanifold. There exist positive constants $C_1,C_2,C_3,C_4, C_5$ and $\eps_0$ such that the following holds. Let $\A\subset \M$ be a set with $d_H(\A,\M)\leq \eps \leq \eps_0$.
    \begin{enumerate}
        \item For every critical point $z\in Z(\A)$, exactly one of these two possibilities is true: either $z$ is very close to $\M$, that is $d_\M(z) \leq C_1\eps^2$, or it is close to a $\mu$-critical point of $\M$, that is $d_{Z_\mu(\M)}(z)\leq C_2\eps$ for $\mu = C_3\eps$.
        \item  If $\M$ satisfies \ref{P4}, then, in the second case, the critical point $z\in Z(\A)$ is close to $Z(\M)$, that is $d_{Z(\M)}(z)\leq C_4\eps$. \label{it:usingP4}
        \item If $\M$ additionally satisfies conditions \ref{P1}, \ref{P2} and \ref{P3} then, in the second case, the set of projections $\pi_\A(z)$ of the critical point $z\in Z(\A)$ is included in a neighborhood of size $C_5\eps$ of $\pi_\M(z')$ for some $z'\in Z(\M)$ such that $d(z,z')\leq C_4\eps$. 
    \end{enumerate}
\end{restatable}

This is to be compared to the previously available guarantees for general compact sets with positive reach (see definition in Section \ref{sec:notations_definitions}): let $\S \subset \R^D$ be a compact set with positive reach $\tau(\S)$, and let $\A \subset \R^D$ be such that $d_H(\A,\S)\leq \eps$.
Then Chazal, Cohen-Steiner and Lieutier \cite{Chazal_CS_Lieutier_OGarticle} show that all critical points of $\A$ are either at distance $4\eps$ from $\S$, or are at distance at least $\tau(\S)-3\eps$ from $\S$. Hence, assuming that $\eps< \tau(\S)/7$, there are no critical values of $d_\A$ in the interval  $(4\eps, \tau(\S)/2)$. In that case, \Cref{eq:sampling_OG_theorem} implies that the critical points of $\A$ that are ``far from $\S$'' are only $O(\sqrt{\eps})$-close to $\mu$-critical points for $\mu=O(\sqrt{\eps})$. 

Hence the added hypotheses on $\M$ yield several improvements with respect to the general case. Firstly,\footnote{This first point was well-known folklore, frequently mentioned by F. Chazal among others.} the critical points that are close to $\M$ are at distance $O(\eps^2)$ from $\M$ instead of $O(\eps)$. 
Secondly, the stability with respect to $\mu$-critical points is also improved, going from $O(\sqrt{\eps})$ to $O(\eps)$ for both the distance to $Z_\mu(\M)$ and the value of $\mu$. Note that these two improvements are solely due to the set $\M$ being a submanifold and to $\A$ being a subset of $\M$, and hold without any of the conditions \Pall. Thirdly and finally, Properties \Pall\ allow us to get the $O(\eps)$-stability of both critical points and their projections.


As a corollary to \Cref{thm:crit_points}, we get a bound on the number of critical points that a finite approximation (i.e.~a point cloud) $\A$ of $\M$ can have.
We say that a finite subset $\A\subset \M$  is a \textit{$(\delta,\eps)$-sampling of $\M$} if $d_H(\A,\M)\leq \eps$ and $\min_{x\neq y\in \A} \|x-y\|\geq \delta$. For $\delta,\eps>0$, the $\delta$-packing number $N(\delta,\eps)$ of $\M$ is  the maximal cardinality of a $(\delta,+\infty)$-sampling of a ball $B(x,\eps)\cap \M$ centered at a point $x\in \M$.

\begin{cor}\label{cor:deterministic_sample}
     Let $\M \subset \R^D$ be a  compact $C^2$ submanifold of dimension $m$ satisfying \Pall. Consider the constants $C_4$ and $C_5$ from \Cref{thm:crit_points}. There exists $\eps_1>0$ such that if  $\A$ is a $(\delta,\eps)$-sampling of $\M$ for some $0<\delta \leq \eps\leq \eps_1$, and if
     $z\in Z(\M)$ and  $s$ is the cardinality of the set of projections $\pi_\M(z)$, then the number of critical points in $Z(\A)$ at distance less than $C_4\eps$ from $z$ is bounded by 
     $ 2^{N(\delta,C_5\eps){s}}$.
     Furthermore 
     \[N(\delta,C_5\eps) \leq (C_mC_5\eps/\delta)^{m}\]
     for some $C_m>0$ that depends on $m$.
\end{cor}

The practical applicability of \Cref{cor:deterministic_sample} relies on the existence of $(\delta,\eps)$-samplings with $\delta$ of the same order as $\eps$. The farthest point sampling algorithm \cite{aamari2018stability} shows that such sets always exist: starting from a finite subset $\A_0\subset \M$ with $d_H(\A_0,\M)\leq \eps$, it outputs a subset $\A_1\subset \A_0$ that is an $(\eps,2\eps)$-sampling of $\M$. For such a subset $\A_1$, \Cref{cor:deterministic_sample} ensures that the number of critical points at distance at least $O(\eps^2)$ from $\M$ is at most $K2^{(D+1)(2C_mC_5)^m}$, where $K$ is the cardinality of $Z(\M)$. In particular, this quantity stays bounded as $\eps$ goes to $0$.

\medskip

\textbf{Morse-like properties of the distance function to a generic submanifold:} 
Properties~\Pall\ and the Genericity and $C^2$ Stability Theorems were primarily motivated by considerations on the regularity and stability of the critical points of $\M$ and of their projections.
However, Properties \Pall\ also happen to enforce a Morse-like behaviour on the evolution of the topology of the offsets $\M^a = \{x\in\R^D : \; d_\M(x)\leq a\}$.
Indeed, Properties \Pall\ are exactly those needed for the critical points of the distance function $d_\M$ to be \textit{Min-type non-degenerate critical points} as defined by V. Gershkovich and H. Rubinstein in \cite{gershkovich1997morse}; if they are satisfied, it is easy to show that the restriction of $d_\M$ to $\R^D\backslash \M$ is a \textit{topological Morse function}, a natural generalization of the usual definition of (smooth) Morse functions that was first defined in \cite{morse_topologically_1959}.
A topological Morse function $f$ is one such that for every point $x$ of its domain $X$, there is a local homeomorphic change of coordinates $\phi: (U,0) \rightarrow (V,x)$ with $U \subset \R^D, V \subset X$ such that either $f\circ \phi(y_1,\ldots,y_d) = f(x) + y_d$, in which case $x$ is a \textit{topological regular point} of $x$, or $f\circ \phi(y_1,\ldots,y_d) = f(x) - \sum_{i=1}^\lambda y_i^2 + \sum_{i=\lambda + 1}^d y_i^2$ for some $\lambda\in\{0,\ldots,d\}$, in which case $x$ is a \textit{non-degenerate topological critical point of index $\lambda$} of $f$ (more details can be found in Section \ref{sec:Morse}). 
In particular, topological Morse functions satisfy the same handle attachment property as classical Morse functions.
As a result, the following theorem holds.
\begin{restatable}[Morse-Like Theorem for the Distance Function]{thm}{Morsethm}\label{thm:Morse}
  Let $\M \subset \R^D$ be a  compact $C^2$ submanifold satisfying \Pall. 
  Then the restriction $d_\M:\R^D\backslash \M \rightarrow \R$ of the distance function to $\M$ is a topological Morse function, whose topological critical points are exactly $Z(\M)$.
  It satisfies the two following properties:
  \par \textbf{(Isotopy Lemma)} Let $0\leq a<b$ be such that $d_\M^{-1}[a,b]$ contains no critical point. Then $\M^a$ is a deformation retract of $\M^b$.
 \par \textbf{(Handle Attachment Lemma)} 
    Let $c>0$ and $0<\epsilon <c$ be such that $d_\M^{-1}[c-\eps,c+\eps]$ contains no critical point except for $z_1,\ldots,z_k \in d_\M^{-1}(c)$, where $z_j$ is of index $i_j$ as a topological critical point.
    Then $\M^{c+\eps}$ is homotopically equivalent to $\M^{c-\eps}$ with cells $D^{i_1},\ldots,D^{i_k}$ of dimension $i_1,\ldots,i_k$ attached,
    i.e.
    \[\M^{c+\eps}\simeq \M^{c-\eps}\cup D^{i_1}\cup\ldots \cup D^{i_k}.\]
\end{restatable}
Similar observations have been made in the case of surfaces in $\R^3$ in \cite{song2023generalized}.
Note that without being restricted, $d_\M :\R^D \rightarrow\R$ cannot be a topological Morse function, as the points of $\M$ themselves are degenerate topological critical points.
Note also that $\M$ does not need to satisfy Conditions \Pall\ for the Isotopy Lemma to hold;
as mentioned above, it was shown in \cite{GroveCriticalPoints} to be the case for any compact set when $a>0$, and the same arguments easily extend to the case $a\geq 0$ when $\M$ is any submanifold.
We nonetheless include this fact in the statement of the theorem to further stress the connection to usual Morse theory.

Thus, Properties \Pall\ grant us control over the changes of topology of the offsets of a generic submanifold; this has interesting consequences regarding the \v Cech persistence diagram of $\M$ (see e.g. \cite{chazal2016structure}, \cite{carlsson2021topological}), a key object in topological data analysis, which will be further explored in future work.

\subsection{Related work}
Given a compact set $\A$, both the medial axis $\mathrm{Med}(\A)$ and the skeleton $\mathrm{Sk}(\A)$ (the set of centers of maximal balls not intersecting $\A$) are well studied in the literature \cite{giblin2000symmetry, giblin2004formal, lieutier2004any, chazal2004stability}. The two concepts are tightly linked, as $\mathrm{Med}(\A) \subset\mathrm{Sk}(\A) \subset \overline{\mathrm{Med}(\A)}$, the latter inclusion being an equality in all ``tame'' cases. We refer to \cite{attali2009stability} for a review of the (in)stability properties of the medial axis. 
Regarding specifically the issue of the stability of the critical and $\mu$-critical points with respect to sampling, our main inspiration was \cite{Chazal_CS_Lieutier_OGarticle}, which considers the case of general compact sets.

The genericity of transversality is another key theme of this article; the first transversality theorems for jets date back to Thom's \cite{thom_transversality}, later extended to multijets by J. Mather in \cite{MATHER1970b}.
Early work on the application of those theorems to the distance function in Euclidean space include \cite{LooijengaThesis} and \cite{Wall1977GeometricPO}, which eventually led to the study of the properties of the skeleton $\mathrm{Sk}(\M)$ of a generic manifold $\M$ by Yomdin \cite{yomdin1981local}, who proves among other results that the set of embeddings $i$ of a manifold $M$ such that for all $z\in \mathrm{Sk}(i(M))$,  the projection set $\pi_{i(M)}(z)$ is a non-degenerate simplex, is dense (and in fact residual), then by Mather \cite{mather1983distance}. Later, J. Damon and E. Gasparovic gave in a series of papers (see \cite{damon1997generic, DamonGasparovic} and references therein)  a thorough analysis of the geometry of the skeleton set for generic sets in different contexts.
Very recently, A. Song, K.M. Yim and A. Monod \cite{song2023generalized} studied the set of critical points of a generic submanifold, though only for surfaces $\S$ embedded into $\R^3$ and  only in terms of density (without openness and stability), with a greater focus on Morse-like properties of the distance function.

\subsection{Outline }
We recall a few definitions and prove an elementary lemma in \Cref{sec:notations_definitions}, before studying in detail the non-oscularity condition \ref{P3} in \Cref{sec:osculation}; in particular, we rephrase it as a curvature condition.
In \Cref{sec:projections}, we prove some properties of the projection on a submanifold $\M$ wherever it is defined and in particular under some non-oscularity assumptions; we are also interested in its stability with respect to perturbations of $\M$.
Later in \Cref{sec:proof_subsample}, we use some of our findings to prove the Stability Theorem for Subsets \Cref{thm:crit_points}.
In \Cref{sec:BSP}, we introduce a condition that we call the Big Simplex Property \ref{BSP} that is equivalent to \ref{P4}, but has the advantage of being  phrased in terms of the local behaviour of $\M$ around the projections of its critical points.
This helps us prove \Cref{thm:Morse} in \Cref{sec:Morse} using the theory of Min-type functions developed in \cite{gershkovich1997morse}, as well as  the $C^2$ Stability Theorem \ref{thm:C2_stab} in \Cref{sec:C2_stability}, which implies in particular the openness of the conjunction of Conditions \Pall.
Finally, the last section is dedicated to proving the density part of the statement of the Genericity Theorem \ref{thm:generic}; it starts with some exposition, where we introduce variants of Thom's transversality theorem and some necessary notions regarding multijets, before moving on to the proof proper.

Throughout the paper, we always assume that the dimension $m$ of the manifold $M$ is greater than or equal to $1$, and that the dimension $D$ of the ambient Euclidean space $\R^D$ is greater than or equal to $2$, as otherwise all statements in the Introduction are trivially true.

\section{Notations, definitions and preliminaries}\label{sec:notations_definitions}

For $s\in \N$, we let $[s]$ be the set $\{1,\dots,s\}$. 
We let $B(x,r)$ be the open ball centered at $x$ of radius $r$ and $\overline B(x,r)$ be its closure; we write $S(x,r)$ for the sphere centered at $x$ of radius $r$. For a set $\S\subset \R^D$, the \textit{reach of $\S$} is defined as
$$\tau(\S) := \sup \{r\in \R:\; \forall x \in \R^D\text{ s.t. }d_\S(x)<r, \;\exists !y\in \S \text{ s.t. } d(x,y) = d_\S(x)\}.$$
Note that if $\S$ has positive reach, then it is necessarily closed. Sets with  positive reach are quite regular in many ways; they were extensively studied by H. Federer in \cite{federer1959curvature}. 

Throughout this article, we consider a fixed compact $C^k$ abstract $m$-dimensional manifold $M$ for some $k\in \{2,3,\dots,\infty\}$ and $m\geq 1$,\footnote{The different theorems stated in the introduction can easily be shown to be true for $m=0$.} together  with a fixed finite atlas $\cA=(\phi_j,U_j)_{j\in J}$. For $1\leq l \leq k$, we consider the set $C^l(M,\R^D)$ of functions of class $C^l$ from $M$ to $\R^D$, and for $1\leq p\leq l$, if $f\in C^l(M,\R^D)$, we let
\[ \|f\|_{C^p}:=\max_{0\leq r\leq p}\max_{j\in J} \sup_{u\in U_j}\op{d^r (f\circ \phi_j)_u}\]
be the $C^p$-norm of $f$, where $d^r g_u$ is the differential of order $r$ of a function $g$ at the point $u$. We can always assume this quantity to be finite, up to restricting each chart to a slightly smaller subset of its domain. This norm defines a topology on  $C^l(M,\R^D)$, called the \textit{Whitney $C^p$-topology}. This topology is independent of the choice of the atlas, and the Whitney $C^{p_2}$-topology is finer than the Whitney $C^{p_1}$-topology on $C^l(M,\R^D)$ for any $p_1\leq p_2\leq l$. If $k=\infty$, the union of all the Whitney $C^p$-topologies on $C^\infty(M,\R^D)$ for $p\geq 1$ defines the Whitney $C^\infty$-topology (i.e.~the set of open sets of the Whitney $C^\infty$-topology is the union over $p\geq 1$ of the sets of open sets of the Whitney $C^p$-topologies).
For $1\leq l \leq k$, we let $\Emb^l(M,\R^D)$ be the subset of $C^l(M,\R^D)$  consisting of all $C^l$ embeddings (immersions that are homeomorphic to their image). It is an open set for the Whitney $C^1$-topology (and hence for all Whitney $C^p$-topologies for $1\leq p\leq l$), as shown in \cite[Prop 5.3]{michor1980manifolds}.

Another, more intrinsic definition (based on jets) of the Whitney $C^p$-topology  is given in \Cref{subsection:introductionJets}; in particular, it extends to the case where the domain of the mappings is another $C^k$ manifold $N$, rather than the Euclidean space $\R^D$. See also \cite{golubitsky1974stable} for a more complete presentation.

If $i\in \Emb^l(M,\R^D)$ for $l\geq 2$, then $\M=i(M)$ is a compact $C^2$ submanifold of $\R^D$. 
For $x\in \M$, we let $T_x \M$ be the tangent space of $\M$ at $x$, and $T_x \M^\bot$ the normal space, that are endowed with the norm induced by the ambient Euclidean space. When there is no possible ambiguity, we let $\pi_x:\R^D \rightarrow T_x \M$ be the orthogonal projection on $T_x \M$ and $\pi_x^\bot : \R^D \rightarrow T_x \M^\bot$ be the orthogonal projection on $T_x \M^\bot$. For $\eta \in T_x\M^\bot$, we let $W_{x,\eta}:T_x \M\to T_x \M$ be the shape operator in the normal direction $\eta$.  This is a self-adjoint operator such that  $\dotp{W_{x,\eta}(v),v}$ is the normal curvature in the normal direction $\eta\in T_x \M^\bot$ along the tangent vector $v\in T_x \M$, see \cite[Chapter 8]{lee2018introduction}.


Finally, we will make frequent use of the following elementary lemma.

\begin{lemma}\label{lem:elementary_metric}
Let $\S$ be any compact set. 
Let $t_0>0$ and let $z_0\in Z(\S)$. Then, there exists  $r_0>0$ such that  all the points in $B(z_0,d_\S(z_0) +r_0)\cap \S$ are at distance less than $t_0$ from the projections  $\pi_\S(z_0)$. 
\end{lemma}

\begin{proof}
   Assume that the statement is false. Then, there exist  sequences $r_n\to 0$ and $x_n\in \S$, such that $\|x_n-z_0\|< d_\S(z_0) + r_n$, but $x_n$ is at distance at least $t_0$ from $\pi_\S(z_0)$. By compactness, we can assume without loss of generality that $x_n\to x\in \S$. But  by continuity, we have $\|x-z\|\leq d_\S(z_0)$, so that $x\in \pi_\S(z_0)$, contradicting the condition that all the points $x_n$ are at distance at least $t_0$ from $\pi_\S(z_0)$. 
\end{proof}


\section{Osculation}\label{sec:osculation}
We fix an embedding $i\in \Emb^2(M,\R^D)$, and consider the compact $C^2$ submanifold $\M=i(M)$.
In this section, we give several equivalent formulations of the condition of non-osculation from Property \ref{P3}.
Some of the material presented here is already well-known; as the proofs are short, we nonetheless include them for the sake of completeness.

If $z_0\in \R^D\backslash \M$ and $x_0\in\pi_\M(z_0)$, then the sphere of radius $d_\M(z_0)$ centered at $z_0$ is tangent to the manifold $\M$ at the point $x_0$. As the inside of this sphere does not intersect $\M$, the curvature at $x_0$ in  the normal direction $z_0-x_0$ has to be smaller than or equal to the one of the sphere along any tangent vector of $T_{x_0}\M$. Two different situations can arise: either the curvature of the manifold is strictly smaller than the one of the sphere (as will be shown later, this is generically the case whenever $z_0$ is a critical point of $\M$), or the curvature of the manifold in the normal direction $\eta:=z_0-x_0$   and in at least one of the directions of $T_{x_0}\M$ is equal to the sphere's, that is there exists $v\in T_x \M$ with $\dotp{W_{x_0,\eta}(v),v}=\|v\|^2$. In that case, we say that the sphere of radius $d_\M(z_0)$ centered at $z_0$ is \textit{osculating} the manifold $\M$ at $x_0$. Both situations are illustrated in Figure \ref{fig:illustration_osculation}.

\begin{figure}
    \centering
    \includegraphics[width=0.6\textwidth]{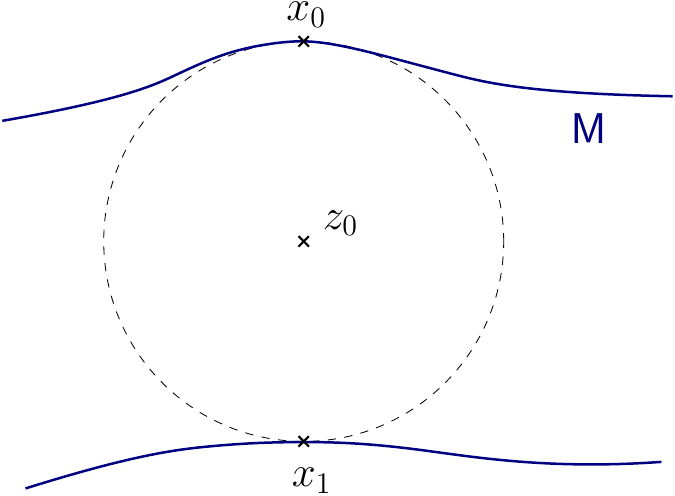}
    \caption{The sphere centered at $z_0$ of radius $d_{\M}(z_0)$ is osculating $\M$ at $x_0$, but not at $x_1$.}
    \label{fig:illustration_osculation}
\end{figure}

In the following lemma, we show that this condition is equivalent to a purely metric one, and that due to the $C^2$ regularity  of $\M$, this metric condition is equivalent to a local version of itself.
Finally, though the non-osculation of $S(z_0, d_\M(z_0))$ at $x_0$ only depends on $\M$, and not on the embedding $i$ such that $\M = i(M)$, it will be convenient later on to have a formulation of the condition in terms of $i$, which is also provided by the lemma.

\begin{lemma}[Osculation Characterization Lemma]\label{lem:characterization_osculation}
    Let $z_0\in \R^D\backslash \M$, $x_0\in \pi_\M(z_0)$ and define $\eta:= z_0-x_0$. Then the sphere of radius $d_\M(z_0)$ centered at $z_0$ not osculating the submanifold $\M$ at $x_0$ is equivalent to the existence of $\alpha>0$ such that one of the following equivalent conditions holds:
    \begin{enumerate}[start=1, label={(\roman*)}]
        \item \label{it:weing} The largest eigenvalue $\lambda_{\max}(W_{x_0,\eta})$ of the shape operator $W_{x_0,\eta}:T_{x_0}\M\to T_{x_0}\M$ at $x_0$ in the direction $\eta$ satisfies $\lambda_{\max}(W_{x_0,\eta})< 1-\alpha$. 
        \item  \label{it:osc_chart} Let $m_0\in M$ be such that $i(m_0)=x_0$, with $(\phi,U)$ a chart of $M$ around $m_0$ with $\phi(u_0)=m_0$. 
        Then, the quadratic form
       \begin{equation*}
         \begin{split}
         \R^m  &\longrightarrow  \R\\
    v&\longmapsto (1-\alpha)\|d(i\circ \phi)_{u_0}(v)\|^2 - \dotp{\eta,d^2( i\circ \phi)_{u_0}(v,v)}
        \end{split}
       \end{equation*}
        is  positive definite.   
        \item \label{it:pointwise_osc} There exists $\alpha'>\alpha$ and $\delta>0$ such that, if $\M_0:=\M\cap \overline B(x_0,\delta)$, then for any $y\in \M_0$,
        \begin{equation}
            \|y-z_0\|^2 \geq \|x_0-z_0\|^2+\alpha'\|y-x_0\|^2 .
        \end{equation} 
        \item \label{it:local_osc} There exists $\alpha'>\alpha$ and $\delta,\delta'>0$ such that, if $\M_0:=\M\cap \overline B(x_0,\delta)$, then for any $z\in B(z_0,\delta')$, $x\in \pi_{\M_0}(z)$ and $y\in \M_0$, it holds that
        \begin{equation}\label{eq:P3_loc}
             \|y-z\|^2 \geq \|x-z\|^2+\alpha'\|y-x\|^2 .
        \end{equation}
    \end{enumerate}
\end{lemma}
When the sphere of radius $d_\M(z_0)$ centered at $z_0$ does not osculate $\M$ at $x_0$, we call a number $\alpha>0$ that satisfies the conditions of Lemma \ref{lem:characterization_osculation} \textit{a degree of non-osculation at $x_0$ of the sphere $S(z_0,d_\M(z_0))$}.

\begin{proof}
    The equivalence between the sphere being non-osculating and the existence of some $\alpha  >0$ that satisfies condition \ref{it:weing} is clear. The equivalence between  \ref{it:weing} and \ref{it:osc_chart} follows from the fact that, if $w=d(i\circ \phi)_{u_0}(v)$, then 
    \begin{equation}\label{eq:weingarten_chart}
            \dotp{W_{x_0,\eta}(w),w} = \dotp{\eta, d^2(i\circ \phi)_{u_0}(v,v)},
    \end{equation}
    as $\eta\in T_{x_0}\M^\top$ (see \cite[Chapter 8]{lee2018introduction}).
    
    Let us now show the equivalence between \ref{it:osc_chart} and \ref{it:pointwise_osc}. 
    Observe that for any $\alpha'>0$ and any $y\in\R^D$,
   $$\| y -z_0\|^2 - \alpha' \|y-x_0\|^2 - \|\eta\|^2 = (1-\alpha')\|y-x_0\|^2 - 2\dotp{\eta,y-x_0}.$$
    Let $i,\phi$ and $u_0$ be as in the statement of condition \ref{it:osc_chart}; then condition \ref{it:pointwise_osc} is equivalent to the existence of some $\alpha'>\alpha$ such that
    $$(1-\alpha')\|i\circ \phi(u)-x_0\|^2 - 2\dotp{\eta,i\circ \phi(u)-x_0}>0$$
    for any $u\neq u_0$ close enough to $u_0$.
    But, as $i\circ \phi$ is $C^2$ and $\mathrm{Im}(d(i\circ \phi)_{u_0})=T_{i\circ\phi(u_0)}\M$ is orthogonal to $\eta$, it holds that
    \begin{align*}
           &(1-\alpha')\|i\circ \phi(u)-x_0\|^2 - 2\dotp{\eta,i\circ \phi(u)-x_0}  \\
   &\qquad =(1-\alpha')\|d(i\circ \phi)_{u_0}(u-u_0)\|^2 - \dotp{\eta,d^2(i\circ \phi)_{u_0}(u-u_0,u-u_0)} + o(\|u-u_0\|^2).
    \end{align*}
    Hence, the existence of some $\alpha'>\alpha$ such that $(1-\alpha')\|d(i\circ \phi)_{u_0}(v)\|^2 - \dotp{\eta,d^2(i\circ \phi)_{u_0}(v,v)} + o(\|v\|^2)$ is strictly positive for $v\neq 0$ small enough is equivalent to the form
    $$ v\longmapsto (1-\alpha)\|d(i\circ \phi)_{u_0}(v)\|^2 - \dotp{\eta,d^2(i\circ \phi)_{u_0}(v,v)}$$
    being  positive definite (due to $d(i\circ \phi)_{u_0}$ being an isomorphism). This shows that conditions \ref{it:osc_chart} and \ref{it:pointwise_osc} are equivalent.

Condition \ref{it:local_osc} obviously implies condition \ref{it:pointwise_osc}.
It remains to show that \ref{it:pointwise_osc} implies \ref{it:local_osc}. We postpone the proof to  Section \ref{sec:projections}, in Corollary \ref{cor:non_osculation_pointwise_implies_local}.
\end{proof}

In the next section, we show, among other things, that the non-oscularity with respect to $\M$ of the sphere $S(z_0,d_\M(z_0))$ at $x_0\in \pi_\M(z_0)$ gives nice regularity and stability properties to the projection map around $z_0$ on a neighborhood of $x_0$.


\section{Regularity and stability of projections}\label{sec:projections}

Let as earlier $i\in \Emb(M,\R^D)$ be a $C^2$ embedding of the abstract manifold $M$, and $\M:=i(M)$ be the resulting submanifold.
In this section, we prove various statements regarding the  projections on $\M$ or on some subsets of $\M$, particularly under some assumptions of non-osculation. These results will be instrumental later in the article.

\subsection{Lipschitz continuity and local unicity of the projection}

The following lemma  states that if the sphere $S(z_0,d_\M(z_0))$ is not osculating $\M$ at $x_0\in \pi_\M(z_0)$, then one can define a Lipschitz projection function from a neighborhood of $z_0$ to a neighborhood of $x_0$ in $\M$.

\begin{lemma}\label{lem:lipschitz_projection}
 Let $z_0\in \R^D\backslash \M$ and $x_0\in \pi_\M(z_0)$. Assume that the sphere $S(z_0,d_\M(z_0))$ does not osculate $\M$ at $x_0$, with $\alpha>0$ a degree of non-osculation. Let $\delta,\delta'>0$ be the constants given in \Cref{lem:characterization_osculation}.\ref{it:local_osc}. If $\M_0=\M\cap \overline B(x_0,\delta)$ and $z\in B(z_0,\delta')$, then $\pi_{\M_0}(z)$ contains exactly one element $x$, which satisfies 
    \begin{equation}
    \|x-x_0\|\leq \frac{1}{\alpha}\|z-z_0\|.
\end{equation}
\end{lemma}

\begin{proof}
Let $\delta,\delta'>0$ and $\alpha'>\alpha$ be the constants given in \Cref{lem:characterization_osculation}.\ref{it:local_osc}.
Let $x\in \pi_{\M_0}(z)$. \Cref{lem:characterization_osculation}.\ref{it:local_osc} implies that every point $y\in \M_0$ different from $x$ satisfies $\|y-z\|>\|x-z\|$. This proves that  $\pi_{\M_0}(z)$ contains exactly one element $x$. Let us prove the second claim. Assume that $x\neq x_0$, for otherwise the claim is trivial. According to \eqref{eq:P3_loc} with $y=x_0$,
    \begin{align*}
    &\|x_0-z\|^2 \geq  \|x-z\|^2 + \alpha'\|x_0-x\|^2\\
    &\geq   \|x-x_0\|^2 + \|x_0-z\|^2 +  2\dotp{x-x_0,x_0-z}+ \alpha'\|x_0-x\|^2 \\
    &\geq   \|x-x_0\|^2 + \|x_0-z\|^2 + 2\dotp{x-x_0,x_0-z_0} + 2\dotp{x-x_0,z_0-z}
   + \alpha'\|x_0-x\|^2,
\end{align*}
that is, using Cauchy-Schwartz inequality
\begin{equation}\label{eq:prop_osc_step1}
    (1+\alpha')\|x-x_0\|^2 \leq -2\dotp{x-x_0,x_0-z_0}+ 2\|x-x_0\| \|z_0-z\|.
\end{equation}
Also, we may use \eqref{eq:P3_loc} with $y=x$ and $x=x_0$ to obtain
\begin{align*}
     \alpha'\|x_0-x\|^2 + \|x_0-z_0\|^2 &\leq \|x-z_0\|^2\\
    \alpha'\|x_0-x\|^2 +  \|x_0-z_0\|^2&\leq  \|x-x_0\|^2 +\|x_0-z_0\|^2 + 2\dotp{x-x_0,x_0-z_0}.
\end{align*}
that is
\begin{equation}\label{eq:prop_osc_step2}
  -2\dotp{x-x_0,x_0-z_0}< (1-\alpha')\|x_0-x\|^2.
\end{equation}
Putting together \eqref{eq:prop_osc_step1} and \eqref{eq:prop_osc_step2} yields
\[ 2\alpha'\|x-x_0\|^2 \leq 2\|x-x_0\| \|z_0-z\|.\]
Dividing by $2\alpha'\|x-x_0\|$ and using that $\alpha'>\alpha$ give the result.
\end{proof}

\subsection{Differential of the projection}

The following proposition, which can be found e.g. in \cite{all_about_projection}, describes the differential of the projection on $\M$ wherever it is well-defined.
\begin{proposition}\label{prop:diff_projection}
Let $\overline \Med(\M)$ be the closure of the medial axis of $\M$, and let $z\in \R^D \backslash (\M \cup \overline \Med(\M) )$. Then the projection on $\M$ uniquely defines a $C^1$ map $p$ on a neighborhood of $z$, and the differential of $p$ in $z$ is given by 
\[ dp_z = \left(\mathrm{id}_{T_{p(z)}\M} - W_{p(z),z-p(z)}\right)^{-1}\circ \pi_{p(z)}:\R^D \rightarrow T_{p(z)}\M,\]
where we recall that $\pi_{p(z)}:\R^D \rightarrow T_{p(z)}\M$ is the orthogonal projection onto $T_{p(z)}\M$.
\end{proposition}

The following lemma serves two purposes: it shows that under simple conditions, the projection at $z_0$ on the image of the germ of a function $f:M\rightarrow \R^D$ (rather than on the image of some embedding $i:M\rightarrow \R^D$) and the differential of that projection are well-defined, and in that context it expresses the  formula for the differential of the projection given in Proposition \ref{prop:diff_projection} in terms of coordinates of charts of $M$.\footnote{Such an expression must be available somewhere in the literature, but we could not find it.} This will be needed when working with jets in Section \ref{sec:density}; sadly, it involves some rather tedious verifications.
\begin{lemma}\label{lem:diff_projection_in_charts}
Let $m_0\in M$ and $x_0,z_0\in \R^D$ and let $f\in C^2(M,\R^D)$ be some function (not necessarily an embedding) that maps $m_0$ to $x_0$.
Furthermore, suppose that
$df_{m_0}$ has full rank and that $\Im(df_{m_0}) \perp z_0-x_0$. Let $\phi : (V,u_0) \rightarrow (M, m_0)$ be any chart of $M$ around $m_0$. 
If the quadratic form
       \begin{equation*}
         \begin{split}
         \R^m  &\longrightarrow  \R\\
    v&\longmapsto \|d(f\circ \phi)_{u_0}(v)\|^2 - \dotp{z_0-x_0,d^2(f\circ \phi)_{u_0}(v,v)}
        \end{split}
       \end{equation*}
        is  positive definite,
then there exist $\delta>0$ and a neighborhood $U\subset M$ of $m_0$ such that
\begin{itemize}
    \item $f|_U : U \rightarrow \R^D$ is an embedding, 
    \item each $z\in B(z_0,\delta)$ admits a unique projection on $\overline{f(U)}$ and that unique projection belongs to $f(U)$, and
    \item this unique projection defines a $C^1$ map $p:B(z_0,\delta) \to f(U)$.
\end{itemize}

In that case, let as above  $\phi : (V,{u_0}) \rightarrow (M, m_0)$ be any chart of $M$ around $m_0$. For any $v\in \Im(df_{m_0})$, let $\Lambda(v) \in \Im(df_{m_0})$ be the only vector in $\Im(df_{m_0})$ such that 
$$\dotp{z_0-x_0, d^2 (f\circ \phi)_{u_0}(d(f\circ \phi)_{u_0}^{-1}(v),d(f\circ \phi)_{u_0}^{-1}(w)) } = \dotp{\Lambda(v),w} $$
for any $w\in \Im(df_{m_0})$.
This defines a linear endomorphism  $\Lambda:  \Im(df_{m_0}) \rightarrow  \Im(df_{m_0})$ that is independent from the choice of the chart $\phi$, and
\begin{equation*}
    \begin{split}
Q: \Im(df_{m_0})& \longrightarrow  \Im(df_{m_0})\\
v &\longmapsto v - \Lambda(v)
    \end{split}
\end{equation*}
is a linear isomorphism.

The differential at $z_0$ of the map $p$ is then equal to
\[ d p_{z_0} = Q^{-1}\circ  \pi_{ \Im(df_{m_0})} : \R^D \rightarrow  \Im(df_{m_0}), \]
where $ \pi_{\Im(df_{m_0})} : \R^D \rightarrow \Im(df_{m_0})$ is the orthogonal projection.
\end{lemma}

\begin{remark}
 As shown in the proof below, the linear map $\Lambda$ is simply the shape operator at $x_0$ in normal direction $z_0-x_0$ of the submanifold $f(U)$.  
\end{remark}

\begin{proof}
 The fact that $f|_U:U\rightarrow \R^D$ is an embedding for a small enough neighborhood $U$ of $m_0$ is simply a consequence of the full rank of $df_{m_0}$. Let us choose such a neighborhood.

 If the form  $ v\mapsto \|d(f\circ \phi)_{u_0}(v)\|^2 - \dotp{z_0-x_0,d^2(f\circ \phi)_{u_0}(v,v)}$ is positive definite for a given chart $\phi : (V,{u_0}) \rightarrow (M, m_0)$ of $M$ around $m_0$ (we can assume that $\phi(V)\subset U$), it is so for any such chart (as $\Im(df_{m_0}) \perp z_0-x_0$). In that case, consider the function
 $g:V \rightarrow \R, u \mapsto \|f\circ\phi(u) - z_0\|^2$. Its first differential at ${u_0}$ is $0$, as $\Im(df_{m_0}) \perp z_0-x_0$, and its second differential is 
 $$(v_1,v_2)\mapsto 2(\dotp{d(f\circ \phi)_{u_0}(v_1),d(f\circ \phi)_{u_0}(v_2)}- \dotp{z_0-x_0,d^2(f\circ \phi)_{u_0}(v_1,v_2)}).$$
 By hypothesis, the associated quadratic form is positive definite, hence ${u_0}$ is a strict local minimum of $g$, and $x_0$ is a strict local minimum of the distance function to $z_0$ restricted to $f(U)$. By making the neighborhood $U$ of $m_0$ in $M$ smaller, we can  thus assume that $x_0$ is the unique projection of $z_0$ on $\overline{f(U)}$. 

Furthermore, we have shown that the sphere $S(z_0,d_{\overline{f(U)}}(z_0))$ does not osculate $\overline{f(U)}$ at $x_0$ 
(our definition of osculation naturally extends to this case), using characterization \ref{it:osc_chart} from the Osculation Lemma \ref{lem:characterization_osculation}.
 If $\overline{f(U)}$ were a submanifold (without boundary), we  could now use the Lipschitz continuity of the projection from \Cref{lem:lipschitz_projection} to conclude that there exists $\delta>0$ such that each $z\in B(z_0,\delta)$ admits a unique projection $p(z)$ on $\overline{f(U)}$ that belongs to $f(U)$ (possibly by making $U$ yet smaller). This would also mean that $B(z_0,\delta)$ would be included in the interior of the complement of the medial axis of $\overline{f(U)}$, hence the projection map $p:B(z_0,\delta) \rightarrow f(U)$ would be $C^1$ with its differential given by \Cref{prop:diff_projection}.
 One can see that the argument still works in the same way despite the boundary.
 Let $\delta$ (and $U$) satisfy the conditions above.

The only part of the proof left is to adapt the formula for the differential of the projection from Proposition \ref{prop:diff_projection} to express $d p_{z_0}$ in the coordinates of $f\circ \phi|_V$.
Let us show that $\Lambda$ is indeed the shape operator at $x_0$ in normal direction $z_0-x_0$ of $f(U)$, as claimed in the remark above.
In  \cite[p. 126]{do1992riemannian}, a bilinear map $B: \Im(df_{m_0})\times \Im(df_{m_0}) \rightarrow \Im(df_{m_0})^\perp$ is defined as follows: for $w_2,w_2\in \Im(df_{m_0}) = T_{x_0}f(U)$, let $W_1, W_2 : f(U)\supset \tilde{U} \rightarrow Tf(U)$ be vector fields on the submanifold $f(U)$ defined around $x_0$ such that $W_1(x_0) = w_1$ and $W_2(x_0)=w_2$, and let $\overline{W_1}$ and $\overline{W_2}$ be local extensions of $W_1$ and $W_2$ to $\R^D$.
Then
$$B(w_1,w_2) = \pi_{T_{x_0}f(U)^\perp}(d(\overline{W_1})_{x_0}(\overline{W_2}(x_0))), $$
where $ \pi_{T_{x_0}f(U)^\perp}$ is the orthogonal projection on $T_{x_0}f(U)^\perp$. The same reference shows that $B(w_1,w_2)$ does not depend on the choice of $W_1,W_2,\overline{W_1},\overline{W_2}$ and is symmetric.

Let us express $B$ in terms of the embedding $f|_U$. We can assume without loss of generality that $ T_{x_0}f(U) = \R^m\times \{0\} \subset \R^D$ and $x_0 = 0$. Then the projection $\pi: \R^m \times \R^{D-m} \rightarrow \R^m$ restricted to $f(U)$ is a diffeomorphism around $0$, hence $\pi\circ f : U\rightarrow \R^m $  is locally the inverse of a chart $\psi : (\tilde{V},0) \rightarrow (M, m_0)$  of $M$ around $m_0$. Around $0\in \R^D$, the submanifold $f(U)$ is parameterized as the graph of the map $f\circ \psi : \R^m \rightarrow \R^D$.

Now let $w_1,w_2\in T_{x_0}f(U)$. We define a vector field $W_i$ for $i=1,2$ as follows:
\begin{equation*}
    \begin{split}
\overline{W_i}: \R^D = \R^m \times \R^{D-m}& \longrightarrow    \R^D\\
(x,y) &\longmapsto d(f\circ \psi)_x( (d(f\circ \psi)_0)^{-1}(w_i)).
    \end{split}
\end{equation*}
Then the vector fields $\overline{W_i}$ are such that their restrictions to $f(U)$ take values in the tangent bundle of $f(U)$, and such that $\overline{W_i}(0)=w_i$. By definition,
\begin{align*}
B(w_1,w_2) = \pi_{T_{0}f(U)^\perp}(d(\overline{W_1})_{0}(\overline{W_2}(0))) =  \pi_{T_{0}f(U)^\perp}( d^2(f\circ \psi)_0((d(f\circ \psi)_0)^{-1}(w_1),w_2)) \\
= \pi_{T_{0}f(U)^\perp}( d^2(f\circ \psi)_0((d(f\circ \psi)_0)^{-1}(w_1),(d(f\circ \psi)_0)^{-1}(w_2))) 
\end{align*}
where the last equality is true because $d(f\circ \psi)_0$ is the inclusion $\R^m \rightarrow \R^m \subset \R^D$.
In Definition 2.2 of \cite[Chapter 6]{do1992riemannian}, the shape operator at $x_0 = 0$ in normal direction $z_0-x_0$ of the submanifold $f(U)$ is defined as the linear operator $W_{0,z_0-x_0} : T_0f(U) \rightarrow T_0f(U)$ such that 
$$\dotp{W_{0,z_0-x_0}(w_1),w_2} = \dotp{ z_0-x_0,B(w_1,w_2)}$$
for any $w_1,w_2 \in T_0f(U)$, i.e. such that 
$$\dotp{W_{0,z_0-x_0}(w_1),w_2} = \dotp{ z_0-x_0, d^2(f\circ \psi)_0((d(f\circ \psi)_0)^{-1}(w_1),(d(f\circ \psi)_0)^{-1}(w_2))}.$$
This is precisely the definition that we gave of the operator $\Lambda$, hence $\Lambda = W_{0,z_0-x_0}$ is the shape operator in normal direction $z_0-x_0$ at $x_0$. The fact that $\Lambda$ does not depend on the choice of the chart $\psi$ can then be deduced  either from the fact that the shape operator does not depend on the choice of the vector fields $\overline{W_i}$ (as shown in \cite{do1992riemannian}), or by direct computation using the expression $\dotp{ z_0-x_0, d^2(f\circ \psi)_0((d(f\circ \psi)_0)^{-1}(w_1),(d(f\circ \psi)_0)^{-1}(w_2))}$.

Proposition \ref{prop:diff_projection} then states the  invertibility of $Q = \id - \Lambda$ and the fact that $d p_{z_0} = Q^{-1}\circ  \pi_{ \Im(df_{m_0})} $, which concludes the proof.
\end{proof}

\subsection{\texorpdfstring{$C^2$}{C2} stability of the projection}

In this subsection, we show 
that a non-osculatory projection on $i(M)$ behaves nicely locally, and that this is stable with respect to small $C^2$ perturbations of the embedding $i$. In particular, we finish the proof of \Cref{lem:characterization_osculation}.
We first prove an almost trivial lemma. 

\begin{lemma}\label{lem:abstract_minimum_unicity}
 Let $V\subset \R^m$ be an open set, let $\Lambda$ be a topological space and let $F:V\times \Lambda \rightarrow  \R$ be a continuous function. For $\lambda \in \Lambda$, write $F_\lambda : v \mapsto F(v, \lambda)$ and assume that the second differential \[ V\times \Lambda \rightarrow \mathrm{Bil}(\R^m,\R^m,\R),\ (v,\lambda)\mapsto d^2(F_\lambda)_v\]
 of $F$ with respect to $v\in V$ exists and is continuous with respect to $(v,\lambda) \in V\times \Lambda$. Assume further that the map $\Lambda \rightarrow C^0(V,\R), \lambda \mapsto F_\lambda$ is continuous with respect to the $\|\cdot\|_\infty$ norm, that $d^2(F_{\lambda_0})_{v_0}$ is positive definite for some $v_0\in V$, $\lambda_0\in \Lambda$ and that  $v_0$ is a global minimum for $F_{\lambda_0}$. 
 Then there exist open neighborhoods $ U_{\R^m} \subset V$ of $v_0$ and  $U_\Lambda \subset \Lambda$ of $\lambda_0$ such that $\overline{U_{\R^m}} \subset V$, that for any $\lambda \in U_\Lambda$ the function  $F_\lambda$ admits exactly one minimum $v_\lambda$ on  $ \overline{U_{\R^m}}$ with $v_\lambda \in U_{\R^m}$ and  that $d^2(F_\lambda)_v$ is positive definite for any $v\in  U_{\R^m}, \lambda \in U_\Lambda$.
\end{lemma}

\begin{proof}
  By assumption the second differential $d (v,\lambda)\mapsto d^2(F_\lambda)_v$ is continuous in $(v,\lambda)\in V\times \Lambda$, and being definite positive is an open property, hence there exists an open neighborhood of $(v_0,\lambda_0)$ on which $d^2(F_\lambda)_v$ is positive definite. By restricting it, we can assume that it is of the shape $  U_{\R^m}\times U_\Lambda\subset V\times \Lambda$ and that $ U_{\R^m}$ is a convex subset of $\R^m$.  For any fixed $\lambda \in U_\Lambda$, the second differential $d^2(F_\lambda)_v$ in $v$ of the function $F_\lambda :  U_{\R^m} \rightarrow \R^m$ is positive definite, hence $F_\lambda$ is strictly convex and admits at most one minimum $v_\lambda$ on the convex set $ U_{\R^m}$.

  As $v_0$ is a minimum for $F_{\lambda_0}$ and $d^2(F_{\lambda_0})_{v_0}$ is positive definite, there exist open neighborhoods $U_1,U_2$ of $v_0$ and $\beta>0$ such that $U_1 \subset U_2  \subset \overline{U_2} \subset  U_{\R^m} $, that $U_2$ is convex and that $F_{\lambda_0}(v_0)< \min \{F_{\lambda_0}(x) :\ x\in  \overline{U_2}\backslash U_1\} - \beta$. By choosing $U_\Lambda$ small enough that $\|F_\lambda - F_{\lambda_0}\|_\infty < \beta/4$ for any $\lambda \in U_\Lambda$, we have that $F_{\lambda}(v_0)< \min \{F_{\lambda}(x) :\ x\in \overline{U_2}\backslash U_1\} - \beta/2$ for such a $\lambda \in U_\Lambda$. Thus $F_\lambda$ admits a minimum $v_\lambda$ on $\overline{U_2}$ and $v_\lambda$ must belong to $U_1 \subset U_2$. As $U_2$ is convex, the argument used above applies again and this minimum must be unique. By redefining $U_{\R^m}$ as its subset $U_2$, we conclude that $F_\lambda$ does admit exactly one minimum $v_\lambda$ on  $ \overline{U_{\R^m}}$ with $v_\lambda \in U_{\R^m}$.
\end{proof}

Now we specialize the lemma above to the special case of interest to us.

\begin{proposition}[$C^2$ stability  of the projection]\label{prop:C2_stability_projection}
    Let $i_0\in \Emb^2(M,\R^D)$ and $z_0\in \R^D$. Let $x_0= i_0(m_0)$ be a projection of $z_0$ on $i_0(M)$ such that $S(z_0, d_{i_0(M)}(z_0))$ is non-osculating $i_0(M)$ at $x_0$. Then there exist $\alpha>0$ and neighborhoods $m_0\in U_M\subset M$, $z_0 \in U_{\R^D} \subset \R^D$ and $i_0\in U_{\Emb}\subset \Emb^2(M,\R^D)$ such that for any $i\in  U_{\Emb}$, any $z\in U_{\R^D}$ has a unique projection $p_i(z)$ on $\overline{i (U_M)}$, that $p_i(z)$ belongs to $i (U_M)$, and that the sphere $S(z,d_{i  (U_M)}(z))$ does not osculate $i(U_M)$ at $p_i(z)$ with $\alpha$ a degree of non-osculation. 
    
    In particular, there exist $\delta,\delta'>0$ such that for any $i\in U_{\Emb}$ and $z\in U_{\R^d}$, if $x=p_i(z)$ and $\M=i(M)$, then the function $p_i:B(z,\delta')\to \M \cap \overline B(x,\delta)=\M_0$ is the orthogonal projection on $\M_0$, which is $C^1$ and $\alpha^{-1}$-Lipschitz continuous.
\end{proposition}

\begin{proof}
Let $\phi :V \rightarrow M$ be a chart of $M$ with $\phi(u_0) = m_0$, and consider the function
 \[ F:V \times \R^D \times  \Emb^2(M,\R^D),\ (u,z,i) \mapsto \|i\circ\phi(u)-z\|^2- \alpha\|i\circ\phi(u)-i(m_0)\|^2.\]
 Its second differential with respect to $u \in V$ is the map $$ V \times \R^D \times \Emb^2(M,\R^D)  \rightarrow \text{Bil}(\R^m,\R^m,\R) $$
 that sends  $(u,z,i)$ to the bilinear map
 \[ d^2(F_{(z,i)})_u : (v_1,v_2)\mapsto 2(\dotp{d(i\circ \phi)_u(v_1), d(i\circ \phi)_u(v_2)} + \dotp{i\circ \phi(u)-z, d^2(i\circ \phi)_u(v_1,v_2)}).\]
Both $F $ and $d^2(F_{(z,i)})_u$ are continuous with respect to  $(u,z,i)$ (with $\Emb^2(M,\R^D) $ equipped with the Whitney $C^2$-topology).
Using Characterization \ref{it:osc_chart} from the Osculation Lemma \ref{lem:characterization_osculation}, we see that if $i\circ \phi(u)\in i(M)$ is a projection of $z$ on $i(M)$, the sphere $S(z,d_{i(M)}(z))$ not osculating $i(M)$ at $i\circ \phi(u)$ 
is equivalent to $d^2(F_{(z,i)})_u$ being positive definite.

Hence we can apply Lemma \ref{lem:abstract_minimum_unicity} to $F$ and $(u_0,z_0,i_0)$ (with $\Lambda = \R^D \times\Emb^2(M,\R^D)  $) to get that there exist neighborhoods $u_0\in U_V\subset V$, $z_0 \in U_{\R^D} \subset \R^D$ and $i_0\in U_{\Emb}\subset \Emb^2(M,\R^D)$ such that for any $(z,i)\in  U_{\R^D}\times U_{\Emb}$, the function $F_{(z,i)}:u\rightarrow \|i\circ\phi(u)-z\|^2$ admits exactly one minimum on $\overline{U_V}$ and that this minimum belongs to $U_V$. 
Define $U_M := \phi(U_V)$. For a given $i\in  U_{\Emb}$, this property is equivalent to the projection $ U_{\R^D}$ on $\overline{i(U_M)} $ being uniquely defined and its image being contained in $i(U_M)$. Let us call $p_i$ this projection.
Lemma \ref{lem:abstract_minimum_unicity} also guarantees that if  $i\in  U_{\Emb}$, $z\in U_{\R^D}$ and if we write  $x:=p_i(z)$ and $u := (i\circ\phi)^{-1}(x) \in U_V$, then 
 \[d^2(F_{(z,i)})_u : (v_1,v_2)\mapsto 2(\dotp{d(i\circ \phi)_u(v_1), d(i\circ \phi)_u(v_2)} + \dotp{i\circ \phi(u)-z, d^2(i\circ \phi)_u(v_1,v_2)})\]
is positive definite. By possibly further restricting $ U_{\Emb}, U_{\R^D}$ and $ U_V$, we can assume that there exists $c>0$ such that the smallest eigenvalue of  $d^2(F_{(z,i)})_u$ is greater than $c$ for any such $i,z,u$ and that the operator norm of $d(i\circ \phi)$ is bounded on $U_\Emb$. Hence there exists $\alpha>0$ such that the form $v \mapsto (1-\alpha)\|d(i\circ \phi)_{u}(v)\|^2 - \dotp{i\circ \phi(u)-z,d^2( i\circ \phi)_{u}(v,v)}$ is definite positive for any such $i,z,u$.
Thus Characterization \ref{it:osc_chart} from the Osculation Lemma \ref{lem:characterization_osculation} guarantees that the sphere $S(z,d_{i  (U_M)}(z))$ does not osculate $i(U_M)$ at $p_i(z)$, with $\alpha$ a degree of non-osculation. 
This proves the first part of the proposition.

For each $i\in U_{\Emb}$, the map $p_i$ is the orthogonal projection from $U_{\R^D}$ to $\overline{i(U_M)}$. The non-osculation of each sphere $S(z,d_{i  (U_M)}(z))$ at   $p_i(z)$ for all $z\in U_{\R^D}$ entails that $p_i$ is $\alpha^{-1}$-Lipschitz continuous, as the proof of \Cref{lem:lipschitz_projection} shows. Besides, $p_i$  is $C^1$ as the orthogonal projection on the manifold $i(U_M)$. Consider $\delta'>0$ and $U\subset U_{\R^D}$ such that $B(z,\delta')\subset U_{\R^D}$ for all $z\in U$. Let $\delta=\delta'/\alpha$. Then, for all $z \in U$, the function $p_i$ is defined on $B(z,\delta')$ and takes its values in $B(p_i(z),\delta)\cap i(U_M)\subset \overline B(p_i(z),\delta)\cap i(M)$.  We replace $U_{\R^D}$ by $U$ to conclude.
\end{proof}

As a corollary, we get the final part of the proof of the Osculation Characterization Lemma.
\begin{cor}\label{cor:non_osculation_pointwise_implies_local}
Characterization \ref{it:pointwise_osc} from  Lemma \ref{lem:characterization_osculation} implies Characterization \ref{it:local_osc}.
\end{cor}
\begin{proof}
Let $z_0, x_0$ and the embedding $i$ from the statement of Lemma \ref{lem:characterization_osculation} and assume that they satisfy  Characterization \ref{it:pointwise_osc}.
We already know that Characterization \ref{it:pointwise_osc} is equivalent to Characterization \ref{it:osc_chart}, which is the one used in the first part of Proposition \ref{prop:C2_stability_projection} above (hence there is no circularity in our argument). Thus the proposition can be applied to $z_0, x_0$ and  $i$ to get the existence of a neighborhood $U_{\R^D}$ of $z_0$ and a neighborhood $\M_0$ of $x_0$ in $i(M)$ such that any $z\in U_{\R^D}$ and $x$ its unique local projection on $\M_0$ satisfy \ref{it:pointwise_osc}. This, in turn, proves that $z_0, x_0$ and  $i$ verify Condition \ref{it:local_osc}.
\end{proof}

\subsection{Multiple local projections}
This subsection is dedicated to an easy but handy lemma that helps us manage multiple local projections on a neighborhood of a critical point.

 Assume that the $C^2$ submanifold $\M$ satisfies \ref{P1}, \ref{P2} and \ref{P3}, and let $z_0\in Z(\M)$ be a critical point with projections $ \pi_\M(z_0)=\{x_1,\ldots,x_s\}$, for some $s\in \{2,\ldots, D+1\}$. As condition \ref{P3} is satisfied, the sphere $S(z_0,d_\M(z_0))$ does not osculate $\M$ at any of the $x_i$.
Using Lemma \ref{lem:lipschitz_projection}, we can pick radii $\delta$ and $\delta'$ so that  there exist local projections $p_j:B(z_0,\delta')\to \M_j= \overline B(x_j,\delta)\cap \M$ for all $j\in [s]$. Moreover, by making $\delta'$ smaller, we can ensure using the Lipschitz continuity of the projections that the image of $p_i$ is included in the relative interior of $\M_i$. Hence, these functions are equal to the projection on a $C^2$ manifold, and they are in particular of regularity $C^1$, with known expressions for their gradients, as stated in Proposition \ref{prop:diff_projection}. By \ref{P2}, the number of critical points is finite, so that the constants $\delta$ and $\delta'$ can be chosen independently of $z_0\in Z(\M)$. We summarize these properties in the next lemma.
\begin{lemma}\label{lem:local_projections_well_defined}
Let $\M\subset\R^D$ be a $C^2$ compact submanifold satisfying \ref{P1}, \ref{P2} and \ref{P3}. 
There exist $\delta', \delta >0$ such that for any $z_0\in Z(\M)$ with $\pi_\M(z_0)=\{x_1,\ldots,x_s\} $, there are $C^1$ functions $p_i: B(z_0,\delta') \rightarrow \M_i =\overline B(x_i,\delta)\cap \M$ for $i=1\dots,s$ such that  for any $z\in B(z_0,\delta')$, we have $\pi_{\M_i}(z)=\{p_i(z)\} \subset B(x_i,\delta)\cap \M$ and $\pi_\M(z) \subset \{p_1(z),\ldots,p_s(z)\}$. Furthermore, if $p_i(z)=x$ and $\eta=z-x$, then $(dp_i)_z = (\mathrm{id} - W_{x,\eta})^{-1} \circ \pi_x$.
\end{lemma}

\begin{proof}
The only point that remains to be proven is  the  inclusion $\pi_\M(z) \subset \{p_1(z),\ldots,p_s(z)\}$. Let $z\in B(z_0,\delta') $ and  $y\in \pi_\M(z)$. Then,
\[ \|y-z_0\|\leq \|y-z\|+\|z-z_0\| \leq \|x_0-z\| +\|z-z_0\|  \leq \|x_0-z_0\|+2\|z-z_0\|.\]
Hence,  $y\in B(z_0,d_{\M}(z_0)+2\delta')$. By \Cref{lem:elementary_metric}, if $\delta'$ is small enough, this implies that $y$ is in one of the sets $M_1,\dots,M_s$, implying that $y\in \pi_{\M_i}(z)$ for some $i=1,\dots,s$.
\end{proof}

\section{Sampling theory for generic manifolds}\label{sec:proof_subsample}

In this section, we prove the Stability Theorem for Subsets (\Cref{thm:crit_points}), which can be seen as a stronger analogue to the results for compact sets from \cite{Chazal_CS_Lieutier_OGarticle} in the special case of well-behaved manifolds sampled without noise.
We restate the  theorem for the reader's convenience.

\subsetthm*

\begin{proof}[Proof of \Cref{thm:crit_points}]
    Let $\A \subset \M$ be a subset with $d_H(\A,\M)\leq \eps$, and let $\eps < \tau(\M)/8$.
\begin{enumerate}[wide, labelwidth=!, labelindent=0pt]
    \item According to the critical values separation theorem \cite[Theorem 4.4]{Chazal_CS_Lieutier_OGarticle}, the distance function $d_\A$ has no critical values in the interval $(4\eps,\tau(\M)-3\eps)$. 
    This proves that  all critical points in $Z(\A)$ are either at distance less than $4\eps$ from $\A$, or at distance at least $\tau(\M)-3\eps$ from $\A$. 
Let us focus on the first case. Let $z\in Z(\A)$ be a point with $d_\A(z) \leq 4\eps$. By definition of critical points,  $z$ is the center of the smallest enclosing ball of $\pi_\A(z) \subset \A$, and this ball is at most of radius $4\eps$. According to \cite[Lemma 12]{attali2013vietoris}, this implies that
\begin{equation}
    d_\M(z) \leq \tau(\M) \p{1 - \sqrt{1-\frac{16\eps^2}{\tau(\M)^2}}} \leq \frac{8\eps^2}{\tau(\M)}\p{1+ \frac{16\eps^2}{\tau(\M)^2}} \leq \frac{10\eps^2}{\tau(\M)},
\end{equation}
where we use the condition $\eps<\tau(\M)/8$ and the inequality $\sqrt{1-u^2} \geq 1-u/2-u^2/2$ for $u\in [0,1]$. This yields the first alternative in \Cref{thm:crit_points}.

We now consider the second case, where $d_\A(z)>\tau(\M)-3\eps$. In that case,  the inequality $\eps\leq \tau(\M)/8$ implies that $d_\M(z)\geq d_\A(z)-\eps \geq \tau(\M)-4\eps >\tau(\M)/2$.  
We require a lemma, which is a variant of  Lemma 3.3 from \cite{Chazal_CS_Lieutier_OGarticle}  suited to manifolds.

\begin{lemma}\label{lem:sampling_improved}
Let $\M\subset \R^D$ be a compact $C^2$ submanifold. 
    Let $\A\subset \M$ be a subset with $d_H(\A,\M)\leq \eps$, and let $z\in Z(\A)$ with $0<r<d_\M(z)< R$. Then, there exists a $\mu$-critical point $z'\in Z_\mu(\M)$, with $\|z-z'\|\leq \eps$ and 
    \begin{equation}
        \mu \leq \frac{\eps}{r} \p{1+\frac{R}{2\tau(\M)}}.
    \end{equation}
\end{lemma}

\begin{proof}
    We adapt the proof of Lemma 3.3 in \cite{Chazal_CS_Lieutier_OGarticle}. We follow the gradient flow of $\nabla d_\M$ starting at $z\in Z(\A)$ along a trajectory parameterized by arc length. If the gradient flow reaches a critical point of $\M$ before time $\eps$, the result holds. Otherwise, let $w$ be the point reached by the gradient flow after time $\eps$. We have
\begin{equation}
    d_\M(w)-d_\M(z) = \int_0^\eps \|\nabla d_\M(C(t))\|\dd t,
\end{equation}
where $t\mapsto C(t)$ is the flow curve. Hence, there exists some point $z'$ along this curve with 
\[ \|\nabla d_\M(z')\| \leq \frac{d_\M(w)-d_\M(z)}{\eps}. \]
Furthermore, it holds that $\|z-w\|\leq \eps$. Let us show that $z'$ is a $\mu$-critical point of $\M$ for $\mu = \frac{\eps}{r}\p{1+\frac{R}{2\tau(\M)}}$. As $z\in Z(\A)$, we have according to \cite[Lemma 3.2]{Chazal_CS_Lieutier_OGarticle}
\begin{equation}\label{eq:bound_dSw}
    d_\M(w)^2 \leq d_\A(w)^2 \leq d_\A(z)^2 + \|z-w\|^2 \leq d_\A(z)^2+\eps^2.
\end{equation}
Let $x\in \pi_\M(z)$ and $y\in \pi_\A(x)$, whence $\|x-y\|\leq \eps$.  
We have
\begin{align}
    d_\A(z)^2 &\leq \|z-y\|^2 = \|z-x\|^2 + \|x-y\|^2 + 2\dotp{z-x,x-y} \nonumber \\
    &\leq d_{\M}(z)^2 + \eps^2 +  2\dotp{z-x,\pi_{x}^\perp(x-y)}, \label{eq:bound_dAz}
\end{align}
where we use that $z-x$ is orthogonal to $T_x \M$. According to \cite{federer1959curvature}, $\|\pi_{x}^\perp(x-y)\|\leq \frac{\|x-y\|^2}{2\tau(\M)}$. Hence, \eqref{eq:bound_dSw} and \eqref{eq:bound_dAz} yield that
\begin{align*}
  \|\nabla d_\M(z')\| &\leq \frac{d_\M(w)-d_\M(z)}{\eps} = \frac{d_\M(w)^2 -d_\M(z)^2}{\eps(d_\M(w)+d_\M(z))}  \\
  &\leq \frac{ 2 \eps^2 + \eps^2 \frac{R}{\tau(\M)}}{2\eps r} = \frac{\eps}{r} \p{1 + \frac{R}{2\tau(\M)}}.
\end{align*}
This concludes the proof.
\end{proof}

Remark that we always have the trivial bound $d_\M(z)\leq \mathrm{diam}(\M)$ for $z\in Z(\A)$. Hence, the  lemma implies that for $z\in Z(A)$ with $d_\M(z)>\tau(\M)/2$, there exists a $\mu$-critical point $z'\in Z_\mu(\M)$ with $\|z-z'\|\leq \eps$ and $\mu =C_2\eps$ for some constant $C_2$ depending on $\M$, which completes the proof of the first point.
\item This is a straightforward consequence of the first item and \ref{P4} - we only need to make $\eps_0>0$ small enough that $C_2\eps_0 < \mu_0$ (where $\mu_0$ is from the statement of \ref{P4}).

\item 

 Let $\delta$ and $\delta'$ be radii such that the statement of \Cref{lem:characterization_osculation}.\ref{it:local_osc} holds uniformly for all the  critical points $z_0\in Z(\M)$ and $x_0\in\pi_{\M}(z_0)$ (these are in finite number, as $\M$ satisfies \ref{P1} and \ref{P2}). Let $z\in Z(\A)$ be a critical point close to some critical point $z_0\in Z(\M)$, with $\|z-z_0\|\leq C_4\eps$ (where $C_4$ is the constant from \Cref{thm:crit_points}.\ref{it:usingP4}). Let $x\in \pi_\A(z)$. Note that $\|x-z_0\| \leq  \|x-z\| + \|z-z_0\| \leq \|x_0-z\| + \|z-z_0\|\leq d_\M(z_0) + 2\|z-z_0\|$. Hence, according to \Cref{lem:elementary_metric}, for $\|z-z_0\|$ small enough, the point $x$ is at distance less than $\delta$ from some vertex $x_0\in \pi_\M(z_0)$ (and this uniformly for each critical point $z_0$).  According to \Cref{lem:lipschitz_projection}, there exists a unique projection $x_1$ of $z$ on $\M_0=\M\cap \overline B(x_0,\delta)$, and $\|x_1-x_0\|\leq \|z-z_0\|/\alpha$. In particular, for $\|z-z_0\|$ small enough, $\|x_1-x_0\|<\delta$. Hence, $x_1$ belongs to the relative interior of $\M_0$ and the vector $z-x_1$ is orthogonal to $T_{x_1}\M$.

As $d_H(\A,\M)\leq \eps$, there exists $y\in \A$ with $\|x_1-y\|\leq \eps$.   We write thanks to \cite[Theorem 4.18]{federer1959curvature}
 \begin{align*}
     \|z-x\|^2 = d_\A(z)^2 &\leq \|z-y\|^2 = \|z-x_1\|^2 + \|x_1-y\|^2 + 2 \dotp{z-x_1,\pi_{x_1}^\perp(x_1-y)}\\
     &\leq d_{\M_0}(z)^2 + \eps^2 \p{1+\frac{d_{\M_0}(z)}{\tau(\M)}}.
 \end{align*}
Note that $1+d_{\M_0}(z)/\tau(\M)$ can be roughly upper bounded by 
\[ 1+\frac{ \diam(\M)}{\tau(\M)} \eqdef C'.\]  
 Hence, $x\in B(z,\sqrt{d_{\M_0}(z)^2+ C'\eps^2})$. As $x\in \M_0$, this implies by \ref{P3} that
 \[ \|x_1-z\|^2+C'\eps^2 \geq  \|x-z\|^2 \geq \|x_1-z\|^2 + \alpha\|x-x_1\|^2,\]
so that $\|x-x_1\|\leq \sqrt{C'/\alpha}\eps$. As we have $\|x_1-x_0\|\leq \frac{1}\alpha \|z-z_0\|$, we obtain
\[
     \|x-x_0\|\leq \|x-x_1\| + \|x_1-x_0\|\leq \sqrt{C'/\alpha}\eps+ \frac{1}\alpha C_4\eps.\qedhere
\]
\end{enumerate}
\end{proof}

\begin{remark}
    The reader will have noticed that the proof does not make use of all of Condition \ref{P1}, but rather only of the property that each critical point has only finitely many projections.
\end{remark}

Finally, we prove \Cref{cor:deterministic_sample}.
\begin{proof}[Proof of \Cref{cor:deterministic_sample}]
Consider the constants $C_4, C_5,\eps_0>0$ from Theorem \ref{thm:crit_points}, and let $\eps_1 = \min(\tau(\M)/8, \eps_0/C_5)$.
Let $\A$ be a $(\delta,\eps)$-sampling of $\M$ for some $0<\delta\leq\eps\leq \eps_1$.     Let $z\in Z(\M)$ with $\pi_\M(z)=\{x_1,\dots,x_s\}$. According to \Cref{thm:crit_points}, if  $z'\in Z(\A)$ is at distance less than $C_4\eps$ from $z$, then $\pi_\A(z) \subset \bigcup_{i=1}^s B(x_i,C_5\eps)$. Hence, the map
\[ 
 \begin{split}
(\A \cap B(x_1,C_5\eps)) \times \cdots \times (\A \cap B(x_s,C_5\eps)) & \longrightarrow    \R^D\\
(y_1,\dots,y_s) &\longmapsto m(\{y_1,\dots,y_s\})
    \end{split} 
    \]
contains all the points of $Z(\A)$ at distance less than $C_4\eps$ from $z$ in its image. As the cardinality of $\A \cap B(x_i,C_5\eps)$ is at most $N(\delta,C_5\eps)$, the total number of such points is $2^{N(\delta,C_5\eps)s}$. 

Let us now prove the second bound. Let $x\in \M$ and let $\{y_1,\dots,y_K\}$ be a maximal $(\delta,+\infty)$-sampling of $B(x,\eps)\cap \M$. Then, the balls $B(y_i,\delta/2)$ are pairwise disjoint for $i=1,\dots,K$. According to \cite[Proposition 8.7]{aamari2018stability}, as $\delta/2\leq 2\eps\leq 2\eps_1\leq \tau/4$, it holds that
\[  \Vol(B(y_i,\delta/2)\cap \M)\geq c_m\delta^m \text{ and } \Vol(B(x,2\eps)\cap \M)\leq C_m\eps^m\]
for two positive constants $c_m,C_m$. Hence, as $\bigsqcup_{i=1}^KB(y_i,\delta/2)\cap \M\subset  B(x,2\eps)\cap \M$, it holds that
\[ c_mK\delta^m \leq C_m\eps^m\]
that is $N(\delta,\eps)\leq \frac{C_m}{c_m}\p{\frac{\eps}{\delta}}^m$. 
\end{proof}

\section{The \texorpdfstring{$\mu$}{mu}-critical points of \texorpdfstring{$\M$}{M} and the Big Simplex Property }\label{sec:BSP}

Let as above $M$ be a $C^k$ compact manifold of dimension $m\geq 1$ for some $k\geq 2$, and let $i\in \Emb^2(M,\R^D)$ and $\M = i(M)$.
In this section, we concern ourselves with the position of the $\mu$-critical points of $\M$. In particular, whenever $\M$  satisfies \ref{P1}, \ref{P2} and \ref{P3}, we define a purely local condition on the critical points of $\M$ and their projections,  called the Big Simplex Property, that turns out to be equivalent to the apparently more global condition \ref{P4}, yet is easier to manipulate in terms of the embedding $i$. This will help us show that \ref{P4} is verified for a dense subset of the space $\Emb^2(M,\R^D)$ of embeddings later in Section \ref{sec:density}, using the language of jets.

\medskip


Let us start with some preliminary observations. 
Consider the compact $C^2$ submanifold $\M\subset \R^D$  
 and let us define the function
\begin{equation*}
    \begin{split}
\eta_\M: [0,1] & \longrightarrow [0,+\infty]\\
\mu &\longmapsto  \sup\{r :\ \exists x \in \R^D \text{ s.t. } \|\nabla d_\M(x)\|\leq \mu \text{ and } d(Z(\M), x) = r\},
    \end{split}
\end{equation*}
the supremum of the distances to $\M$ at which a $\mu$-critical point can be found.
If $\M$ is of diameter $\text{diam}(\M)>0$ and $x\in \R^D$ is at distance $L+\text{diam}(\M)$ from $Z(\M)\subset \Conv (\M)$, then it is at distance at least $L$ from $\M$ and the norm of the generalized gradient $\nabla d_\M(x)$ is lower-bounded by $\frac{\sqrt{L^2 - \frac{1}{2}\text{diam}(\M)^2}}{L}$. Hence, $\eta_M$ is  finite for $\mu\in [0,1)$, and for such a $\mu$  the supremum of the definition is a maximum  due to the lower semicontinuity of the norm of the gradient. 
On the other hand, we always have $\eta_\M(1) =+  \infty$. 
We are more interested in $\eta_\M(\mu)$ for small values of $\mu$, and we see that  $\lim_{\mu\stackrel{>} {\rightarrow} 0}\eta_\M(\mu) = 0$. Indeed, assume that it is not the case: there exists $\eps >0$  and a sequence $(x_n)_n \subset\R^D$ of points such that $\lim_{n\rightarrow \infty}\|\nabla d_\M(x_n)\|=0$ and such that $d(Z(\M), x_n)\geq \eps$ for all $n$. But thanks to the same argument as above, the sequence $(x_n)_n$ is bounded, hence a subsequence that converges to some point $x\not\in Z(\M)$ can be extracted, which yields a contradiction with the lower semicontinuity of the norm of the generalized gradient.

Property \ref{P4} states that there exist $C = C(\M)$ and $\mu_0 \in (0,1)$ such that for any $0\mu\in [0,\mu_0)$ we have $\eta_\M(\mu)\leq C\mu$, and will be our main object of interest in this section. Note that Property \ref{P4} is equivalent to the apparently stronger Property \ref{P4'}:
    \begin{enumerate}[start=4, label={(P\arabic*')}]

    \item For any $\mu_0\in (0,1)$, there exists a constant $C>0$ such that for every $\mu\in [0,\mu_0)$, the set $Z_\mu(\M)$ is included in a tubular neighborhood of size $C\mu$ of $Z(\M)$, that is every point of $Z_\mu(\M)$ is at distance less than $C\mu$ from $Z(\M)$.\label{P4'}
  
    \end{enumerate}

Indeed, let $\mu_0>0$ and $C>0$ be given to us by  \ref{P4}, and let $\mu_0'>0$ be as in the statement of \ref{P4'}. We have shown above that $\eta_\M(\mu_0')<\infty $. Let $z$ be a $\mu$-critical point of $\M$; if $\mu<\mu_0$, then $d(z,Z(\M))<C\mu$ using \ref{P4}, and if $\mu \in [\mu_0,\mu_0')$, then $d(z,Z(\M))\leq \eta_\M(\mu_0') = \mu_0 \frac{\eta_\M(\mu_0')}{\mu_0}\leq \mu \frac{\eta_\M(\mu_0')}{\mu_0}$, hence in all cases $d(z,Z(\M))\leq \mu\left( C + \frac{\eta_\M(\mu_0')}{\mu_0}\right)$ and $\M$ satisfies \ref{P4'}.

\subsection{Position of the \texorpdfstring{$\mu$}{mu}-critical points}\label{subsec:position_mu_critical_points}

As a first step to reformulating Condition \ref{P4}, we need to  understand the behavior of the medial axis $\mathrm{Med}(\M)$ around $z_0\in Z(\M)$. 
Assume  that $\M$  satisfies properties \ref{P1},\ref{P2} and \ref{P3}, and let $z_0\in Z(\M)$ with  projections $\{x_1,\ldots,x_s\} = \pi_\M(z_0)$, for some $s\in \{2,\ldots, D+1\}$. Let 
\[E =\mathrm{Vec}(\{x_{j_1}-x_{j_2}:\ j_1,j_2 \in [s]\}).\]
As  $\M$ satisfies \ref{P1}, the dimension of $E$ is $s-1$. Let us also write $\pi_{E} : \R^D\rightarrow E$ and $\pi_{E}^\perp:\R^D \rightarrow E^\perp$ the perpendicular projections. We will show that for $\mu_0$ small enough, the tangent space of the set $Z_{\mu_0}(\M) = \{z \in \R^D:\ \|\nabla d_\M(z)\|\leq \mu_0\}$ at $z_0$ (in the sense of Federer \cite{federer1959curvature}) is equal to the orthogonal $E^\perp$ of the vector space $E$. 
Namely, we show that every point $z$ of the medial axis close enough to $z_0$ are such that $\|\pi_E(z-z_0)\|$ is of order $\|z-z_0\|^2$, while for every direction $h\in E^\perp$, there is a point of $Z_{\mu_0}(\M)$ at distance of order $\|h\|^2$ from $z_0+h$.
This is illustrated in Figure \ref{fig:position_medial_axis}.

Let  $p_j:B(z_0,\delta')\to B(x_j,\delta)\cap \M$ for $j\in [s]$ be the local projections from Lemma \ref{lem:local_projections_well_defined} (for some $\delta',\delta>0$), which we know to be Lipschitz continous.  We define 
\begin{equation}
    \begin{split}
I: B(z_0,\delta') & \longrightarrow \mathcal{P}([s])\\
z &\longmapsto \{j\in [s]:\ p_j(z)\in \pi_{\M}(z)\}.
    \end{split}
\end{equation}
In particular, $\pi_\M(z) = \{p_j(z):\ j\in I(z)\}$. We also introduce the notation $m_I(z) = m(\{p_j(z):\ j\in I\})$ for $I\subset [s]$ and $z\in B(z_0,\delta')$, where we recall that $m(\sigma)$ denotes the center of the smallest enclosing ball of the set $\sigma$. 

\begin{figure}
    \centering
    \includegraphics{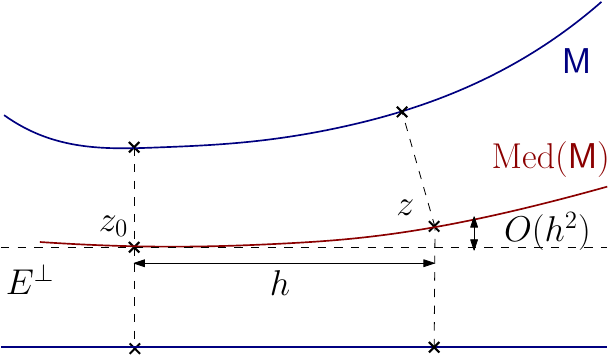}
    \caption{ Using the notations of Subsection \ref{subsec:position_mu_critical_points}, for $\mu_0>0$ small enough, the tangent space of $Z_{\mu_0}(\M)$ at the critical point $z_0\in Z(\M)$ is equal to $E^\perp$, and the points of $Z_{\mu_0}(\M)$ that are at a small distance $h$ from $z_0$ must be at distance $O(h^2)$ from the affine space $z_0 + E^\perp$.}
    \label{fig:position_medial_axis}
\end{figure}

For $I \subset [s]$, let 
\begin{equation}
    \begin{split}
g_I: B(z_0,\delta') & \longrightarrow \R\\
z &\longmapsto \frac{\|z-m_I(z)\|}{d_\M(z)}.
    \end{split}
\end{equation}
Note that $\|\nabla d_\M(z)\|=g_{I(z)}(z)$ for any $z\in B(z_0,\delta')$. 

The following Lemma states that $g_{[s]}$ and  $g_I$ (for any $I\subset [s]$) are Lipschitz continuous and  H\"older continuous respectively around $z_0$; it also yields some control over their Lipschitz and H\"older constants, which will be needed later in Section \ref{sec:C2_stability} where we let the embedding $i: M\rightarrow \R^D$ vary.

\begin{lemma}\label{lem:g_I}
    There exist constants $\delta'', \ell_1, \ell_2>0$ that depend only, and  continuously, on the Lipschitz constants of the projections $p_j:B(z_0,\delta')\to B(x_j,\delta)\cap \M$ and on the distance between $z_0$ and the boundary of the simplex with vertices $\{x_1,\dots,x_s\}$
    such that the function $g_{[s]}: B(z_0,\min(\delta',\delta''))  \longrightarrow \R$ is $\ell_1$-Lipschitz continuous and that for any $I\subset [s]$, 
    the function $g_I: B(z_0,\min(\delta',\delta''))  \longrightarrow \R$ is $(1/2)$-H\"older continuous with H\"older constant $\ell_2$.
\end{lemma}


\begin{proof}
The distance $d_\M(z_0)$ is lower bounded by the distance between $z_0$ and the boundary of the simplex with vertices $\{x_1,\dots,x_s\}$ and the function $z\mapsto d_\M(z)$ is also lower bounded and Lipschitz continuous on a ball of radius $d_\M(z_0)/2$ around $z_0$.
We know from Lemma \ref{lem:lipschitz_projection} that the projections $p_j:B(z_0,\delta')\to B(x_j,\delta)\cap \M$ are Lipschitz continuous.
As the function $A\mapsto m(A)$ is locally $1/2$-H\"older continuous with respect to the Hausdorff distance \cite[Lemma 16]{attali2013vietoris}, the statement regarding H\"older continuity follows.

By hypothesis, $z_0$ belongs to the interior of the simplex spanned by $\{x_1,\dots,x_s\}$. Hence, $m_{[s]}(z)$ also belongs to the interior of the simplex spanned by $\{p_1(z),\dots,p_s(z)\}$ for $z$ close enough to $z_0$ (depending  on the Lipschitz constants of the $p_j$s and how far  $z_0$ is from the boundary of the simplex spanned by $\{x_1,\dots,x_s\}$). When this is the case, the function $m_{[s]}(z)$ is actually equal to the center of the circumsphere of the points $p_1(z),\dots,p_s(z)$.
When restricted to the simplices whose circumsphere has its center at distance at least $c>0$ from their boundary, the coordinates of the center of the circumsphere is a Lipschitz function of the coordinates of the vertices for a Lipschitz constant that depends on $c$.
Thus $m_{[s]}(z)$ is Lipschitz continuous with respect to the coordinates of the points $p_1(z),\dots,p_s(z)$, hence with respect to those of $z$, when $z$ is close enough to $z_0$, and so is the function $g_{[s]}$ as a composition of Lipschitz maps, with the definition of ``close" and the Lipschitz constant depending on the Lipschitz constants of the $p_j$s and how far  $z_0$ is from the boundary of the simplex spanned by $\{x_1,\dots,x_s\}$.
\end{proof}


As another step towards understanding the local behavior of the $\mu$-critical points around $z_0$, the next lemma states that $\mu$-critical points (for $\mu$ small enough) that are close to $z_0$ must belong to $z_0 + E^\perp$ as a first approximation.

\begin{lemma}\label{lem:mu_critical_implies_perpendicular}
There exist $\mu_0,\delta_0,C>0$ such that for every $\mu_0$-critical point $z\in B(z_0,\delta_0)$, we have  $I(z)=[s]$ and
\[ \|\pi_{E}(z-z_0)\|\leq C\|z-z_0\|^2.\]
\end{lemma}

\begin{proof}
Let $\delta, \delta'>0$ and $p_j:B(z_0,\delta')\rightarrow  B(x_j,\delta)\cap \M$ for $j=1,\ldots, s$ be as in \Cref{lem:local_projections_well_defined}, and such that the functions $g_I$ are H\"older continuous (\Cref{lem:g_I}). 
Note that $g_I(z_0)>0$ for all $I\subsetneq [s]$.  Hence, it holds that if $I\subsetneq [s]$, then
\[
g_I(z)> \frac 12\min_{I\subsetneq [s]} g_I(z_0) \eqdef \mu_0
\]
for $\|z-z_0\|\leq \delta_1$ for some $\delta'>\delta_1>0$. If  $z\in B(z_0,\delta_1)$ is $\mu_0$-critical, then 
\[ \|\nabla d_{\M}(z)\| = \frac{\|z-m_{I(z)}(z)\| }{d_{\M}(z)}  =g_{I(z)}(z) \leq \mu_0,\]
hence $I(z)=[s]$. This proves the first point.

Let us now prove the second statement. Let $z\in B(z_0,\delta_1)$ be a $\mu_0$-critical point and $h=z-z_0$. 
For any $j\in [s]$,
\begin{align*}
    \|z-p_j(z)\|^2 &= \|(z_0-x_j) - (p_j(z)-h-x_j)\|^2 \\
    & = \|z_0 -x_j\|^2 +2 \dotp{z_0 -x_j,h+x_j-p_j(z)} + \| h +x_j-p_j(z)\|^2.
\end{align*}
As $z_0\in\Conv (\{x_j:\ j\in [s]\})$, we have $z_0-x_j \in E$. This implies that $\dotp{z_0 -x_j, h} = \dotp{z_0 -x_j, \pi_{E}(h)}$. Also, as $p_j$ is Lipschitz continuous, it holds that $\|p_j(z)-x_j\|\leq L h$ for some constant $L>0$. Furthermore, as $z_0-x_j\in T_{x_j} M^\bot$ and $p_j(z)\in  B(x_j,\delta)\cap \M$, according to \cite[Theorem 4.8.7]{federer1959curvature},
\[ |\dotp{z_0-x_j, x_j-p_j(z)}| =|\dotp{z_0-x_j, \pi_{x_j}^\bot(x_j-p_j(z))}| \leq \frac{d_{\M_j}(z) \|x_j-p_j(z)\|^2}{2\tau(\M)}\leq C_1\|h\|^2 \]
 for some constant $C_1$ depending on $\tau(\M)$,  $d_\M(z_0)$ and $L$. 
 In total, we have shown that
\begin{equation}
      \|z-p_j(z)\|^2= \|z_0 -x_j\|^2 +2 \dotp{z_0 -x_j,\pi_{E}(h)} +R_j(h),
\end{equation}
where $|R_j(h)|\leq C_2 \|h\|^2$ for some positive constant $C_2$. 
Now, by definition $\|z-p_{j_1}(z)\|^2 = \|z-p_{j_2}(z)\|^2$ for any $1\leq j_1,j_2\leq s$ (and $s$ is at least $2$). Hence,
\[ \|z_0 -x_{j_1}\|^2 +2 \dotp{z_0 -x_{j_1}, \pi_{E}(h)} + R_{j_1}(h) = \|z_0 -x_{j_2}\|^2 +2 \dotp{z_0 -x_{j_2}, \pi_{E}(h)} + R_{j_2}(h),\]
so
\[  |\dotp{x_{j_1} -x_{j_2}, \pi_{E}(h)}|\leq |R_{j_1}(h)+R_{j_2}(h)|.\]
As $E$ is spanned by the vectors $x_{j_1} -x_{j_2}$, this implies that $\|\pi_{E}(h)\|\leq C_4h^2$ for some $C_4$ that depends on the geometry of the nondegenerate simplex spanned by the points $x_j$.
\end{proof}

The geometry of the points close to the critical point $z_0$ that belong to the most central stratum of the medial axis, i.e. those such that $I(z) = [s]$, can be more precisely described:
\begin{lemma}\label{lem:core_medial_axis_is_C2}
There exists $\delta>0$ such that the local projections $p_j:B(z_0,\delta)\to B(x_j,\delta')\cap \M$ for $j\in [s]$ from Lemma \ref{lem:local_projections_well_defined} are well-defined and that the \textit{core medial axis} of $\M$ around $z_0$ 
\[M_c(z_0,\M):= \{z \in  B(z_0,\delta)  :\; I(z) = [s]\}\] 
is a $(D-s+1)$-dimensional $C^2$ submanifold with tangent space $T_{z_0}M_c(z_0,\M) = E^\perp$ in $z_0$.
\end{lemma}

\begin{proof}
Let $\delta, \delta'>0$ and $p_j:B(z_0,\delta)\rightarrow  B(x_j,\delta')\cap \M = \M_j$ for $j=1,\ldots, s$ be as in \Cref{lem:local_projections_well_defined}.
When restricted to $B(z_0,\delta)$, the function $z\mapsto d(z,\M_j)^2$ is equal to $\phi_j :B(z_0,\delta) \rightarrow \R, z \mapsto \| z - p_j(z)\|^2$.
We know that the map $p_j$ is $C^1$, hence so is $\phi_j$.
Its differential is
\[(d\phi_j)_z : h \mapsto 2\dotp{z-p_j(z), h - d(p_j)_z(h)}  = 2\dotp{z-p_j(z), h}  \]
for $h \in \R^ D$, as $\Im (d(p_j)_z(h)) \subset T_{p_j(z)}\M $ and thus $ \Im (d(p_j)_z(h))\perp z-p_j(z)$.
This shows that $\phi_j$ is in fact $C^2$.

By definition, the core medial axis $M_c(z_0,\M)$ is defined by the $s-1$ equations
\[ \phi_1(z) - \phi_j(z) = 0 \qquad j=2,\ldots, s\]
on $B(z_0,\delta)$.
The differential of the function $z \mapsto \phi_1(z) - \phi_j(z)$ is $h \mapsto 2\dotp{p_j(z)-p_1(z), h}$, and the $s-1$ vectors $p_j(z)-p_1(z)$ form a linearly independent family for $z_0 = z$, as $\M$ satisfies \ref{P1}. By continuity, we can make $\delta>0$ smaller to ensure that  this remains true for all $z\in B(z_0,\delta)$.
Then $M_c(z_0,\M)$ is the zero set of the submersion $F:B(z_0,\delta) \rightarrow \R^{s-1},z\mapsto (\phi_1(z) - \phi_j(z))_{j\in [s]}$, and the kernel of the differential $d_{z_0}F$ is the perpendicular of $\mathrm{Vec}(\{p_j(z_0)-p_1(z_0) :\ j=2,\ldots,s\}) = E$, i.e. $E^\perp$. This is   enough to conclude.
\end{proof}

We now show that $\mu_0$-critical points can be found by moving from the critical point $z_0$ in any orthogonal directions $h\in E^\perp$, up to a quadratic error of order $\|h\|^2$; this is essentially a consequence of the previous lemma.

\begin{lemma}[Existence of $\mu$-critical points]\label{lem:existence_of_mu_critical_points}
For any $\mu_0\in (0,1)$, there exist $\delta_1,C>0$ such that for any $h\in E^\perp$ with $\|h\|<\delta_1$, there exists a $\mu_0$-critical point $z $ of $\M$ with $\pi_{E^\perp}(z-z_0) =h $, $\|\pi_{E}(z-z_0)\|\leq C \| h\|^2 $ and such that $I(z)= [s]$.
\end{lemma}

\begin{proof}
This is a direct consequence of \Cref{lem:core_medial_axis_is_C2}. Indeed, as $T_{z_0}M_c(z_0,M)=E^\perp$, the orthogonal projection $\psi: z\in M_c(z_0,M)\to \pi_{E^\perp}(z-z_0)\in E^\perp$ is a local diffeomorphism at $z_0$. Its inverse $\psi^{-1}$ is defined on a small ball around $0$ in $E^\perp$. If $h$ is in this ball, then $z=\psi^{-1}(h)$ is such that $\pi_{E^\perp}(z-z_0)=h$. Furthermore, 
\[ \|\pi_E(z-z_0)\| = \|z-z_0-\pi_{E^\perp}(z-z_0)\| = \| \psi^{-1}(h)-\psi^{-1}(0) -h\|\leq C\|h\|^2\]
for some constant $C$ as $\psi^{-1}$ is of class $C^2$ and  $d\psi^{-1}_0$ is equal to the identity on $E^\perp$. The point $z$ is $\mu$-critical for $\mu=\|\nabla d_\M(z)\|$. As $I(z)=[s]$, we have $\|\nabla d_\M(z)\|=g_{[s]}(z)$, which is a Lipschitz-continuous function on a neighborhood of $z_0$ according to \Cref{lem:g_I}. As $\|\nabla d_\M(z_0)\| = 0$,  $\mu$ will be smaller than $\mu_0$ for any $\|h\|<\delta_1$ if $\delta_1$ is small enough..
\end{proof}

\subsection{The Big Simplex Property}\label{subsec:BSP}
In the following proposition, we translate the global condition \ref{P4} into a more practical and purely local (around a critical point $z_0\in Z(\M)$) condition, the \textit{Big Simplex Property}, that can be expressed in terms of the volumes of the simplices formed by the points $z$ close to $z_0$  and their projections $p_j(z)$ for $j\in [s]$, under the assumption that $z-z_0$ is perpendicular to $E = \mathrm{Vec}(\{x_{j_1}-x_{j_2}:\ 1\leq j_1,j_2 \leq s\})$. The proof makes direct use of the two previous lemmas shown in \Cref{subsec:position_mu_critical_points}.

\begin{proposition}[Big Simplex Property]\label{prop:big_simplex_is_enough}
Let $\M \subset \R^D$ be a compact $C^2$ submanifold that satisfies conditions \ref{P1}, \ref{P2}, \ref{P3}.
Then the following ``Big Simplex Property'':

\begin{enumerate}[start=1, label={(BSP)}]
\item \label{BSP} For any critical point $z_0\in Z(\M)$, there are constants $\delta',\delta,L>0$ 
such that if we write $\pi_\M(z_0) = \{x_1,\ldots,x_s\}$, the local projections $p_j:B(z_0,\delta')\to B(x_j,\delta)\cap \M$ from Lemma \ref{lem:local_projections_well_defined} are well-defined and  the following holds.   For   $h\in E^\perp$ with $\|h\|\leq \delta'$, we let $\Delta(h)$ be the $s$-simplex with vertices $z=z_0+h$ and $p_j(z)$ for $1\leq j \leq s$. Then,
the $s$-volume of $\Delta(h)$ satisfies
\begin{equation}\label{eq:inequality_BSP}
\Vol_{s}(\Delta(h))\geq L \|h\|.
\end{equation}
\end{enumerate}
is equivalent to Property \ref{P4} introduced earlier:
\begin{enumerate}[start=4, label={(P\arabic*)}]
\item There exist constants $C>0$ and $\mu_0\in (0,1)$ such that for every $\mu\in [0,\mu_0)$, the set $Z_\mu(\M)$ is included in a tubular neighborhood of size $C\mu$ of $Z(\M)$, that is every point of $Z_\mu(\M)$ is at distance less than $C\mu$ from $Z(\M)$.
\end{enumerate}
\end{proposition}

Before proving the equivalence between \ref{BSP} and \ref{P4}, we provide an example of a simple manifold $\M$ \emph{not} satisfying either conditions (albeit satisfying \ref{P1}, \ref{P2} and \ref{P3}). See also Figure \ref{fig:counterexample_P4} for an illustration.

\begin{example}\label{ex:counterexample_P4}
Consider the  $C^2$ submanifold $\M\subset \R^2$ that is equal on $[-1,1]\times [-2,2]$ to the union of the graph of the functions $x\mapsto 1+x^3$ and $x\mapsto -1-x^3$, see \Cref{fig:counterexample_P4}. 
The distance function to $\M$ admits a single critical point $z_0=(0,0)$ whose projections define a non-degenerate simplex, and the sphere $S(z_0,1)$ is non-osculating $\M$ at each projection. On the other hand, for each $x>0$ the projections of the point $p(x):=(x+3x^2+3x^5,0)$ are $(x, 1+x^3)$ and $(x,-1-x^3)$, hence the norm of the generalized gradient of $d_\M$ at $p(x)$ is
$$\|\nabla d_\M(p(x))  \| = \frac{d(p(x),(x,0))}{d(p(x),(x,1+x^3))} =\frac{3x^2 + O(x^5)}{1 + O(x^3)} = 3x^2 + o(x^2),$$
while the distance from $p(x)$ to $z_0$ is $x+o(x)$. This shows that $\M$ does not satisfy condition \ref{P4}. It is easy to modify this example so that $\M$ be compact, yet the behaviour of the $\mu$-critical points around $z_0$ remain the same.

\begin{figure}
    \centering
    \includegraphics[width=0.7\textwidth]{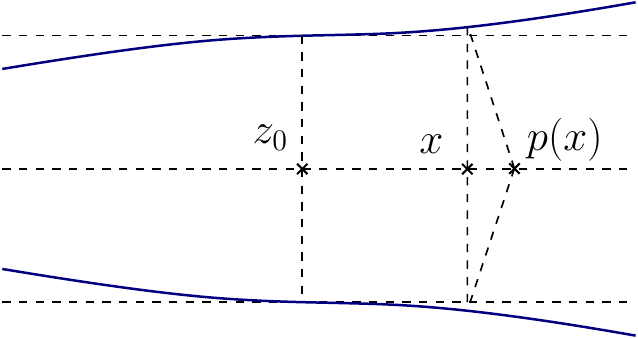}
    \vspace{0.3cm}
    \caption{The point $p(x)$ satisfies $m(\pi_{\M}(p(x)))=x$, with $\|\nabla d_\M(p(x))  \|=3x^2 + o(x^2) \ll \|x\|$ for small values of $x$. Hence, $p(x)$ is a $\mu$-critical point with $\mu= 3x^2+ o(x^2)$ but at distance from $z_0$ greater than $x\gg \mu$  (when $x$ is small). }
    \label{fig:counterexample_P4}
\end{figure}

\end{example}

Let us now prove \Cref{prop:big_simplex_is_enough}.
\begin{proof}

Let $z_0\in Z(\M)$ be a critical point with projections $\pi_\M(z_0) = \{x_1,\ldots,x_s\}$, and let $\delta',\delta, \mu_0>0$ be small enough that $p_j: B(z_0,\delta') \rightarrow B(x_j,\delta)\cap \M$ is as in Lemma \ref{lem:local_projections_well_defined} for $j=1,\ldots,s$ and that the conclusions of Lemma \ref{lem:mu_critical_implies_perpendicular} stand, i.e.~if $z\in B(z_0,\delta')$ is $\mu_0$-critical, then $I(z)=[s]$ and 
\[ \|\pi_{E}(z-z_0)\|\leq C\|z-z_0\|^2\]
for some $C>0$ that depends on $z_0$. Further assume that the conclusion of \Cref{lem:g_I} hold, that is the function $g_{[s]}$ is $\ell$-Lipschitz continuous.  
Let $z\in B(z_0,\delta')$ be $\mu_0$-critical. Then
\begin{equation}\label{eq:reduce_to_perp}
   |\|\nabla d_\M(z)\|-g_{[s]}(z_0+ h)| =   |g_{[s]}(z) - g_{[s]}(z_0+ h)|\leq \ell C_0\|z-z_0\|^2.
\end{equation}
Now for $h\in E^\perp$ with $\|h\|\leq \delta'$, let $\tilde{\Delta}(h)$ be the $(s-1)$-simplex whose vertices are the points $p_j(z_0+h)$ for $j\in [s]$, and let $\Delta(h)$ be the $s$-simplex obtained by adding the vertex $z_0+h$ to those of $\tilde\Delta (h)$. Let $\tilde E(h)$ be the affine space spanned by $\tilde \Delta(h)$.
The $s$-volume of $\Delta(h)$ can be computed as 
\[ \mathrm{Vol}_{s}(\Delta(h)) = \frac{1}{s} \mathrm{Vol}_{s-1}(\tilde{\Delta}(h))d(z_0+h,\tilde E(h)). \]
Note that according to \ref{P1}, when $h=0$, the projection of $z_0+h$ on $\tilde E(h)$ belongs to the interior of the simplex $\tilde\Delta(h)$. By continuity of the projections $p_j$, this property is still valid when $h$ is small enough. Hence, for $h$ small enough,
\begin{equation}\label{eq:volume_delta_h}
    \mathrm{Vol}_{s}(\Delta(h)) = \frac{1}{s} \mathrm{Vol}_{s-1}(\tilde{\Delta}(h))d(z_0+h,\tilde \Delta(h)). 
\end{equation} 

\textbf{(BSP)$\implies$(P4)}: Assume  that \ref{BSP} is satisfied. Then, $d(z_0+h,\tilde \Delta(h)) \leq \|z_0+h-m_{[s]}(z_0+h)\| = d_\M(z_0+h)g_{[s]}(z_0+h)$. Furthermore, the volume of $\mathrm{Vol}_{s-1}(\tilde{\Delta}(h))$ can be crudely bounded by $(\diam(M))^{s-1}$.
Hence, if $z$ is a $\mu_0$-critical point close enough to $z_0$ that $h=\pi_{E^\perp}(z-z_0)$ satisfies the condition \ref{BSP}, we obtain that
\[ L\|h\|\leq \mathrm{Vol}_{s}(\Delta(h))  \leq \frac{(\diam(M))^{s-1}}{s} d_\M(z_0+h)g_{[s]}(z_0+h) \leq \frac{(\diam(M))^{s}}{s} g_{[s]}(z_0+h). \]
Hence, \eqref{eq:reduce_to_perp} yields
\begin{align*}
    \|\nabla d_\M(z)\| \geq \frac{s}{(\diam(M))^s} L\|h\| - \ell C_0\|z-z_0\|^2.
\end{align*}
\Cref{lem:mu_critical_implies_perpendicular} also implies that $\|h\|\geq \|z-z_0\|/2$ if $\|z-z_0\|$ is small enough. Hence, we obtain for $z\in B(z_0,\delta')$ and $\delta'$ small enough
\[  \|\nabla d_\M(z)\|  \geq \|z-z_0\| \p{\frac{sL}{2(\diam(M))^s}- \delta'\ell C_0}=\frac{\|z-z_0\|}{C_1},\]
where $C_1>0$ for $\delta'$ small enough. In particular, if $z$ is a $\mu$-critical point for $z\in B(z_0,\delta')$ with $\mu<\mu_0$, then $\|z-z_0\|\leq C_1\mu$.

We have shown that \ref{BSP} implies \ref{P4} "locally". To conclude ``globally", we proceed as follows : the constants $\mu_0, \delta'$ and $C_1$ depend on the chosen critical point $z_0$. As $\M$ satisfies property \ref{P2}, it admits only finitely many critical points. Let $\overline{\mu}_0$,  $\overline{\delta}'$ and $\overline{C}_1$ be the minima over the critical points of the associated constants $\mu_0$, $\delta'$ and $C_1$ respectively. 
Then any $\overline{\mu}_0$-critical $z\in \R^D$ that is at distance less than $\overline{\delta}'$ from a critical point $z_0$ of $\M$ verifies
\begin{equation}\label{eq:bound_for_small_mu}
   \overline{C}_1 \|\nabla d_\M(z)\|   \geq    \|z-z_0\| \geq  d(z,Z(\M)).
\end{equation}
Now remember that we have shown at the start of this section that $\eta_{\M}(r)= \sup\{r :\ \exists x \in \R^D \text{ s.t. } \|\nabla d_\M(x)\|\leq \mu \text{ and } d(Z(\M), x) = r\}$ goes to $0$ as $\mu\rightarrow 0$. This means that there exists $0<\mu_1<\overline{\mu}_0$ such that any $\mu_1$-critical point $z$ must be at distance less than $\overline{\delta}'$ from $Z(\M)$, where \eqref{eq:bound_for_small_mu} applies. This allows us to conclude.
\medskip

\textbf{(P4)$\implies$(BSP):} Assume now that \ref{P4} is satisfied with constants $C$ and $\mu_0$, and consider again $z_0\in Z(\M)$ with the associated spaces $E$ and $E^\perp$. Let $h\in E^\perp$ be small enough  that \Cref{lem:existence_of_mu_critical_points} applies: there exists a $\mu_0$ critical point $z$ of $\M$ with $\pi_{E^\perp}(z-z_0)=h$, $I(z)=[s]$ and $\|\pi_E(z-z_0)\|\leq C_2 \|h\|^2$ for some constant $C_2>0$ that is independent from the choice of $h$. Condition \ref{P4} ensures that
\begin{equation}\label{eq:implication_P4}
   C \|\nabla d_\M(z)\|\geq \|z-z_0\| \geq \|h\|.
\end{equation} 
 Let $\tilde \Delta(z-z_0)$ be the simplex with vertices $p_j(z)$ for $j\in [s]$ (assuming that $h$ is small enough for the projections $p_j(z)$ to be well-defined). 
By continuity, as $\tilde \Delta(0) $ is a nondegenerate simplex (since $\M$ satisfies \ref{P1}), the volume $\mathrm{Vol}_{s-1}(\tilde{\Delta}(h))$ is lower bounded for $h$ small enough. 
As the projections $\pi_\M(z)$ coincide with the vertices of $\tilde \Delta(z-z_0)$, we also have that $d(z,\tilde \Delta(z-z_0))= \|z-m_{[s]}(z)\|$ (as noted in \cite[Section 2.1]{Chazal_CS_Lieutier_OGarticle}).
Thus, starting from Equation \eqref{eq:volume_delta_h} and for $h$ small enough,
\begin{align*}
   \Vol_s(\Delta(h))&= \frac{1}{s} \mathrm{Vol}_{s-1}(\tilde{\Delta}(h))d(z_0+h,\tilde \Delta(h))  \\
   &\geq C_3 (d(z,\tilde \Delta(z-z_0)) - \|z_0+h-z\| - d_H(\tilde \Delta(h),\tilde \Delta(z-z_0))) \\
   &\geq C_3(\|z-m_{[s]}(z)\| - C_2\|h\|^2 - \max_{j\in [s]}\|p_j(z)-p_j(z_0+h)\|) \\
   &\geq C_3(d_\M(z) \|\nabla d_\M(z)\| - C_2\|h\|^2 - K C_2\|h\|^2) \\
   &\geq C_3(C_4 \|h\| - C_2\|h\|^2 - K C_2\|h\|^2)
\end{align*}
where we use the $K$-Lipschitz continuity of the $p_j$s on a neighborhood of $z_0$ for some constant $K>0$, Inequality \eqref{eq:implication_P4}, and the fact that $d_\M(z)$ is lower bounded on a neighborhood of $z_0$. Hence, for $\|h\|$ smaller than some constant $\delta''$, this quantity is larger than $\frac{C_3C_4}{2}\|h\|=L\|h\|$. This proves that \ref{BSP} holds at any given point $z_0\in \M$, concluding the proof.
\end{proof}

We choose to translate \ref{P4} into the Big Simplex Property because the volume of the simplex $\Delta(h)$ (in the notations of the proof of Proposition \ref{prop:big_simplex_is_enough} above) can be easily expressed in terms of the first derivative of the projections $p_i$, which themselves depend on the shape operator of $\M$ around the points $x_i \in \pi_\M(z_0)$; this is formally stated in the Lemma \ref{lem:formula_simplex_volume} further below. This will be used for both showing that the condition is open in \Cref{sec:C2_stability} (along with Conditions \ref{P1}, \ref{P2} and \ref{P3}) and that it is dense in \Cref{sec:density}.

\subsection{Computing the volume of the simplex}
We compute an approximation of the volume of the simplex $\Delta(h)$ introduced in Subsection \ref{subsec:BSP} expressed using the differentials of the local projections $p_i: B(z_0,\delta') \rightarrow B(x_i,\delta)\cap \M$, where $\M$ is as above a $C^2$ compact manifold satisfying \ref{P1}, \ref{P2} and \ref{P3}. This allows us to characterize \ref{BSP} (and hence \ref{P4}) purely locally, in terms of the critical points $z_0$, their projections $x_i$ and the curvature of $\M$ at the points $x_i$ which determines the differentials of the maps $p_i$.
\begin{lemma}\label{lem:formula_simplex_volume}
    Let $z_0\in Z(\M)$ be a critical point of $\M$ with projections $\pi_{\M}(z_0)=\{x_1,\ldots,x_s\}$. Let $\delta',\delta>0$ be small enough that the projections $p_i: B(z_0,\delta') \rightarrow B(x_i,\delta)\cap \M$ from Lemma \ref{lem:local_projections_well_defined} are well-defined. For $h\in E^\perp$ small enough, let as before $\Delta(h)$ be the  $s$-simplex with vertices $z_0+h$ and $p_j(z_0+h)$ for $j\in [s]$. Let $X=[x_1-z_0,\dots,x_s-z_0]\in \R^{D\times s}$, and let $R_i(h)$ be the matrix obtaining by replacing the $i$th column of $X$ by $r_i(h) = p_i(z_0+h)-x_i-h$. Define likewise $A_i(h)$, where we replace the $i$th column of $X$ by  $(dp_i)_{z_0}(h)-h$. Then, for any $h \in E^\perp$ with $ \|h\|< \delta' $ it holds that
\begin{equation*}
\begin{split}
\Vol _{s}(\Delta(h))^2 &= \frac{1}{(s!)^2}\sum_{i,j =1}^{s}  \det(R_{i}(h)^\top R_{j}(h))+ O(\|h\|^3) \\
&= \frac{1}{(s!)^2} B(h)+ o(\|h\|^2),
\end{split}
\end{equation*}
where $B$ is the quadratic form given by
\begin{equation}\label{eq:first_form}
 B(h) =  \sum_{i,j =1}^{s}  \det(A_{i}(h)^\top A_{j}(h)).
 \end{equation}
 Furthermore, the constant in the term $O(h^3)$ depends only on $s$, $\diam(M)$, and the Lipschitz constants of the projections $p_j$, $j\in [s]$.
\end{lemma}

Controlling the constant in the term $O(h^3)$ will be needed in Section \ref{sec:C2_stability}, once we let the embedding of the manifold into $\R^D$ vary.

\begin{proof}
Let $h \in E^\perp$ be such that $ \|h\|< \delta' $  and let
\[ v_j= p_j(z_0+h)-(z_0+h)=(x_j-z_0) +r_j(h) =(x_j-z_0)  + (d p_j)_{z_0}(h)-h + o(\|h\|)\] 
for $j\in [s]$. Let $L(h)$ be the $l\times l$ matrix with $L_{i,j}(h)=\dotp{v_i,v_j}$. Then,
\[ \Vol_s(\Delta(h)) = \frac{1}{s!} \det \left(L(h)\right)^{\frac{1}{2}}.\]
Define the $D\times s$ matrices $V=[v_1,\dots,v_s]$, $X=[x_1-z_0,\dots,x_s-z_0]$. For $\tau\subset [s]$, let $R_{\tau}(h)$ be the matrix obtained from $X$ by replacing the $i$th column of $X$ by $r_i(h)$  for all $i\in \tau$. We write $R_i(h)=R_{\{i\}}(h)$ and $R_{i,j}(h)=R_{\{i,j\}}(h)$.  
By multilinearity of the determinant,  as the columns of the matrix $V^\top V$ are of the form $V^\top (x_i-z_0) + V^\top r_i(h)$, it holds that
\begin{align*}
    \det(L(h)) &= \sum_{\tau\subset [s]} \det(V^\top R_{\tau}(h)) \\
    &= \det(V^\top X) + \sum_{j=1}^s \det(V^\top R_{j}(h)) + \sum_{i\neq j}\det(V^\top R_{i,j}(h))+ O(\|h\|^3).
\end{align*}
Similarly, we write for any matrix $M$ of size $D\times l$,
\begin{align*}
    \det(V^\top M) = \det(X^\top M) + \sum_{i=1}^s \det( R_{i}(h)^\top M) + \sum_{i\neq j}\det( R_{i,j}(h)^\top M)+ O(\|h\|^3).
\end{align*}
Hence,
\begin{align*}
    \det(L(h)) &= \det(X^\top X) + 2\sum_{i=1}^s \det(R_{i}(h)^\top X) + 2\sum_{i\neq j}\det(R_{i,j}(h)^\top X) \\
    &\quad + \sum_{i,j=1}^l \det(R_{i}(h)^\top R_{j}(h)) + O(\|h\|^3).
\end{align*}
Any matrix of the form $X^\top M$ is non-invertible as $X$ is not of full rank. Hence,
\begin{align*}
    \det(L(h))=\sum_{i,j=1}^l \det(R_{i}(h)^\top R_{j}(h)) + O(\|h\|^3).
\end{align*}
Note that the $O(\|h\|^3)$ can actually be made explicit, and is given as a sum of terms of the form $\det(R_\tau(h)^\top R_{\tau'}(h))$ where $|\tau|+|\tau'|\geq 3$. 
As the $p_j$s are $\ell$-Lipschitz continuous for some $\ell>0$, we have $\|r_j(h)\|\leq (\ell+1)\|h\|$ for $j\in [s]$. As $\|x_j-z_0\| \leq \diam(\M)$, we can use the multilinearity of the determinant to show that $|\det(R_\tau(h)^\top R_{\tau'}(h))|$ is smaller than 
\[ C\diam(\M)^{s-|\tau|}((\ell+1)\|h\|)^{|\tau|}\diam(\M)^{s-|\tau'|}((\ell+1)\|h\|)^{|\tau'|},\]
where $C$ is the maximal determinant of a $s\times s$ matrix of dot products of vectors of $\R^D$ of norm smaller than $1$. This is of order $\|h\|^3$, with a constant depending only on $\diam(\M)$, $s$ and $\ell$.

It remains to prove the second equality in \Cref{lem:formula_simplex_volume}, which follows from the equality $R_{i}(h) = A_{i}(h)+o(\|h\|)$ and that the $i$th columns of $A_i(h)$ and $R_i(h)$ are of order $O(\|h\|)$.
\end{proof}

We get the following characterization of \ref{BSP} as a simple corollary of \Cref{lem:formula_simplex_volume}:

\begin{cor}\label{cor:BSP_equivalent_BI_non_degenerate}
 Let $\M$ be a $C^2$ compact manifold satisfying \ref{P1}, \ref{P2} and \ref{P3}. Then $\M$ satisfies \ref{P4} if and only if for any $z_0\in Z(\M)$ the associated quadratic form $B:E^\perp\rightarrow\R$ defined in \Cref{lem:formula_simplex_volume} is non-degenerate.
\end{cor}
\begin{proof}
Property \ref{BSP} is satisfied if and only if for any $z_0\in Z(\M)$, there exist $L,\delta'>0$ such that for any $h\in E^\perp$ with $\|h\|\leq \delta'$, the $s$-volume of the associated $s$-simplex $\Delta(h)$ satisfies 
$$\Vol_{s}(\Delta(h))\geq L \|h\|.$$
As we have shown in \Cref{lem:formula_simplex_volume} that
$\Vol _{s}(\Delta(h)) = \frac{1}{s!}\sqrt{B(h) + o(\|h\|^2)}$,
this is true if and only if $B:E^\perp\rightarrow\R$  is non-degenerate for any such $z_0$. As \ref{BSP} is equivalent to \ref{P4} according to \Cref{prop:big_simplex_is_enough}, the conclusion holds.
\end{proof}

\section{Topological Morse theory for distance functions to generic submanifolds
}
\label{sec:Morse}

Let $\M \subset \R^D$ be a $C^2$ compact submanifold.
In this section, we show that Conditions \Pall\ on $\M$ lead to Morse-like properties for the distance function $d_\M$ to $\M$, as stated in Theorem \ref{thm:Morse}.
A similar exposition can be found in \cite{song2023generalized} in the case of hypersurfaces. 

Let us first define \textit{topological Morse functions}, which were initially introduced in \cite{morse_topologically_1959}.
   Let $U\subset \R^d$ be an open set and let $f:U\to \R$ be a continuous function. 
A point $z\in U$ is said to be a \textit{topological regular point of $f$} if there is a homeomorphism $\phi:V_1\to V_2$ between open neighborhoods $V_1$ of $0$ in $\R^d$ and $V_2$ of $U$ in $\R^d$ with $\phi(0)=z$ and such that for all $x=(x_1,\dots,x_d)\in V_1$, 
        \begin{equation}
            f\circ \phi(x) = f(z) + x_d.
        \end{equation}
Conversely, a point $z\in U$ is said to be a \textit{topological critical point of $f$} if it is not a topological regular point of $f$.
Finally, a point $z\in U$ is said to be a \textit{non-degenerate topological critical point of $f$ of index $i$} if there exist an integer $0\leq i \leq d$ and a homeomorphism $\phi:V_1\to U_2$ between open neighborhoods $V_1$ of $0$ in $\R^d$ and $V_2$ of $U$ in $\R^d$ with $\phi(0)=z$ such that for all $x=(x_1,\dots,x_d)\in V_1$, 
        \begin{equation}
            f\circ \phi(x) = f(z) - \sum_{j=1}^i x_j^2 + \sum_{j=i+1}^d x_j^2.
        \end{equation}
The function $f$ is said to be a \textit{topological Morse function} if all its topological critical points are non-degenerate.

In what follows, we let ``critical points'' refer to the gradient-based definition that we have been using up till now, i.e. those $x\in \R^D\backslash\M$ such that $\nabla d_\M(x) =0$, and ``topological critical points'' to the definition from the previous paragraph.
The first lemma of this section states that non-critical points are also topologically non-critical.
\begin{lemma}\label{lem:non_critical_points_are_topo_regular}
 Let $x\in \R^D \backslash \M$.
 If $x\not \in Z(\M)$, then $x$ is a topological regular point for $d_\M$.
\end{lemma}
\begin{proof}
One can use similar arguments as those in the proof of Lemma 1.4 from \cite{Cheeger_critical_points} or Proposition 1.7 from \cite{GroveCriticalPoints}. 
Briefly summarized, there exists a small open neighborhood $U\subset \R^D$ of $x$, a constant $C>0$ and a unit vector  $X\in \R^D$ such that the integral flow $\Psi_X$ of the constant vector field associated to $X$ on $U$ satisfies 
\begin{equation}
    \label{eq:flow_increases_value}
    d_\M(\Psi_X(z,t_2)) - d_\M(\Psi_X(z,t_1)) \geq C(t_2-t_1)
\end{equation}
 for all $z\in U$ and $t_1,t_2$ close enough to $0$.
Up to a rotation and an affine transformation (to simplify notations), we can assume that $x=0\in\R^D$ and $X = (0,\ldots,0,1)\in\R^D$, and let $\pi :\R^D \rightarrow \R^{D-1} $ be the orthogonal projection on the first $D-1$ coordinates.
Then for every $y\in \R^{D-1}$ and every $r\in \R$ close enough to $0$,  Equation \eqref{eq:flow_increases_value} ensures that there exists a unique $t(y,r)\in \R$ such that $d_\M((y,t(y,r))) = d_\M(x) + r $, and that $(y,r)\mapsto t(y,r)$ is continuous.
This allows us to define the desired local change of coordinates $\phi :(y,r) \mapsto (y,t(y,r))$, whose inverse is $z \mapsto  (\pi(z),d_\M(z) - d_\M(x))$. 
\end{proof}
Note that the lemma does not assume that $\M$ satisfies Properties \Pall.

To handle critical points, we will need the machinery of \textit{Min-type functions} from \cite{gershkovich1997morse}.
Given $x\in \R^D$ and two real-valued functions $f,g$ each defined on some neighborhood of $x$, we write $f \simeq_x g$ if the germs of $f$ and $g$ at $x$ are equal.
Let $U\subset \R^D$ be an open set and $f:U \rightarrow \R$ be $C^k$ for some $k\geq 1$.
We say that $f$ is $C^k$ Min-type at $x\in U$ if there exist real-valued $C^k$ functions $\alpha_1,\ldots,\alpha_n$ defined on some neighborhood of $x$ such that
\[ f \simeq_x \min\{\alpha_1,\ldots, \alpha_n\}.\]
Furthermore, assume that there exists such a family $\{\alpha_1,\ldots,\alpha_n\}$ that also satisfies the following properties:
\begin{enumerate}
    \item $\{\alpha_1,\ldots, \alpha_n\}$ is a representation of $f$ at $x$ of minimal size; in other words, if a family of  $C^k$ functions $\beta_1,\ldots,\beta_m$ is such that $f\simeq_x \min\{\beta_1,\ldots, \beta_m\}$, then $m\geq n$.
    \item The gradients $\nabla_x \alpha_1, \ldots, \nabla_x\alpha_n \in \R^D$ are the vertices of a non-degenerate simplex.
    \item $0\in \Conv(\nabla_x \alpha_1, \ldots, \nabla_x\alpha_n )$.
    \item The restriction of $f$ to the locally defined $C^k$ submanifold $\{\alpha_1 = \ldots = \alpha_n\}$ is a Morse function around $x$.
    \item Any subset of cardinality $n-1$ of $\{\nabla_x \alpha_1, \ldots, \nabla_x\alpha_n\}$ is linearly independent in $\R^D$.
\end{enumerate}
Then $x$ is said to be a \textit{non-degenerate Min-type critical point} of $f$ (notions of degenerate Min-type critical points and Min-type regular points also exist, but do not serve our current purpose).
Theorem 1 from Subsection 3.2 of \cite{gershkovich1997morse} states if $f$ is $C^k$ Min-type (for any $k\geq 1$) at some $x\in \R^D$, and if $x$ is a non-degenerate Min-type critical point of $f$, then it is also a non-degenerate topological critical point of $f$.
Hence, to prove the following proposition, we only have to show that if $\M$ satisfies \Pall\ and $z_0\in Z(\M)$, then the distance function $d_\M$ is $C^1$ Min-type at $z_0$ and satisfies the five conditions listed above.
\begin{lemma}\label{lem:critical_points_are_topo_non_degenerate}
Assume that $\M$ satisfies Properties \Pall, and let $z_0\in Z(\M)$. Then $z_0$ is a non-degenerate topological critical point of $d_\M$.
\end{lemma}
\begin{proof}
Let us first show that $d_\M$ is $C^1$ Min-type at $z_0$. Let $\pi_\M(z_0) =\{x_1,\ldots,x_s\}$; we have shown in Lemma \ref{lem:local_projections_well_defined} that there exists (disjoint) neighborhoods $U$ of $z_0$ in $\R^D$ and $\M_1,\ldots, \M_s$ of $x_1,\ldots,x_s$ in $\M$ and $C^1$ maps $p_i : U\rightarrow \M_i$ (for $i\in [s]$) such that for any $z\in U$, we have 
$\pi_\M(z) \subset \{p_1(z),\ldots,p_s(z)\}$.
This means in particular that 
\[d_\M(z) = \min \{d(z,p_1(z)),\ldots, d(z,p_s(z)) \} \]
for $z\in U$.
It was also shown in the proof of Lemma \ref{lem:core_medial_axis_is_C2} that the functions $d_{\M_i}:z\mapsto d(z,p_i(z)) $ were $C^2$, hence $d_\M$ is $C^2$ Min-type at $z_0$ (in fact, it is $C^k$ Min-type, where $k$ is the regularity of $\M$).

Let us now show that $z_0$ is a non-degenerate Min-type critical point of $d_\M$.
The fact that $d_{\M_1},\ldots,d_{\M_s}$ is a minimal representation of $d_\M$ is stated in Lemma 8 of \cite{song2023generalized}.
As the gradient of $d_{\M_i}$ at $z_0$ is equal to $\frac{z_0 - x_i}{\|z_0 - x_i\|} =(z_0 - x_i)d_\M(z_0)^{-1} $, Conditions 2. and 3. from the definition are guaranteed by \ref{P1}.
So is Condition 5., as $z_0$ belonging to the relative interior of $\Conv(x_1,\ldots,x_s)$ is equivalent to $0$ belonging to the relative interior of $\Conv(x_1-z_0,\ldots,x_s-z_0) = - d_\M(z_0)\Conv(\nabla_{z_0} d_{\M_1},\ldots, \nabla_{z_0} d_{\M_s})$, which is itself equivalent to $0$ not belonging to any strict face of the simplex $\Conv(\nabla_{z_0} d_{\M_1},\ldots, \nabla_{z_0} d_{\M_s} )$, which is finally equivalent to the linear independence of any strict subset of $\{\nabla_{z_0} d_{\M_1},\ldots, \nabla_{z_0} d_{\M_s} \}$.

Furthermore, let us remark that the function $g$ given by the restriction of $d_{\M}$ to the core medial axis $M_c(z_0,\M) = \{d_{\M_1} = \ldots = d_{\M_s} \}$ described in Lemma \ref{lem:core_medial_axis_is_C2} is $C^2$ (it is equal to the $C^2$ function $d_{\M_1}$ on the core medial axis). Besides, $g$ has a nondegenerate Hessian at $z_0$. Indeed, for $z\in M_c(z_0,\M)$ and any $j\in [s]$, the gradient $\nabla g(z)$ is given by the projection on $T_z M_c(z_0,\M)$ of $z-p_j(z)$, that is $\nabla g(z)=\frac{\pi_{T_z M_c(z_0,\M)}(z-p_j(z))}{d_\M(z)}$. 
    As the center $m_z = m(\{p_j(z):\ j\in [s]\})$ of the smallest enclosing ball of the projections $p_1(z),\dots,p_s(z)$ is written as a convex combination of the projections, we also have
    \[ \nabla g(z) = \frac{\pi_{T_z M_c(z_0,\M)}(z-m_z)}{d_\M(z)}.\]
    But if $z$ is close enough to $z_0$, then $m_z$ is the orthogonal projection of $z$ onto the affine space $\mathrm{Aff}(p_1(z),\ldots,p_s(z))$ spanned by the points $p_j(z)$.
    Indeed, this is the case if and only if $z - m_z$ is perpendicular to $\mathrm{Aff}(p_1(z),\ldots,p_s(z))$,
    which is equivalent to $\{p_1(z),\ldots,p_s(z) \} \subset \left (m_z + (z-m_z)^\bot \right) \cap \overline{B}(z,d_\M(z))$,
    which is itself equivalent to the boundary of the smallest enclosing ball of the set $\{p_1(z), \ldots, p_s(z)\}$ (whose center is $m_z$) being equal to its circumsphere.
    But the boundary of the smallest enclosing ball of a set $\{y_1,\ldots,y_s\}$ and its circumsphere coincide if and only if the center of the circumsphere belongs to the convex hull of the set. This property is satisfied for $x_1,\ldots,x_s$ (as $z_0$ is critical), and as $x_1,\ldots,x_s$ form a non-degenerate simplex in $\R^D$, this remains true for $y_1,\ldots,y_s$ close enough to $x_1,\ldots,x_s$, hence for $p_1(z), \ldots, p_s(z)$ with $z$ close enough to $z_0$.   
    As $M_c(z_0,M)$ is defined as the set of points equidistant to all the sets $M_j$s, the vector space spanned by the vectors $\{p_j(z)-p_1(z):\ j\in[s]\}$ is actually the orthogonal of the tangent space  $T_z M_c(z_0,\M)$. Hence, $m_z$ being the orthogonal projection of $z$ onto the affine space spanned by the points $p_j(z)$ means that $z-m_z \in  T_z M_c(z_0,\M)$, and 
    \[ \nabla g(z) =\frac{z-m_z}{d_\M(z)}. \]
According to \ref{P4},  which is the key ingredient here, when $z$ is close enough to $z_0$, the norm of $\nabla g(z)$ satisfies an inequality of the type $\|\nabla g(z)\|\geq c\|z-z_0\|$ (as $\|\nabla g(z)\|$ is also the norm of the generalized gradient of the unrestricted distance function $d_\M$, which means that $z$ is $\|\nabla g(z)\|$-critical). In particular, the Hessian of $g$ at $z_0$ is non degenerate, and Condition 4. is satisfied.


As mentioned before the statement of the lemma, applying Theorem 1 from Subsection 3.2 of \cite{gershkovich1997morse} is now enough to conclude.
\end{proof}

\begin{remark}
The critical points of $d_\M$ can be more finely characterised as follows: assume as in the lemma that $\M$ satisfies Properties \Pall, and let $z_0\in Z(\M)$ with $\pi_\M(z_0) = \{x_1,\ldots,x_s\}$. As shown in the proof, the restriction of $d_\M$ to the core medial axis $M_c(z_0,\M) = \{d_{\M_1} = \ldots = d_{\M_s} \}$ on a neighborhood of $z_0$ is a $C^2$ Morse function, and $z_0$ is a critical point of this function. Let $i(z_0,d_\M|_{M_c(z_0,\M)})$ be its index (with respect to the restricted function).
Then Theorem 1  from Subsection 3.2 of \cite{gershkovich1997morse} also states that the index $i$ of $z_0$ as a topological critical point of the (unrestricted) distance function $d_\M$ is 
\[ i = s-1 +  i(z_0,d_\M|_{M_c(z_0,\M)}).\]
The two integers $(s,i(z_0,d_\M|_{M_c(z_0,\M)})$ provide a more precise description of the geometry of $d_\M$ around $z_0$ than its index alone does.
\end{remark}

We can now prove Theorem \ref{thm:Morse}, which we restate here for the reader's convenience.
\Morsethm*
\begin{proof}
The two previous lemmas have shown that all points $z\in \R^D\backslash \M$ are non-degenerate topological critical points for $d_\M$ if $z\in Z(\M)$, and topological regular points otherwise.
This makes the restriction of $d_\M$ to $\R^D\backslash \M $ a topological Morse function.
As noted in the Introduction, the Isotopy Lemma for $a >0$ was proved in \cite{GroveCriticalPoints} for the distance function to any compact set; it trivially extends to the case $a = 0$ in the case of submanifolds.
The Handle Attachment Lemma proceeds from arguments that are similar to those used for smooth Morse functions --see \cite[Theorem 5]{song2023generalized} for a detailed proof.
\end{proof}

\section{The \texorpdfstring{$C^2$}{C2} stability theorem }\label{sec:C2_stability}

Let as usual $M$ be a $C^k$ compact manifold for some $k\geq 2$.
In this section, we show that the set of all embeddings $i\in \Emb(M,\R^D)$ such that $i(M)$ satisfies Conditions \ref{P1}, \ref{P2}, \ref{P3} and \ref{P4} is open for the Whitney $C^2$-topology, or in other words that the conjunction of the four conditions is stable with respect to small $C^2$ perturbations.
This proves the ``open" part of the Genericity Theorem \ref{thm:generic}.
We also show that for an embedding $i$ that satisfies  \Pall, small $C^2$ perturbations of $i$ leave the critical points of $i(M)$ and their projections stable.
The $C^2$ Stability Theorem below, whose proof makes many of the more technical results of the previous sections necessary, is the combination of these statements.

\stabilitythm*

Before proving \Cref{thm:C2_stab}, we give three (counter)examples showcasing degenerate behaviors when \Pall\ is not satisfied.

\begin{example}\label{ex:P1_P2_P4_not_open}
It is interesting to note that while the conjunction of \ref{P1}, \ref{P2}, \ref{P3} and \ref{P4} defines an open condition on the set of embeddings $\Emb^2(M,\R^D)$, this is e.g. not true of \ref{P1},\ref{P2} and \ref{P4} without \ref{P3}, as showcased by the following example: consider the curve $\M$ illustrated on the left of Figure \ref{fig:P1_P2_P4_not_open}. The right part is parameterized around $x= (0,1)$ as $\theta \mapsto (1+\theta^4) (\cos(\theta),\sin(\theta))$ for $\theta \in ]-\frac{\pi}{4}, \frac{\pi}{4}[$, and there is a critical point at $z_0 = (0,0)$ with $\pi_\M(z_0) = \{x, x_1, x_2\}$. The curve $\M $ satisfies \ref{P1}, \ref{P2} and \ref{P4}, but not \ref{P3} due to its curvature around $x$. Consider now any $C^\infty$ function $f:\R \rightarrow [0,+\infty)$ such that $f(\theta) = \theta^2$ for any $\theta \in ]-\frac{\pi}{8},\frac{\pi}{8}[$, $f(\theta)\leq \theta^2$ for any $\theta\in \R$ and $f(\theta) = 0$ for any $\theta$ with $|\theta|\geq\frac{\pi}{4}$. For $a>0$, consider the modified curve $\M_a$ shown on the right of Figure \ref{fig:P1_P2_P4_not_open} and obtained by leaving $\M$ untouched, except around $x$ where it is parameterized by 
$$\theta \mapsto (1+\theta^4 -af(\theta)) (\cos(\theta),\sin(\theta))$$
for $\theta \in ]-\frac{\pi}{4}, \frac{\pi}{4}[$.  Let $i_0:M\rightarrow \R^2$ be some embedding whose image is $\M$; then for any $k\geq 0$, there is a constant $C_k>0$ and an embedding $i_a:M\rightarrow \R^2$ whose image is  $\M_a$ and at distance less than $C_ka$ in norm $\|-\|_{C^k}$ from $i_0$. Yet for any $a>0$ small enough, there is a critical point $z_a$ of $\M_a$ close to $z_0$ with four projections (which is more than $D+1 = 3$), two of them resulting from a ``splitting" of $x$, hence $\M_a$ does not satisfy \ref{P1}.
\end{example}

\begin{example}\label{ex:P1''_not_satisfied}
Consider (some compact version of) the curve $\M$ shown on the left of Figure \ref{fig:P1''_not_satisfied}. While Conditions \ref{P2}, \ref{P3}, \ref{P4} and the first part of Condition \ref{P1} are satisfied, the second part of Condition \ref{P1} is not, as the critical point $z_0$ belongs to the boundary of the simplex whose vertices are its projections $\pi_\M(z_0) = \{x_1,x_2,x_3\}$. 
A small perturbation results in the curve $\M'$ shown on the right of the figure, where $z_0$ has split into two critical points (one with three projections and one with only two), illustrating how \ref{P1} is necessary to the stability of $Z(\M)$.
\end{example}

\begin{figure}
    \centering
    \includegraphics[width=0.8\textwidth]{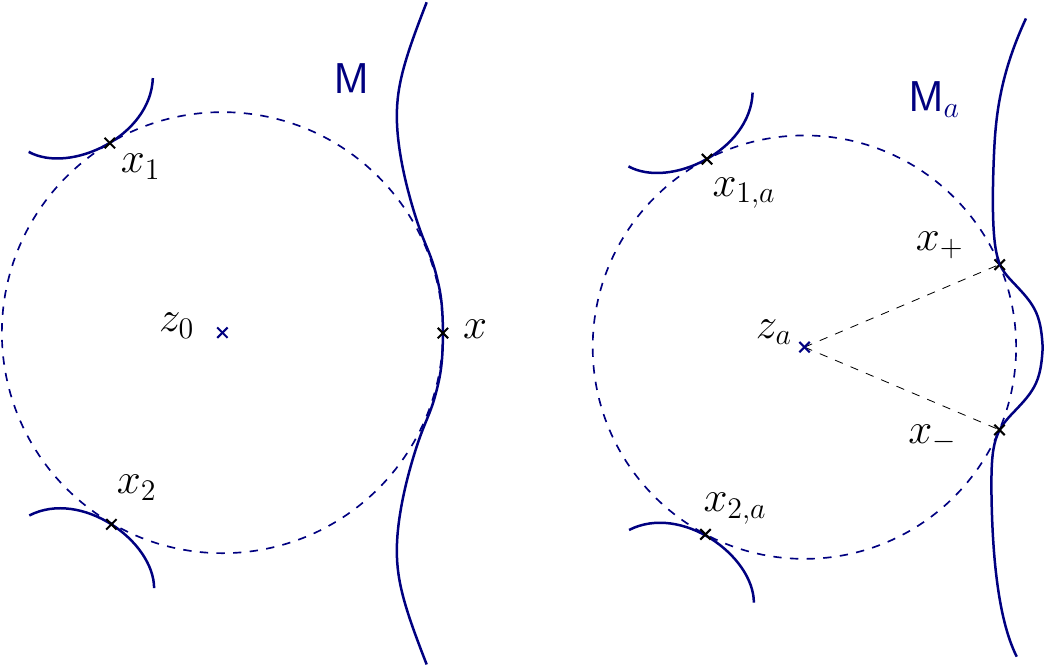}
    \caption{Illustration of Example \ref{ex:P1_P2_P4_not_open}: the curve $\M$ on the left satisfies Conditions \ref{P1}, \ref{P2} and \ref{P4}, but not \ref{P3}. An arbitrarily small modification results in the curve $\M_a$ on the right, which does not fulfill Condition \ref{P1}.}
    \label{fig:P1_P2_P4_not_open}
\end{figure}

\begin{figure}
    \centering
   \includegraphics[width=0.9\textwidth]{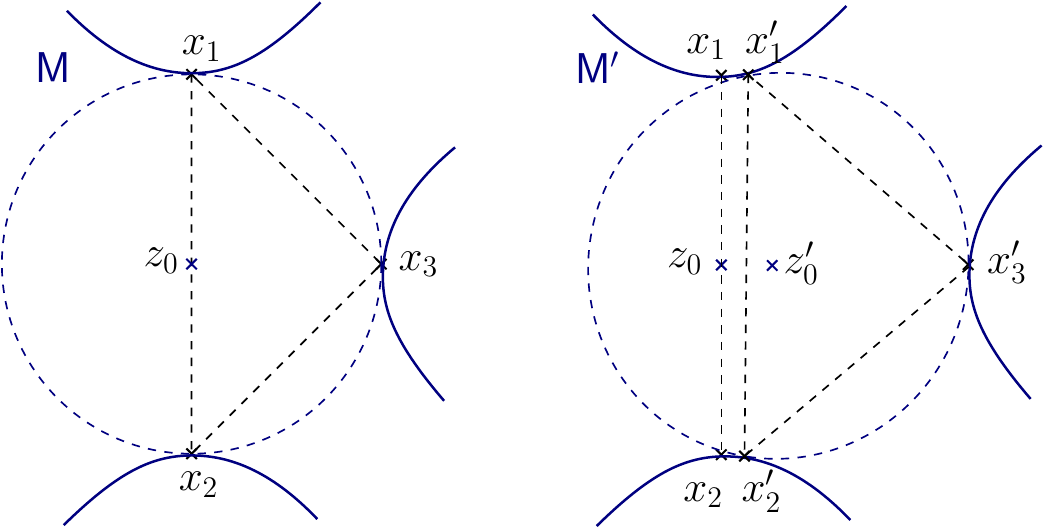}
   \caption{Illustration of Example \ref{ex:P1''_not_satisfied}: the critical point $z_0\in Z(\M)$ splits into two critical points of the perturbed curve $\M'$.}
   \label{fig:P1''_not_satisfied}
\end{figure}


\begin{example}\label{ex:disappearing}
    Critical points may not be stable when Condition \ref{P4} is not satisfied. Consider the curve $\M\subset \R^2$ given in \Cref{ex:counterexample_P4}, that is equal on $\Omega=[-1,1]\times [-2,2]$ to the union of the graph of the functions $f_0:x\mapsto 1+x^3$ and $-f_0:x\mapsto -1-x^3$. For $a>0$, consider the curve $\M_a$ given by the union of the graph of $f_a:x\mapsto 1+ax+x^3$ and $-f_a$ on $\Omega$ (and modified outside $\Omega$ so that $\M_a$ is close to $\M$ in the $C^2$-Whitney topology). 
    The point $(0,0)$ is a critical point of $\M$, but the curve $\M_a$ has no critical points in $B((0,0),1)$. Indeed, if  $z$ was such a point, it would have a strictly greater $x$ coordinate than all of its projections, due to the orthogonality of the projection and to the strict monotonicity of $f_a$; hence $z$ would not belong to the convex hull of its projections, leading to a contradiction.
    Hence, the critical point $(0,0)\in Z(\M)$ disappears in the perturbed submanifold $\M_a$, and the existence of a bijection $\Psi$ between critical points as in \Cref{thm:C2_stab} does not hold on a neighborhood of $\M$.
\end{example}

The remainder of this section is dedicated to proving \Cref{thm:C2_stab}. We decompose the proof in a sequence of claims.

Let $\eps>0$, and let us assume that $\eps$ is small enough that the balls of radius $\eps$ around the critical points $z_0\in Z(\M)$ and around their projections are all pairwise disjoint.

\begin{claim}\label{claim:beginning of the claim journey}
    For any $\delta>0$, we can make a neighborhood $\cU \subset \Emb^2(M,\R^D)$ of $i$ small enough that if $i'\in \cU$,  the critical points of $\M'$ are $\delta$-close to those of $\M$. 
\end{claim}
\begin{proof}
 Indeed, let $\tau(\M)>0$ be the reach of $\M$. If $\|i-i'\|_{C^2}$ is small enough, then $\tau(\M') > \tau(\M)/2$ (see \cite[Theorem 4.19]{federer1959curvature}) and Theorem 3.4 from \cite{Chazal_CS_Lieutier_OGarticle} states that each critical point $z_0' \in Z(\M')$ (which is necessarily at distance at least $\tau(\M')$ from $\M'$) is at distance at most $C_1\sqrt{d_H(\M,\M')}\leq  C_1\sqrt{\|i-i'\|_{C^2}}$ from a $\mu$-critical point of $\M$, with $ \mu = C_2 \sqrt{d_H(\M,\M')}\leq  C_2\sqrt{\|i-i'\|_{C^2}}$ for some constants $C_1, C_2>0$ that depend on $\M$.
As observed at the start of Section \ref{sec:BSP}, there exists $\mu_0>0$ such that any $\mu_0$-critical point of $\M$ is at distance at most $\delta/2$ from $Z(\M)$. Hence if $\|i-i'\|_{C^2}$ is small enough, any critical point of $\M'$ is at distance at most $\delta/2$ of a $\mu_0$-critical point of $\M$, which is itself at distance at most $\delta/2$ from a critical point of $\M$, which proves our claim. Note that the proof does not make use of Properties \Pall.
\end{proof}

Now let $\cU$ be a neighborhood of $i$ small enough that for  $i' \in \cU$ each point $z_0' \in Z(\M')$ is $\eps$-close to some $z_0\in Z(\M)$; then for  $i'\in \cU$ we can define a map $\phi: Z(\M') \rightarrow Z(\M)$ which associates to each critical point $z_0'$ the (unique) critical point $z_0 \in Z(\M)$ at distance less than $\eps$. We make that assumption for the reminder of the proof. We consider a fixed $z_0\in \M$, for which we write $\pi_\M(z_0)=\{x_1,\dots,x_s\}$ for some $2\leq s\leq D+1$, and  a point $z'_0\in Z(\M')$ such that $\phi(z'_0)=z_0$.

\begin{claim}
    Let $0<\delta<\eps$. We can make $\cU$ small enough that 
    each projection of $z_0'$ must be $\delta$-close to one of the $x_j$s.
\end{claim}

\begin{proof}
  Let $i'\in \cU$, $z_0'\in Z(\M')$ and $x'\in \pi_{\M'}(z_0')$. There exists a point $y\in \M$ at distance less than $d_H(\M,\M')$ from $x'$. This point is such that 
\begin{align*}
 \|\phi(z_0')-y\| \leq d_H(\M,\M') + \|\phi(z_0')-x'\| \leq  d_H(\M,\M') + \|\phi(z_0')-z_0'\| + \|z_0'-x'\| 
 \\ =  d_H(\M,\M') + \|\phi(z_0')-z_0'\| +  d_{\M'}(z_0') \leq  2d_H(\M,\M') + 2\|\phi(z_0')-z_0'\|+ d_\M(\phi(z_0')).  
\end{align*}
Hence, $y\in B(\phi(z_0'), 2d_H(\M,\M') + 2\|\phi(z_0')-z_0'\|+ d_\M(\phi(z_0')))\cap \M$. The Hausdorff distance $d_H(\M,\M')$ and the distance between each $z_0' \in Z(\M')$ and the corresponding $\phi(z_0')\in Z(\M)$ can be made arbitrarily small (as shown above) by making $\cU$ small enough; hence, using \Cref{lem:elementary_metric}, we see that $y$ can be made as close as we want to $ \pi_{\M}(\phi(z_0'))$ (this is true even without $\M$ satisfying Conditions \Pall). Hence, $x'$ can also be made as close as we want to $ \pi_{\M}(\phi(z_0'))$. By applying the same reasoning to each of the finitely many critical points of $\M$ (here we use \ref{P2}, and implicitly \ref{P1} to assert that $\pi_{\M}(\phi(z_0'))$ is a finite set), we prove the claim.
  \end{proof}
  
In the remainder of the proof, we assume that $\cU$ is small enough that for any $i'\in \cU$, each projection of each critical point $z_0'$ of $\M'$ is $\eps$-close to some projection of $\phi(z_0')$.

\begin{claim}\label{claim:def_pj}
   We can make $\cU$ small enough for there to be constants $\delta,\delta',\alpha>0$ such that the following holds for any $i'\in \cU$ and any $z_0'\in Z(\M')$: the point $z'_0$   has at most a single projection $x'_j$ that is $\eps$-close to any given projection $x_j$ of $\phi(z_0')$. Furthermore, the sphere $S(z_0',d_{\M'}(z_0'))$ does not osculate $\M'$ at $x'_j$, and there exist local orthogonal projections $p'_j:B(z'_0,\delta')\to \M'_j=\overline B(x'_j,\delta)\cap \M'$ given by \Cref{lem:local_projections_well_defined} that are $\alpha^{-1}$-Lipschitz continuous. In particular, $\M'$ satisfies \ref{P3}. 
\end{claim}

\begin{proof}
    This is a consequence of the $C^2$ stability of the projection proved in Proposition \ref{prop:C2_stability_projection}, and of the crucial hypothesis that $\M$ satisfies \ref{P3}. 
    Indeed, let $z_0\in Z(\M)$, $x_0\in \pi_{\M}(z_0)$ and $m_0\in M$ with $i(m_0) = x_0$. As the sphere $S(z_0,d_\M(z_0))$ is non-osculating at the point $x_j\in \pi_\M(x_j)$, there exist $\alpha>0$ and neighborhoods $U_M\subset M$, $U_{\R^D} \subset \R^D$ and $U_{\Emb}\subset \Emb^2(M,\R^D)$ of $m_0$, $z_0$ and $i$ such that for any $i'\in  U_{\Emb}$, any $z'_0\in U_{\R^D}$ has a unique projection $x'_j=p'_j(z'_0)$ on $\overline{i' (U_M)}$, that $x'_j$ belongs to $i' (U_M)$, and that the sphere $S(z'_0,d_{i'  (U_M)}(z'_0))$ does not osculate $i'(U_M)$ at $x'_j$ with $\alpha$ a degree of non-osculation.  

    Furthermore, if $\M'=i'(M)$, then there exist $\delta,\delta'>0$ such that the function $p'_j:B(z'_0,\delta')\to \M'\cap \overline B(x_j',\delta)=\M'_j$ is the orthogonal projection on $\M'_j$, which is $C^1$ and $\alpha^{-1}$-Lipschitz continuous.
    We can assume that $\delta,\delta'<\eps$.


Let us make $\cU$ small enough that $\cU \subset U_{\Emb}$ and that if $i'\in \cU$, any critical point $z_0'\in Z(\M')$ in $\phi^{-1}(z_0)$ belongs to $U_{\R^D}$ and any projection $x'\in \pi_{\M'}(z_0')$ of $z_0'$ close to $x_j$ is at distance at most $\delta$ from $x_j$. 
Then for such an embedding $i'$, critical point $z_0'$ and projection $x'$, 
we have shown that $x' \in i'(U_M)$. As we know from the conclusions of Proposition \ref{prop:C2_stability_projection} on the $C^2$ stability of the projection that $z_0'$ has a unique projection on $\M'$ in $i'(U_M)$, the projection $x'$ is necessarily the only one to be $\delta$-close to $x_j$ (hence the only one to be $\eps$-close to $x_j$).
We also get that the sphere $S(z_0',d_{i'  (U_M)}(z_0')) = S(z_0',d_{\M}(z_0')) $ does not osculate $i'(U_M)$ at $x'$.
By making $\cU $ small enough for the same reasoning to work for each of the finitely many critical points of $\M$ and their finitely many projections (and by taking the minima of the constants $\delta, \delta'$ and $\alpha$ corresponding to each critical point $z_0$ and each projection $x_j$), we prove the claim.
\end{proof}

For the remainder of the proof, we assume that $\cU$ is small enough for the conclusions of the previous claim to hold. 

\begin{claim}\label{claim:exactly_one_projection}
We can make $\cU$ small enough that 
there is exactly one projection $x'_j\in \pi_{\M'}(z_0')$ that is $\eps$-close to each $x_j$.
\end{claim}

\begin{proof}
Let us consider again $z_0\in Z(\M)$ with $\pi_\M(z_0) = \{x_1,\ldots,x_s\}$.
If $i'\in \cU$ and $z_0' \in \phi^{-1}(z_0) $, we already know that for each $x_j\in \pi_\M(z_0)$ there can be at most one projection $x'_j\in \pi_{\M'}(z_0')$ that is $\eps$-close to $x_j$, and that each projection of $z_0'$ must be  $\eps$-close to one of the $x_j$.

As $\M$ satisfies \ref{P1}, the critical point $z_0$ must belong to the interior of the simplex with vertices $x_1,\ldots,x_s$. In particular,  for any $J\subsetneq [s]$
\[\frac{\|z_0-m(\{x_j:\ j\in J\})\|}{d_\M(z_0)}>0.\]
For $i'\in \cU$ and $z_0' \in \phi^{-1}(z_0) $, let $I(z_0')\subset [s]$ be the set of indices $j$ such that $z_0'$ has a (unique) projection $x_j'$ that is $\eps$-close to $x_j$.
By making $\cU$ small enough, we can  ensure (by continuity) that for any $i'\in \cU$ and any $z_0' \in \phi^{-1}(z_0) $ the following holds: the point $z_0'$ is close enough to $z_0$ and each of its projections $x_j'$ is close enough to $x_j$ that if we let $J\subset I(z_0')$ with $J\subsetneq [s]$, then 
\[\frac{\|z_0'-m(\{x_j':\ j\in J\})\|}{d_{\M'}(z_0')}>0.\]
In particular, as 
\[\nabla d_{\M'}(z_0') = \frac{z_0'-m(\{x_j':\ j\in I(z_0')\})}{d_{\M'}(z_0')},\]
the point $z_0'$ being critical forces $I(z_0')=[s]$. Hence the projections of $z_0'$ are in bijection with those of $z_0$. By applying the same reasoning to each of the finitely $z_0\in Z(\M)$, we prove the claim.
\end{proof}

Again, we assume $\cU$ to be small enough for its conclusions to hold for the remainder of the proof.

\begin{claim}\label{claim:P1}
    We can make $\cU$ small enough that $\M'$ satisfies \ref{P1}.
\end{claim}

\begin{proof}
 This is a simple consequence of the previous claims. As $\M$ satisfies \ref{P1}, we know that for each $z_0\in Z(\M)$, the projections $\pi_\M(z_0)$ are the vertices of a non-degenerate simplex to the relative interior of which $z_0$ belongs.
These properties are open in the coordinates of $z_0$ and its projections (under the constraint that $z_0$ necessarily belongs to the simplex, which is satisfied for any critical point).
We have shown that by making $\cU$ small enough, we can make any critical point $z_0'\in \M'$  arbitrarily close to a corresponding critical point $\phi(z_0') \in \M$ and the projections of $z_0'$  arbitrarily close to the corresponding projections of $\phi(z_0')$ (with which they are in bijection) if $i'\in \cU$.
As there are only finitely many critical points of $\M$ with finitely many projections, this is enough to conclude that any $\M'$ satisfies \ref{P1} for $i'\in \cU$ and $\cU$ small enough, which we assume to be the case for the remainder of the proof.
   \end{proof}
\medskip

For $z_0\in Z(\M)$ with $\pi_\M(z_0) = \{x_1,\ldots,x_s\}$, let $E=\mathrm{Vec}(\{x_j-x_k:\ j,k\in [s]\})$ and  $E'=\mathrm{Vec}(\{x_j'-x_k':\ j,k\in [s]\})$. For $I\subset [s]$, recall the functions $m_I$ and $g_I$ from \Cref{sec:BSP}, and define the functions $m'_I(z)=m(\{p'_j(z):\ j\in I\}$ and $g'_I(z)= \frac{\|z-m'_I(z)\|}{d_{\M'}(z)}$ implicitly associated to some $i'\in \cU$.
They are defined for $z\in B(z'_0,\delta')$, where $\delta'$ is the constant of \Cref{claim:def_pj}. As the projections $p'_j$ are Lipschitz continuous for a uniform constant $\alpha^{-1}$ over $i'\in \cU$ and as the distance between $z_0'$ and the boundary of the simplex with vertices $\pi_{\M'}(z_0')$ can be uniformly controlled over $i'\in \cU$ (as in the proof of the previous claim), \cref{lem:g_I} implies that the functions $g'_I$ are $(1/2)$-H\"older continuous and $g'_{[s]}$ is Lipschitz continuous
on a small ball $B(z'_0,\delta'')$, where $\delta''$, the Lipschitz constant of $g'_{[s]}$ and  the H\"older  constants of the $g'_I$ are uniform over $i'\in \cU$.

Remark also that the various constants in \Cref{lem:mu_critical_implies_perpendicular} depend only on the reach and the diameter of $\M$, on the Lipschitz continuity of the $p_j$s, the H\"older constants of the $g_I$s and the Lipschitz constant of $g_{[s]}$, on the level of non-degeneracy of the simplex with vertices $x_1,\dots,x_s$ (e.g., measured through its determinant) and on the distance between $z_0 $ and the boundary of this simplex.
As the local projections $p'_j$s are all $\alpha^{-1}$-Lipschitz continuous and defined on $B(z'_0,\delta')$ for some $\alpha$ and $\delta'$ uniform over $i'\in\cU$, by making the non-degeneracy of the simplices and the distances of the critical points to their boundaries uniform over $i'\in\cU$ in the conclusion of \Cref{claim:P1}, and using the remark on the H\"older constant of the previous paragraph, the constants of \Cref{lem:mu_critical_implies_perpendicular} can be made uniform over $i'\in \cU$ as well,  thence the following claim holds.

\begin{claim}\label{claim:perpendicular_uniform}
    We can make $\cU$ small enough  that there exist constants $\delta_1,\mu_0,C>0$ such that for all $i'\in \cU$ and $z_0' \in Z(\M')$, if $z\in B(z'_0,\delta_1)$ is $\mu_0$-critical (for $\M'$), then $\|\pi_{E'}(z-z'_0)\|\leq C\|z-z'_0\|^2$.
\end{claim}

We now turn to proving that Condition \ref{P4}, which we have shown in Proposition \ref{prop:big_simplex_is_enough} to be equivalent under our hypotheses to Condition \ref{BSP}, is open.

\begin{claim}\label{claim:BSP}
We can make $\cU$ small enough that there exist constants $\delta, \delta', L>0$ such that  for all $i'\in \cU$, the submanifold $\M'=i'(M)$ satisfies \ref{BSP} with constants $\delta, \delta', L$ for any $z_0'\in Z(\M')$. In particular, $\M'$ satisfies \ref{P4}. 
\end{claim}

The uniformity of $\delta, \delta'$ and $L$ over $i'\in \cU$ is not needed for $\M'$ to satisfy \ref{BSP} (hence \ref{P4}), but it will prove useful for the following claims.

\begin{proof}
Using the same notations as above, for $i'\in \cU$, let $h'\in E'^\perp$ be small enough that $p_j'(z'_0+h')$ is well-defined for $j\in [s]$. Let $\Delta'(h')$ be the simplex with vertices $z'_0+h'$ and $p'_j(z'_0+h')$ for $j\in [s]$. Define $X'=[x'_1-z'_0,\dots,x'_s-z'_0]$, $r'_j(h')=p'_j(z_0'+h')-x_j'-h'$ and $R_j'(h')$ the matrix obtained from $X'$ by replacing the $j$th column of $X'$ by $r'_j(h')$. According to \Cref{lem:formula_simplex_volume}, the volume of $\Delta'(h')$ satisfies
\[ \Vol_s(\Delta'(h'))^2 \geq  \frac{1}{(s!)^2}\sum_{i,j =1}^{s}  \det(R'_{i}(h')^\top R'_{j}(h'))-C_0\|h'\|^3, \]
where $C_0$ depends on $\diam(\M')$ and the Lipschitz constant $\ell$ of the projections $p'_j$ (and $s$, which is the same for any $i'\in \cU$ and $z_0'\in \phi^{-1}(z_0)$). As those quantities are uniformly bounded over $\cU$, the constant $C_0$ can be chosen independent of $i'\in \cU$. 

Consider the matrices $\tilde R_j(X',h)$, obtained by replacing the $j$th column of $X'$ by $r_j(h)$ (similarly defined as $r_j(h)=p_j(z_0+h)-x_j-h$). Let $\delta_0>0$ be such that the $p_j$s are defined and $\ell$-Lipschitz continuous on $B(z_0,2\delta_0)$. For $\|h\|\leq \delta_0$, define
\[ F(X',h)= \sum_{i,j=1}^s \frac{1}{\|h\|^2}\det(\tilde R_j(X',h)^\top \tilde R_i(X',h))\]
and $F(X')=\inf_{h\in E^\perp, \|h\| \leq \delta_0}  F(X',h)$. As $\|r_j(h)\|/\|h\|$ is bounded by $\ell+1$ for $h\in E^\perp$ with $\|h\|\leq\delta_0$, the family of functions $X'\mapsto F(X',h)$ indexed by $h\in E^\perp$ with $\|h\| \leq \delta_0$ is equicontinuous. This implies in particular that the function $X'\mapsto F(X')$ is continuous.  Let $X=[x_1-z_0,\dots,x_s-z_0]$. As the submanifold $\M$ satisfies Condition \ref{P4} (which is equivalent to Condition \ref{BSP}), \Cref{lem:formula_simplex_volume} ensures that $F(X)>C_1$ for $\delta_0$ small enough and for some constant $C_1>0$.  By continuity, as long as $z'_0$ and  the projections $x'_j$ are close enough to $z_0$ and to the projections $x_j$ respectively, then
\[ \sum_{i,j=1}^s \det(\tilde R_j(X',h)^\top \tilde R_i(X',h)) \geq C_1\|h\|^2\]
for $h\in E^\perp$ with $\|h\|\leq \delta_0$.

Let $h'\in E'^\perp$ and let $h= \pi_{E^\perp}(h')$. 
For any $\beta>0$, we can ensure that $\|h'-h\| \leq \beta \|h'\|$ for any such $h'$ by making the  $x'_j$s close enough to the $x_j$s. Then
\begin{align*}
    \|r_j(h)-r'_j(h')\|&\leq \beta\|h'\| + \|p'_j(z'_0+h')-x'_j - (p_j(z_0+h)-x_j)\| \\
    &\leq  \beta\|h'\| + \|\int_0^1 (d(p'_j)_{z'_0+th'}(h')- d(p_j)_{z_0+th}(h))\dd t\|
\end{align*} 
Note that $p'_j(z'_0+th')$ and $p_j(z_0+th)$ belong respectively to $i'(U_M)$ and $i(U_M)$ for some open set $U_M$ of $M$ that can be made arbitrarily small (by making $\U$ smaller). 
We know from \Cref{prop:diff_projection} that the differential $d(p_j)_z$ of the projection can be expressed in terms of the orthogonal projection onto $T_{p_j(z)}\M$ and of the shape operator of $\M$ at $p_j(z)$.
As a result, for any constant $\gamma>0$, if $\|i'-i\|_{C^2}$, $\delta'>0$ and the set $U_M$ are small enough, then for any $z,z'\in B(z_0,\delta')$ the projections $p'_j(z')$ and $p_j(z)$ are the images by $i$ and $i'$ of points $m,m'\in U_M$ that are sufficiently close for us to have  $\op{d(p'_j)_{z'}-d(p_j)_z}\leq \gamma$.
Assuming that this is the case,
\begin{align*}
    \|r_j(h)-r'_j(h')\|    &\leq  \beta\|h'\| + \|\int_0^1 d(p_j)_{z_0+th}(h-h') \dd t\| + \|\int_0^1 (d(p'_j)_{z'_0+th'}- d(p_j)_{z_0+th})(h')\dd t\| \\
    &\leq \beta \|h'\| + \ell \beta \|h'\| + \gamma \|h'\| = (\beta(\ell+1)+\gamma)\|h'\|.
\end{align*} 
As the determinant is locally Lipschitz continuous, it holds that for some constant $C_2>0$ and under the same hypotheses 
\begin{align*}
    |\det(R_j'(h')^\top R_i'(h'))&-\det(\tilde R_j(X',h)^\top \tilde R_i(X',h))| \\
    &\leq C_2(\|r'_j(h')-r_j(h)\| \|r_i(h)\|+\|r'_i(h')-r_i(h)\| \|r'_j(h')\|) \\
    &\leq 2C_2\ell(\beta(\ell+1)+\gamma) \|h'\|^2.
\end{align*}
Hence for any  $h'\in E'^\perp$ with $\|h'\|\leq \delta_0$ and any $i'\in\cU$,
\begin{align*}
    \Vol_s(\Delta'(h'))^2& \geq  \frac{C_1}{(s!)^2}\|h\|^2 - \frac{1}{(s!)^2}\sum_{i,j =1}^{s}  2C_2\ell(\beta(\ell+1)+\gamma) \|h'\|^2 -C_0\|h'\|^3 \\
&\geq \frac{C_1}{2(s!)^2} \|h'\|^2
\end{align*} 
 if we choose $\cU$ small enough so that $\beta$ and $\gamma$ are small enough. We can then make $\cU$ even smaller so that the above considerations hold simultaneously for all the finitely many critical points $z_0\in Z(\M)$. 
\end{proof}

\begin{claim}
 We can make $\cU$ small enough that the associated map $\phi:Z(\M') \rightarrow Z(\M)$ is injective. In particular, $\M'$ satisfies \ref{P2}.
\end{claim}
\begin{proof}
    Interestingly, the proof of this claim once again makes use of Condition \ref{P4} (see also Example \ref{ex:P1''_not_satisfied} regarding the importance of the second part of Condition \ref{P1}).
    Let $z$ and $z'$ be critical points in $Z(\M')$ with $\phi(z)=\phi(z')=z_0$. Recall from \Cref{claim:beginning of the claim journey} that we can make $\|z-z_0\|$ and $\|z'-z_0\|$ arbitrarily small if needed. Write $\pi_{\M'}(z')=\{x'_1,\dots,x'_s\}$ and let $E'=\mathrm{Vec}(\{x'_j-x'_k:\ j,k\in [s]\})$. According to \Cref{claim:perpendicular_uniform}, whenever $\|z-z'\|\leq \|z-z_0\|+\|z'-z_0\|<\delta_1$, as $z$ is a critical point of $\M'$, we have
    \[ \|\pi_{E'}(z-z')\|\leq C\|z-z'\|^2.\]
    Let $p'_j$ be the local projections defined on $B(z',\delta')$, for some $\delta'$ given by \Cref{claim:def_pj}. Note that $z\in B(z',\delta')$ when $\|z-z'\|$ is made small enough by taking $\cU$ small enough. In particular, we have $\pi_{\M'}(z) = \{p_1'(z),\dots,p_s'(z)\}$ thanks to Claims \ref{claim:def_pj} and \ref{claim:exactly_one_projection}.
    Let $h=\pi_{E'^\perp}(z-z')$ and let $\tilde \Delta'(h)$ be the simplex with vertices $p_1'(z'+h),\dots,p_s'(z'+h)$.
    Let $\Delta'(h)$ be the simplex obtained by adding $z'+h$ to the vertices of $\tilde\Delta'(h)$.   
    As $\M'$ satisfies \ref{BSP}, it also holds that if $\|z-z'\|$ is small enough, then
    \[ \Vol_s(\Delta'(h)) \geq L\|h\|,\]
    where we use the uniformity with respect to $i'\in \cU$ of the constants from \Cref{claim:BSP}.
    Let $m'(h)$ be the center of the smallest enclosing ball of the vertices $p_1'(z'+h),\dots,p_s'(z'+h)$.     Remark that, as $g'_{[s]}(z)=0$ and $g'_{[s]}$ is $\ell$-Lipschitz continuous for a constant $\ell$ that is uniform over all $i'\in \cU$,
    \begin{align*}
        \Vol_s(\Delta'(h)) &\leq \frac{1}{s} \Vol_{s-1}(\tilde\Delta'(h))\|z'+h-m'(h)\| \leq\frac{1}{s} (\diam(\M'))^{s-1}d_{\M'}(z'+h) g'_{[s]}(z'+h)\\
        &\leq \frac{2}{s} (\diam(\M'))^{s} (g'_{[s]}(z'+h)-g'_{[s]}(z)) \\
        &\leq \frac{2}{s} (\diam(\M'))^{s} \ell \|\pi_{E'}(z-z')\|\\
        &\leq \frac{2}{s} (\diam(\M'))^{s} \ell C\|z-z'\|^2.
    \end{align*} 
    As $\|h\|\geq \|z-z'\|/2$ as long as $\|z-z'\|$ is small enough, we obtain that
    \[ \frac{L}{2}\|z-z'\| \leq \frac{2}{s} (\diam(\M'))^{s} \ell C\|z-z'\|^2.\]
    When $\|z-z'\|$ is small enough, this is only possible when $z=z'$. We make $\cU$ small enough so that this argument applies around all critical points $z_0\in Z(\M)$ to conclude regarding the injectivity of $\phi$. 
\end{proof}

The last point of the proof of \Cref{thm:C2_stab} is that the map $\phi: Z(\M')\to Z(\M)$ can also be made  surjective. The map $\Psi$ from the statement of the theorem is then defined as $\phi^ {-1}$.
\begin{claim}
 We can make $\cU$ small enough that for any $i'\in \cU$, the associated map $\phi:Z(\M') \rightarrow Z(\M)$ is surjective.
\end{claim}

\begin{proof}
Let $t>0$ be a critical value for the distance function $d_M$, and let $z_1,\ldots,z_l \in Z(\M)$ be all the critical points with associated critical value $t$.
Then the Handle Attachment Lemma from Theorem \ref{thm:Morse} states that for $\eta>0$ small enough, the set $d_\M^{-1}(-\infty,t+\eta]$ has the homotopy type of $d_\M^{-1}(-\infty,t-\eta]$ with  cells (i.e. closed balls) $e_{1,d_1},\ldots,e_{l,d_l}$ attached along their boundary, where $d_j\in\{1,\ldots,D\}$ is the dimension of $e_{j,d_j}$. Conversely, the Isotopy Lemma from the same theorem ensures that there is no change in homotopy type between 
$d_\M^{-1}(-\infty,a]$ and $d_\M^{-1}(-\infty,b]$ if $d_\M$ has no critical value in the interval $[a,b]$.

Consider now the persistence diagram $\mathrm{dgm}(\M)$ of the sublevel filtration of $d_\M$ \cite{edelsbrunner2022computational}. 
As for smooth Morse functions, the critical points of $d_\M$ are precisely in bijection with the non-zero coordinates of the points in the persistence diagram; though the arguments are similar, we provide a proof for completeness.
First, the Isotopy Lemma ensures that to each non-zero coordinate (either birth or death) of a point in the persistence diagram $\mathrm{dgm}(\M)$ corresponds one critical value of $d_\M$. 
Now, let $t$ be as above some critical value with associated critical points  $z_1,\ldots,z_l$ and let $\eta>0$ be small enough that $d_\M^{-1}(-\infty,t+\eta]$ has the homotopy type of \[d_\M^{-1}(-\infty,t-\eta] \cup \left(e_{1,d_1} \cup \ldots \cup e_{l,d_l} \right),\]
where the  cells $e_{1,d_1},\ldots,e_{l,d_l}$ are attached along their boundaries (and the homotopy equivalence is compatible with the inclusion of $d_\M^{-1}(-\infty,t-\eta] $).

We would like to compute the Mayer-Vietoris sequence of 
$d_\M^{-1}(-\infty,t-\eta]$ and $e_{1,d_1} \cup \ldots \cup e_{l,d_l} $, but their interiors do not cover $d_\M^{-1}(-\infty,t-\eta] \cup \left(e_{1,d_1} \cup \ldots \cup e_{l,d_l} \right)$; we simply replace them by homotopically equivalent sets.
Let $A$ be an open ``thickening'' of $d_\M^{-1}(-\infty,t-\eta] $ in $d_\M^{-1}(-\infty,t-\eta] \cup \left(e_{1,d_1} \cup \ldots \cup e_{l,d_l} \right)$, i.e.  the union of $ d_\M^{-1}(-\infty,t-\eta] $ with a small open neighborhood of the union of the boundaries $\partial e_{j,d_j}$ of the cells $e_{j,d_j}$ (so that $d_\M^{-1}(-\infty,t-\eta] \cup \left(e_{1,d_1} \cup \ldots \cup e_{l,d_l} \right) \backslash A$ is the union of $l$ closed disks of dimension $d_1,\ldots,d_l$ respectively).
Similarly, let $B$ be the union of $l$ open disks $D_1,\ldots,D_l$ of dimension $d_1,\ldots,d_l$ respectively with $D_i \subset e_{i,l_i}$, such that $A\cup B = d_\M^{-1}(-\infty,t-\eta] \cup \left(e_{1,d_1} \cup \ldots \cup e_{l,d_l} \right)$, and $A\cap B$ is homeomorphic to $\bigsqcup_{j=1}^l \partial e_{j,d_j}  = \bigsqcup_{j=1}^l S^{d_j-1}$, where $S^d$ is the $d$-dimensional sphere.

We can now consider the Mayer–Vietoris exact sequence with coefficients $\Z_2 $ associated to the pair $\left( A ,B\right)$. 
As $  A\cap B $ is homotopically equivalent to $\bigsqcup_{j=1}^l \partial e_{j,d_j}  = \bigsqcup_{j=1}^l S^{d_j-1}$, the homology group $H_k( A \cap B )$ is isomorphic to $(\Z_2)^{n_{k+1}}$ for $k\geq 1$, where $n_k$ is the number of $k$-dimensional cells among $e_{1,d_1} , \ldots , e_{l,d_l}$.
Furthermore, we have $H_k(B) = H_k(  e_{1,d_1} \cup \ldots \cup e_{l,d_l} ) =0 $ for $k\geq 1$.

Around $k\geq 1$, the Mayer-Vietoris sequence is isomorphic to
\[
\begin{tikzcd}
  { } \arrow[r] &
 H_{k+1}(A) \arrow[r,"i_{k+1}"] & 
  H_{k+1}(A\cup B )  \arrow[r] & \Z_2^{n_{k+1}} \arrow[r] & 
   H_{k}(A) \arrow[r,"i_{k}"] & 
  H_{k}(A\cup B )  \arrow[r]  & { } 
\end{tikzcd}
\]
Let us identify $H_i(A) \simeq H_i(d_\M^{-1}(-\infty,t-\eta] )$ and $H_i(A\cup B) \simeq H_i(d_\M^{-1}(-\infty,t+\eta] )$;  the dimension $D_{k+1,b}$ of the cokernel of $H_{k+1}(d_\M^{-1}(-\infty,t-\eta]) \xrightarrow{i_{k+1}}
  H_{k+1}(d_\M^{-1}(-\infty,t+\eta] )$ is precisely the number of births of intervals  between $t-\eta$ and $t+\eta$ (hence precisely at $t$) in the persistence module of the filtration.
  Similarly, the dimension $D_{k,d}$ of the kernel of $H_{k}(d_\M^{-1}(-\infty,t-\eta]) \xrightarrow{i_{k}}
  H_{k}(d_\M^{-1}(-\infty,t+\eta] )$ is the number of deaths of intervals at $t$ in the persistence module of the filtration.
  Using the exactness of the sequence, one finds that $n_{k+1} = D_{k+1,b} +D_{k,d}$, meaning that each $(k+1)$-cell among $e_{1,d_1} , \ldots , e_{l,d_l}$ corresponds exactly either to the birth of an interval for the homology of degree $k+1$, or to the death of an interval for the homology of degree $k$ (in particular, a $k$-cell and a $(k+1)$-cell cannot ``cancel each other out''). An almost identical reasoning applies for $k=0$.
  Hence we have shown that the non-zero birth and death values of the intervals of the filtration are in bijection with the critical points of $d_\M$.

Now let $\rho>0$ and consider an embedding $i'$ close enough to $i$ that all the previous claims stand and such that  $\M'=i(M)$ is $\rho$-close to $\M$ for the Hausdorff distance. This means that $\|d_{\M'} - d_{\M}\|_\infty \leq \rho$, and the bottleneck stability theorem \cite{cohen2005stability} ensures that the two persistence diagrams $\mathrm{dgm}(\M)$ and $\mathrm{dgm}(\M')$ associated to the sublevel filtrations of $d_{\M'}$ and $d_{\M}$ are $\rho$-close (for the bottleneck distance). When $\rho$ is small enough, this implies that the points of $\mathrm{dgm}(\M')$ have at least as many non-zero coordinates (counted with multiplicity if several points share a coordinate) as those of $\mathrm{dgm}(\M)$, hence that $d_{\M'}$ has at least as many critical points as $d_{\M}$.
As $\phi:Z(\M')\to Z(\M)$ is injective, it must then be bijective, and in particular surjective. This proves the claim, and thence the theorem.
\end{proof}

\section{Density of Conditions \texorpdfstring{\Pall}{(P1-4)}}\label{sec:density}

Let as usual $M$ be a compact abstract manifold of regularity $C^k$ with $k\geq 2$.
In this section, we show that the set of $C^k$ embeddings $i:M \rightarrow \R^D $ such that $i(M)$ satisfies geometric properties \Pall\ is dense in  $\text{Emb}^k(M,\R^D)$ for the Whitney $C^k$-topology, which completes the proof of the Genericity Theorem \ref{thm:generic}.

It would not be too difficult to show that these geometric properties are ``locally generic"  using elementary methods. However, the difficulty resides in controlling $i(M)$ globally to make sure  that  situations contradicting one of the properties \Pall\ do not emerge anywhere.

To do so, we make repeated use of one of the many variants of R. Thom's transversality theorem (see e.g. \cite{golubitsky1974stable} for more on this theorem and its applications), following the examples of Yomdin \cite{yomdin1981local}, Mather \cite{mather1983distance} or Damon and Gasparovic \cite{DamonGasparovic} (among others) who applied similar methods.

Remember that if $X,Y$ and $Z\subset Y$ are manifolds of regularity at least $C^1$, a mapping $\phi : X \rightarrow Y $ is \textit{transverse} to $Z$ if for each $x\in \phi^{-1}(Z)$ we have $\Im(d\phi_x) + T_{\phi(x)}Z = T_{\phi(x)}Y$.
Informally, most versions of Thom's theorem state that for any such $X, Y, Z$, the set of mappings $\phi : X\rightarrow Y $ such that $\phi$ is transverse to $Z$ is dense (or sometimes residual, or even open and dense) in the set of such mappings; in other words, ``most" mappings are transverse to $Z$. The exact statement depends on the exact set of mappings considered, the topology with which we equip it, and whether or not we consider \textit{(multi)jets}, of which we now give a terse definition.

\subsection{A brief introduction to jets and multijets}\label{subsection:introductionJets}
Jets and multijets are useful tools that help us study and control the values and derivatives of maps between manifolds; these, in turn, determine the geometric properties that are of interest to us when considering the special case of embeddings $i:M\rightarrow\R^D$. As we  make extensive use of multijets in the next subsections, we choose to succinctly introduce them here for completeness; the reader already familiar with the definitions can safely skip this subsection. Our presentation mostly follows \cite{golubitsky1974stable} (see also \cite{kolar2013natural} for more details).

Let $X$ and $Y$ be $C^k$ manifolds (in the cases of interest to us, $X$ will be our compact manifold $M$ and $Y$ will be the euclidean space $ \R^D$), and let $x\in X$ and $y\in Y$. 
For $0\leq r\leq k$, let us consider the set of mappings $ C^r(X,Y)_{x,y} :=\{f \in C^r(X,Y) : \; f(x) = y\}$ and let us define the following relation of equivalence on $C^r(X,Y)_{x,y}$:
\begin{quote}
   Let $U$ be a neighborhood of $x$ in $X$ and $V$ be a neighborhood of $y$ in $Y$, and let $\phi : U\rightarrow \R^{\dim X}, x\mapsto 0$ and $\psi : V\rightarrow \R^{\dim Y}, y\mapsto 0$ be charts of $X$ and $Y$ respectively. Then we say that $f,g \in C^r(X,Y)_{x,y}$ have $r$-th order contact at $x$, which we write as ``$f \sim_r g$ at $x$", if the differentials  $d^l(\psi \circ f \circ \phi^{-1})_0$ and $d^l (\psi \circ g \circ \phi^{-1})_0$ in $0$ of order $l = 0,\ldots, r$ coincide.
\end{quote}
It is easy to see that if the property is verified for a choice of charts $\phi, \psi$ around $x$ and $y$, then it is verified for any such pair of charts; hence, the definition is coherent. Note that the equivalence class of $f\in C^r(X,Y)_{x,y}$ depends only on the germ of $f$ at $x$. 
The \textit{$r$-jets} $J^r(X,Y)_{x,y}$ from $x$ to $y$ are defined as the set $C^r(X,Y)_{x,y}$ quotiented by the equivalence relation ``$\sim_r$ at $x$", and the \textit{$r$-jets} from $X$ to $Y$ are defined as
$$J^r(X,Y) : = \bigcup_{x\in X, y\in Y } J^r(X,Y)_{x,y}.$$
We call $x$ the \textit{source} and $y$ the \textit{target}  of an $r$-jet of $J^r(X,Y)_{x,y}$.
Out of convenience, given an $r$-jet $a\in J^r(X,Y)$, we explicitly specify its source $x$ and target $y$ by writing it as
$$a = (x, y , [f]),$$
where $f\in C^r(X,Y)_{x,y}$ is any representative of the equivalence class $a$.

Let $A_n^r$ be the vector space of polynomials in $n$ variables of degree less than or equal to $r$ with $0$ as their constant term, and let $B^r_{n,m} := \bigoplus_{i=1}^m A_n^r$. It is a real vector space of dimension $m\dim_\R A_n^r = m \left( \binom{n+r}{n}-1\right) $.
Given an open set $U\subset \R^n$ and a $C^r$ map $g:U \rightarrow \R$, define $T^r g:U \rightarrow A_n^r$ as the map that associates to $x\in U$ the terms of degree $1$ to $r$ of the Taylor series of $g$ in $x$.
Now let $U\subset \R^n$ be an open set, and consider a $C^r$ map $f = (f_1,\ldots, f_m):U\rightarrow \R^m $. For any $x\in U$, the polynomials $T^r(f)(x):=(T^rf_1(x),\ldots, T^r f_m(x)) \in B^r_{n,m}$ fully describe the first $r$ derivatives $d^rf_x$ of $f$ in $x$ (excluding the ``$0$-th derivative" $f(x)$). Conversely, for any $P\in B^r_{n,m}$ and any $x\in U$, there exists a $C^r$ map $g:U\rightarrow \R^m $ such that $T^r(g)(x) = P$. In other words, there is a canonical bijection between $B^r_{n,m}$ and the possible values of the first $r$ derivatives of a $C^r$ map $\R^n \supset  U \rightarrow \R^m$ at any point $x\in U$.

Given as before two $C^k$ manifolds $X$ and $Y$, this means that the space of $r$-jets $J^r(X,Y)$ (for $0\leq r\leq k$) can naturally be seen as a $C^{k-r}$ manifold  of dimension $\dim X + \dim Y + \dim B^r_{\dim X, \dim Y} = \dim X +  \dim Y \binom{\dim X+r}{\dim X} $. Indeed, a pair of charts $\phi :U \rightarrow \R ^{\dim X}$ and $\psi: V \rightarrow \R ^{\dim Y}$ of $X$ and $Y$ induces a map
\begin{equation*}
    \begin{split}
       \xi_{\phi,\psi} :\;  \{(x,y,[f])\in J^r(X,Y) : \; x\in U, y\in V\} & \longrightarrow \phi(U) \times \psi(V) \times B^r_{\dim X, \dim Y} \\
         a = (x,y,[f]) & \longmapsto (\phi(x), \psi(y), T^r(\psi\circ f\circ \phi^{-1})(\phi(x))).
    \end{split}
\end{equation*}
In other words, given the $r$-jet $a$,  we simply consider the first $r$ derivatives in the source $x$ of $a$ of one of its representatives $f$  expressed in the system of coordinates induced by the charts $\phi$ and $\psi$. It is easy to show that the family of such maps $\{\xi_{\phi,\psi}\}$ to which an atlas $\{\phi\}$ of $X$ and an atlas $\{\psi\}$ of $Y$ give rise does yield an atlas of $J^r(X,Y)$. In what follows, we always consider $J^r(X,Y)$ with this manifold structure and with the associated topology. We will call charts of  $J^r(X,Y)$ induced in such a way by charts of $X$ and $Y$ \textit{adapted charts} (this is not standard terminology).

Any map $f\in C^k(X,Y)$ induces a canonically defined $C^{k-r}$ map
\begin{equation*}
    \begin{split}
       j^r(f) :\; X & \rightarrow J^r(X,Y)\\
        x & \longmapsto (x, f(x), [f]).
    \end{split}
\end{equation*}
called the \textit{$r$-jet of $f$}, which essentially encodes the first $r$-derivatives of $f$ in all $x\in X$.
It can be shown that the sets
\[ \{f\in C^k(X,Y) : \; j^r(f)(X) \subset U \},\]
where $U$ is any open set of $J^r(X,Y)$, form a basis for the Whitney $C^r$-topology (again, see \cite{golubitsky1974stable} for details) which we introduced in the special case where $X$ is a compact manifold and $Y = \R^D$ in Section \ref{sec:notations_definitions}; in fact, this is how the topology is often defined.

The space of jets and the $r$-jet $j^r(f)$ associated to a mapping $f\in C^k(X,Y)$ can help us monitor the $r$ first derivatives of $f$ in each $x\in X$ separately, but the geometric properties that are of interest to us require that we control the image and derivatives of a mapping in several points simultaneously and with respect to each other; we want embeddings $i:M\rightarrow \R^D$ for which ``there is no $m_1,\ldots,m_s\in M$  such that $i(m_1),\ldots,i(m_s)$ and the (higher) derivatives of $i$ in $m_1,\ldots,m_s$ are in a certain configuration". To that end, we need to consider \textit{multijets}, defined as follows :

Consider the generalized diagonal
\[\Delta^{(s)} =:\{(x_1,\ldots,x_s) \in X^s : \; \text{there exist } 1 \leq i< j \leq s \text{ s.t. } x_i = x_j\} \]
and define $X^{(s)}:= X^s \backslash \Delta^{(s)}$, the product $s$ times of $X$ without repetitions. It is an open subset of $X^s$. Let $\alpha : J^r(X,Y) \rightarrow X$ be the source map, i.e.~$\alpha$ maps $(x,y,[f])$ to its source $x$. One can  consider the product map $\alpha^s = \alpha\times \ldots \times \alpha : J^r(X,Y)^s \rightarrow X^s$.
Then 
$$ J_s^r(X,Y) := (\alpha^s)^{-1}\left(X^{(s)}\right)$$
is the \textit{$s$-fold $r$-jet bundle} from $X$ to $Y$; in other words, the space of $s$-tuples of $r$-jets with pairwise distinct sources. A \textit{multijet bundle} is such a space for any $s\geq 2$.
As $\alpha^{s}$ is continuous, the set of multijets $J_s^r(X,Y)$ is an open set of $J^r(X,Y)^s $ and inherits a $C^{k-r}$ manifold structure as a result (where $X$ and $Y$ are as above $C^k$ manifolds for some $k\geq r$). An obvious but important property of $J_s^r(X,Y)$  is that its set of source points $X^{(s)}$ is not compact, even if $X$ is; this will be the cause of some distress later on.
As for $J^r(X,Y)$, we call charts of $J_s^r(X,Y)$ induced by charts of $X$ and $Y$  \textit{adapted charts}.

In line with our previous notations, we usually write an $s$-fold $r$-jet $a\in J_s^r(X,Y)$ as
$$a = (x_1,\ldots, x_s, y_1,\ldots, y_s, [f_1],\ldots,[f_s]),$$
where $[f_i]$ is a class of equivalence for the relation ``$\sim_r$ at $x_i$".
Just as before, a map $f\in C^k(X,Y)$ induces a multijet $C^{k-r}$ map
\begin{equation*}
    \begin{split}
       j^r_s(f) :\; & X^{(s)}  \longrightarrow  J_s^r(X,Y)\\
        & (x_1,\ldots,x_s) \longmapsto (x_1,\ldots, x_s, f(x_1),\ldots, f(x_s), [f],\ldots,[f])
    \end{split}
\end{equation*}
(which is simply the map $j^r(f)\times \ldots \times j^r(f) : X^{(s)}  \rightarrow  J_s^r(X,Y)$ up to our choice of notations).

\subsection{Two transversality theorems}\label{subsection:transversality_thms}
We are now ready to quote the two statements that will be our main tools in the remainder of the section. The first one is a particularly complete variant of Mather's adaptation of Thom's transversality theorem to multijets (in \cite{MATHER1970b}), which we get from Damon's work. Before that, a few definitions are required.

Remember that a closed subset $F$ of a manifold $X$ is \textit{Whitney stratified} if it is a locally finite disjoint union of smooth submanifolds $F_\alpha$ called \textit{strata}, such that if $F_\alpha \cap \overline{F_\beta} \neq \emptyset$ (where $\overline{F_\beta}$ is the closure of $F_\beta$ in $X$), then $F_\alpha \subset \overline{F_\beta}$, and such that each pair of strata satisfies two technical conditions on their tangent spaces and secant lines, called \textit{Whitney's conditions a) and b)}. 
As such a stratification is usually not unique, one often says that $F$ is \textit{Whitney stratifiable} to say that it admits at least one such stratification. The dimension of a Whitney stratified set is the maximum of the dimensions of its strata (it can be shown that it does not depend on the choice of stratification). 

If $f :X \rightarrow Y$ is a map between two $C^k$ manifolds and $F\subset X$, $G\subset Y$ are closed Whitney stratified sets with strata $\{F_\alpha\}$ and $\{G_\beta\}$, we say that $f$ is \textit{transverse to $G$ on $F$} if for each pair of strata $F_\alpha$ and $G_\beta$, the map $f|_{F_\alpha} : F_\alpha \rightarrow Y$ is transverse to $G_\beta$ in the usual sense.

In what follows, we (and the reader) will be shielded from the full complexity of Whitney stratifications; we only need to know that semialgebraic subsets  of the Euclidean space (whose definition we give later in this paper) are Whitney stratifiable, as was shown by Thom in \cite{Thom1966PropritsDL}. To learn more about Whitney stratifications, read e.g., \cite{MatherRefWhitneyStrat} or \cite{gibson1976topological}.

Finally, remember that a subset of a topological space $B$ is \textit{residual} (in $B$) if it is a countable intersection of subsets of $B$ whose interiors are dense in $B$. If $B$ is a \textit{Baire space}, then its residual subsets are dense.

The theorem below is a special case of Corollary 1.11 from \cite{damon1997generic}.
\begin{thm}[Damon]\label{thm:Damon_Transversality}
Let $s\geq 2$ and $r\in \N$ and let $X,Y$ be $C^k$ manifolds for $k\in \mathbb{N}_{\geq r+1}\cup \{\infty\}$. Assume that $X$ is compact, and let $A\subset X^{(s)}$ and  $W \subset J_s^r(X,Y)$ be closed  Whitney stratified sets in their respective ambient space.
Then 
\[ F_{W,A} := \{f \in C^k(X,Y) : \; j^r_s(f) \text{ is transverse to }W\text{ on }A \text{ in } J_s^r(X,Y)\}\]
is a residual subset of $C^k(X,Y) $ for the Whitney $C^{k}$-topology. Furthermore, if $A$ is compact, then $F_{W,A}$ is also open in the Whitney $C^{r+1}$-topology.

\end{thm}

\begin{remark}
    Note that being dense in the Whitney $C^{k}$-topology implies being dense in the Whitney $C^{l}$-topology for all $l\leq k$, and being open in the Whitney $C^{r+1}$-topology implies being open in the Whitney $C^{l}$-topology for all $r+1\leq l \leq k$. 
\end{remark}

\begin{proof}
This statement is an altered, simplified and less general version of \cite[Corollary 1.11]{damon1997generic} for increased readability. Nonetheless, should a distrustful reader want to precisely compare the two to check that \Cref{thm:Damon_Transversality} is indeed a direct consequence of Damon's theorem, let us give a few remarks (using his notations).

Firstly, the distinction that he makes between the weak $C^{r+1}$ (which he calls ``regular") and the Whitney $C^{r+1}$ topology is moot as they coincide in our case (because we assume $X$ to be compact).  We also let the set $\mathscr{H}$ from Damon's statement be the set of all mappings $C^k(X,Y)$ (and the map $\Phi$ is simply the identity). As observed in the \textit{Note} below Damon's definition 1.8 in the same article, it is a standard fact that $C^k(X,Y)$ satisfies what he calls ``having smooth image in $J_s^r(X,Y)$ off $\Delta^{(s)}$"; furthermore (still using his notations),  $\prescript{}{s}{\mathscr{H}}^{(r)} = \prescript{}{s}{ C^k(X,Y)}^{(r)}$ is simply equal to $\{j^r_s(f)(x_1,\ldots, x_s) :\; f\in C^k(X,Y), (x_1,\ldots, x_s) \in X^{(s)}\} = J_s^r(X,Y)$, hence the condition that ``$\Phi$ be transverse to $W$ in $J_s^r(X,Y)$" is trivially satisfied.


Another difference is that Damon's theorem is for $C^\infty$ maps from $X$ to $Y$, but as he points out in the remark before his Definition 1.1, the same statement works for $C^k$ maps for $k\geq r+1$.

Finally, Damon's original result only states that $F_{W,A}$ is dense in $C^\infty$ for the Whitney $C^{r+1}$-topology (rather than for the Whitney $C^k$-topology). However, the proof of his theorem shows that we also have $F_{W,A}$ dense in the Whitney $C^{m}$-topology for any $m\in \mathbb{N}$ such that $r+1\leq m\leq k$. This is our statement if $k$ is finite, and if $k = \infty$, the set $F_{W,A}$ being dense in the Whitney $C^m$-topology for all $m\in \mathbb{N}$ implies that $F_{W,A}$ is dense in the Whitney $C^\infty$-topology by its definition.
\end{proof}

We also need another transversality statement that specifically deals with the case of distance functions in Euclidean spaces. For $M$ a manifold and $f\in C^k(M,\R^D)$, let us define 
\begin{equation*}
    \begin{split}
    \rho_f :  M \times \R^D &\longrightarrow  \R\\
    (m,z) &\longmapsto \|f(m)-z\|^2.
    \end{split}
\end{equation*}
If we see $z\in \R^D$ as a parameter and write $\rho_{f,z} := \rho_f(-,z) \in C^k(M,\R^D) $, we can consider the multijet map of $\rho_f$ with respect to $m\in M$ :
\begin{equation*}
    \begin{split}
(j^r_s)_M\rho_f : M^{(s)}\times \R^D & \longrightarrow J^r_s(M,\R)\\
(m_1,\ldots,m_s,z) &\longmapsto (j^r(\rho_{f,z}) (m_1), \dots, j^r(\rho_{f,z}) (m_s)).
    \end{split}
\end{equation*}
In what follows, we will write an $s$-fold $r$-jet $a\in J^r_s(X,Y)$ in the special case where $X = M$ is a compact manifold and $Y = \R$ as
$$a = (m_1,\ldots,m_s,r_1,\ldots, r_s, [f_1], \ldots, [f_s]),$$
with $m_i \in M$ and $r_i \in \R$.
We will say that a subset $F\subset J_s^r(M,\R)$ is invariant under addition if it is invariant under the action of $\R$ defined as 
$$t + (m_1,\ldots,m_s,r_1,\ldots, r_s, [f_1], \ldots, [f_s]) := (m_1,\ldots,m_s,r_1 +t,\ldots, r_s +t, [f_1+t], \ldots, [f_s+t]) $$
for any $t\in \R$.
Then the following statement is true, which is a variant of a theorem by Looijenga \cite{LooijengaThesis}; this version in particular comes (modulo some alterations and simplifications) from Damon's and  Gasparovic's \cite{DamonGasparovic}.
\begin{thm}[Damon, Gasparovic]\label{thm:distance_function_transversality}
Let $s\geq 2$ and $r\in \N$ and let $M$ be a compact $C^k$ manifold for $k\in \mathbb{N}_{\geq r+1}\cup \{\infty\}$. Let $A\subset M^{(s)}\times \R^D$ and  $W \subset J_s^r(M,\R^D)$ be closed  Whitney stratified sets in their respective ambient space. Furthermore, assume that each stratum of $W$ is invariant under addition.
Then 
\[D_{W,A} := \{f \in C^k(M,\R^D) :\; (j^r_s)_M\rho_f \text{ is transverse to }W\text{ on }A \text{ in } J_s^r(M,\R)\}\]
is a residual subset of $C^k(M,\R^D) $ for the Whitney $C^{k}$-topology. Furthermore, if $A$ is compact, then $D_{W,A}$ is also open in the Whitney $C^{r+1}$-topology.
\end{thm}

\begin{proof}
 As for Theorem \ref{thm:Damon_Transversality}, we briefly explain how the statement is a direct adaptation of Theorem 16.5 in \cite{DamonGasparovic}. 

The remarks we made in the proof of \Cref{thm:Damon_Transversality} regarding the weak and Whitney $C^{r+1}$ topologies, the compacity of $M$, the topologies in which $F_{W,A}$ (here $D_{W,A}$) is dense and the distinctions between $C^\infty$ and $C^k$ maps remain valid.

Two new elements are to be commented. First, we let $E = \emptyset$ and $Y = M^{(s)}$ in Damon and Gasparovic's statement. Secondly, similarly to the ``proof" of Theorem \ref{thm:Damon_Transversality}, we let $\mathscr{H}$ be $C^k(M,\R^D)$. Then it remains to show that, in the words of Damon and Gasparovic, the map
\begin{equation*}
    \begin{split}
\rho :C^k(M,\R^D) & \longrightarrow C^k(M\times \R^D, \R^D)\\
f &\longmapsto  \rho_f: (m,z) \mapsto \|f(m)-z\|^2
    \end{split}
\end{equation*}
(which corresponds to $\Psi$ in their statement) is ``transverse to $W$" (on $X$) (this is defined in Definition 16.4 of \cite{DamonGasparovic}) for any  closed  Whitney stratified set $W\subset  J_s^r(M,\R^D)$ whose strata are invariant under addition. This is shown in the same way as in Looijenga's proof of the original statement (in \cite{LooijengaThesis}; one can also read C. Wall's account \cite{Wall1977GeometricPO}, which is easier to find online).
\end{proof}

\subsection{Geometric properties \texorpdfstring{\Pall\ }{(P1-4)} as transversality conditions}\label{subsection:sets_Wi}

Our plan in this subsection is to translate the geometric properties \Pall\ that we introduced earlier in terms of transversality of the multijet map $j^r_s (i)$  of the embedding $i:M\rightarrow \R^D$ to some well-chosen set $W\subset J_s^r(M, \R^D)$ (respectively the transversality of the multijet map  $(j^r_s)_M\rho_i$ to $W\subset J_s^r(M, \R)$); later, in Subsection \ref{subsection:density}, we use Theorems \ref{thm:Damon_Transversality} and \ref{thm:distance_function_transversality} to show that the transverse embeddings are a dense  set among all embeddings. There will be two subcases: either $W$ represents configurations to be absolutely avoided and the intersection of the image of $j^r_s(i)$ (respectively $(j^r_s)_M\rho_i$) with $W$ will be empty due to a dimension argument, or the intersection will be non-empty but the transversality will constrain the geometry of the configurations in the intersection to be well-behaved.
As in the previous subsections, we write  $a\in J^r_s(M,\R)$  as
$a = (m_1,\ldots,m_s,r_1,\ldots, r_s, [f_1], \ldots, [f_s])$
and  $b \in J^r_s(M,\R^D)$ as $b = (m_1,\ldots,m_s,x_1,\ldots, x_s, [g_1], \ldots, [g_s])$ with $m_i \in M$, $r_i \in \R$  and $x_i \in \R^D$.
As usual, we let $M$ be a compact $C^k$ manifold of dimension $m$, but unlike in the remainder of this article, we have to assume in this subsection that $k\geq 3$ rather than $k\geq 2$.
The reason for this is that if $i\in \Emb^k(M,\R^D)$, then $j^r_s(i)$ is of regularity $C^{k-r}$, hence the map $j^r_s(i)$ being transverse to some set $W\subset J_s^r(M,\R^D)$ only makes sense if $r +1\leq k$. In what follows, some of the sets $W\subset J_s^r(M,\R^D)$ which we will consider can only be defined if $r\geq 2$, hence we need $k\geq r+1 = 3$ to ask that $j^2_s(i)$  be transverse to them.
We will see in the next subsection that this subtlety matters not, and that the density part of Theorem \ref{thm:generic} remains true when $M$ is only $C^2$ (and in fact even only $C^1$, as suggested by Remark \ref{rmk:C1_also_ok}).

\subsubsection{Condition \ref{P1}, part 1}
The definition of the first set $W$ and its properties come from Yomdin's \cite{yomdin1981local}. It allows us to handle a weaker version of Condition \ref{P1} that we call \ref{P0}. We take care of the second part of \ref{P1} later in Subsubsection \ref{subsubsec:P1''}.


\begin{definition}
    For any $s,r\in \mathbb{N}$ with $s\geq 2$ and $r\leq k-1$, we define the following set:
\[ W_0^{(s)} := \{a\in J_s^r(M,\R) :\; r_1 = \ldots = r_s \} .\]
\end{definition}

\begin{remark}
   $W_0^{(s)} $ is clearly a submanifold of $ J_s^r(M,\R)$ of codimension $s-1$ (hence in particular a Whitney stratified set) and invariant under addition.
\end{remark}

\begin{lemma}[Yomdin, \cite{yomdin1981local}]\label{lem:P0}
For any $0\leq r \leq k-1$, if $i\in \Emb^k(M,\R^D)$ is such that $(j^r_s)_M\rho_i : M^{(s)}\times \R^D \rightarrow  J_s^r(M,\R) $ is transverse to $W_0^{(s)} $ for $s=2,\ldots,D+2$, then $\M = i(M)$ satisfies condition
\begin{enumerate}[start=0, label={(P\arabic*)}]
    \item For any  critical point $z\in Z(\M)$, the projections $\pi_\M(z)$ are the vertices of a non-degenerate simplex of $\R^D$. In particular, $z$ has at most  $D+1$ projections.
\label{P0}
    \end{enumerate}
\end{lemma}

\begin{proof}
This is proved by Yomdin in \cite{yomdin1981local}. In fact, he proves it for the points $z$ in the cut locus of $i(M)$, which is in general larger than its medial axis.
\end{proof}

\begin{remark}
It is very easy to find manifolds that do not verify condition \ref{P0}; e.g., the center $z$ of a (Euclidean) sphere $S$ of dimension greater than $1$ is such that $\pi_S(z)$ has infinitely many points. Examples where the points of $\pi_{i(M)}(z)$ are in finite numbers but are not affinely independent are equally easy to construct.
\end{remark}
The rest of the subsection is modelled after these first definition, remark and lemma: we will define sets $W$, show that they are Whitney stratified and of the correct dimension, and show that embeddings that are transverse to them have nice geometric properties (proving that such embeddings are dense will come later).

\subsubsection{Conditions \ref{P2} and \ref{P3}}

\begin{definition}\label{def:W2}
For any $s,r\in \mathbb{N}$ with $s\geq 2$ and $1\leq r\leq k$, we let $W_2^{(s)}$ be the set of multijets $(m_1,\ldots,m_s,x_1,\ldots, x_s, [f_1], \ldots, [f_s]) \in J^r_s(M,\R^D)$ such that there exists $z\in \Conv(x_1,\ldots,x_s)$ with $d(x_1,z) = \ldots = d(x_s,z)$ and such that for all $i\in [s]$, $\Im((df_i)_{m_i} : T_{m_i}M \rightarrow T_{x_i}\R^D) \perp (z-x_i)$.
\end{definition}

\begin{remark}
As for $(m_1,\ldots,m_s,x_1,\ldots, x_s, [f_1], \ldots, [f_s]) \in J^r_s(M,\R^D)$ the image of $(df_i)_{m_i}$ in $T_{x_i}\R^D$ does not depend on the choice of representative $f_i$ for $[f_i] \subset C^k(M,\R^D)_{m_i,x_i}$, the set $W_2^{(s)}$ is well-defined.
\end{remark}

The definition of $W_2^{(s)}$ should be understood as follows: let $i:M\rightarrow \R^D$ be an embedding. If $z_0$ is a critical point of $i(M)$ with $\pi_{i(M)}(z_0) = \{x_1,\ldots, x_s\}$, then 
\[ (m_1,\ldots,m_s,x_1,\ldots, x_s,[i],\ldots,[i])\in W_2^{(s)} \cap \Im(j^r_s(i) : M^{(s)} \rightarrow J^r_s(M,\R^D)) \]
for some $m_1,\ldots,m_s\in M$ with $z_0$ satisfying the constraints for $(m_1,\ldots,m_s,x_1,\ldots, x_s,[i],\ldots,[i])$ in the definition of $W_2^{(s)}$. This is proved further below with Lemma \ref{lem:P2_P3}.
Note that the set of $W_2^{(s)}$ also depends on $r$, though we hide that dependency in our notations.

\begin{lemma}\label{lem:w2_structure}
    For any $s,r\in \mathbb{N}$ with $2\leq s\leq D+1$ and $1\leq r\leq k-1$, the set $W_2^{(s)}$ is a finite union of closed Whitney stratifiable subsets $W_{2,\alpha}^{(s)}$ of  $J^r_s(M,\R^D)$ of codimension at least $ sm$.
    
\end{lemma}



To prove Lemma \ref{lem:w2_structure}, rather than to identify possible strata and to manually show that they verify Whitney's conditions, we will use the fact that semialgebraic subsets of an Euclidean space are Whitney stratifiable (see \cite{Thom1966PropritsDL}).
Remember that a set $ A\subset \R^n$ is called \textit{semialgebraic set} if it can be written as
\[ A = \bigcup_{i=1}^l \bigcap_{j=1}^{k_i}\{ x_1,\ldots, x_n \in \R^n :\; P_{i,j}(x_1,\ldots,x_n) \; \text{op}_{i,j} \; 0\},\]
 where $P_{i,j}$ is a polynomial in $n$ variables and $\text{op}_{i,j} \in \{<, = , >\}$ for all $i = 1,\ldots, l$ and $j=1,\ldots,k_i$. This definition is deceptively simple, as there are some semialgebraic sets that one would not naively expect to be. To learn more on semialgebraic sets, read e.g., \cite{RislerBenedettiAlgAndSemialg}.

\begin{proof}
We fix $s\in \{2,\dots,D+1\}$ in the proof, and write $W_2=W_2^{(s)}$ to alleviate notation. 
First of all, let us define the set
\begin{equation}\label{eq:def_tildeW2}
\begin{split}
    \tilde{W}_2 :=& \{(m_1,\ldots,m_s,x_1,\ldots, x_s, [f_1], \ldots, [f_s]) \in J^r(M,\R^D)^s \;: \; \exists z\in \R^D \text{ s.t. } \\&
    \qquad   \|x_1 -z\|^2 = \ldots = \|x_s-z \|^2, \; \forall i \in [s],\; \Im((df_i)_{m_i} : T_{m_i}M \rightarrow T_{x_i}\R^D) \perp (z-x_i) \ \; \\&
   \qquad \text{and }\exists \lambda_1,\ldots, \lambda_s \in [0,1] \text{ s.t. } \sum_{i=1}^s \lambda_i = 1 \text{ and } z= \sum_{i=1}^s \lambda_i x_i\}.
   \end{split}
\end{equation}

There are two differences between $W_2$ and $\tilde{W}_2$: the first is that it's defined as a subset of $J^r(M,\R^D)^s$ rather than of $J_s^r(M,\R^D) \subset J^r(M,\R^D)^s$, the second is that the constraints in its definitions are the same, but have been rewritten in a more visibly algebraic fashion.
As a result, we have that $W_2 = \tilde{W}_2 \cap J_s^r(M,\R^D) $. We will use $\tilde{W}_2$ to circumvent the lack of compactness of $M^{(s)}$.

\paragraph{Closedness.} 
Let us first show that $\tilde{W}_2$ is closed. Let $(a_n)_{n\in \mathbb{N}} \subset \tilde{W}_2$ be such that $\lim_{n\rightarrow \infty} a_n = a\in J^r(M,\R^D)^s $, and let us write $a_n = (m_{1,n}, \ldots, m_{s,n}, x_{1,n}, \ldots, x_{s,n}, [f_{1,n}], \ldots, [f_{s,n}])$. We want to show that $a\in \tilde{W}_2$.

For each $n\in \mathbb{N}$, let $z_n\in \R^D$ be a point such that $ d(x_{1,n},z_n) = \ldots = d(x_{s,n},z_n)$, $ \Im((df_{i,n})_{m_{i,n}} : T_{m_{i,n}}M \rightarrow T_{x_{i,n}}\R^D) \perp (z_n-x_{i,n})$ for $i\in [s]$ and $z_n\in \Conv (x_{1,n},\ldots,x_{s,n}) $ (in fact, such a point $z_n$ is unique). Write $z_n = \sum_{i=1}^s \lambda_{i,n} x_{i,n} $ for some  $\lambda_{i,n} \in [0,1]$.

We can find a subsequence for which $(\lambda_{1,n},\ldots,\lambda_{s,n})$ converges to $(\lambda_1,\ldots, \lambda_s)$, hence for which $z_n$ converges to $z:= \sum_{i=1}^s \lambda_{i} x_{i} \in \R^D$.
Then by definition $z \in \Conv (x_{1},\ldots,x_{n}) $ and by continuity $ d(x_{1},z) = \ldots = d(x_{s},z)$. Furthermore, choose charts of $M$ around $m_1,\ldots, m_s$, and consider the induced adapted chart of $J^r(M,\R^D)^s $ around $a$.
In the coordinates given by this chart, the derivative $(df_{i,n})_{m_{i,n}} : T_{m_{i,n}}M \rightarrow T_{x_{i,n}}\R^D$ (for each $i=1,\ldots,s$) can be represented by $D$ degree $1$ homogeneous polynomials in $m$ variables, or equivalently by a $D\times m$ matrix $A_{i,n}$.
The condition that $\Im((df_{i,n})_{m_{i,n}}) \perp (z_n-x_{i,n})$ can be written as $(A_{i,n})^\top (z_n-x_{i,n}) = 0$ (as soon as $n$ is large enough for $m_{i,n}$ to belong to the domain of the chart around $m_i$), which is a closed constraint, hence $\Im((df_{i})_{m_{i}} : T_{m_{i}}M \rightarrow T_{x_{i}}\R^D) \perp (z-x_{i})$ and $z$ satisfies all three conditions in the definition of $\tilde{W}_2$. This shows that $a\in \tilde{W}_2$, and thus that $\tilde{W}_2$ is closed in $J^r(M,\R^D)^s$. As $W_2 = \tilde{W}_2 \cap J^r_s(M,\R^D)$, we get that $W_2$ is also closed in $J^r_s(M,\R^D)$.


\paragraph{Finite union of closed Whitney stratified sets.}

A preliminary remark: as was mentioned earlier, it was shown by Thom in \cite{Thom1966PropritsDL} that semialgebraic subsets of Euclidean space are Whitney stratifiable. As $\tilde{W}_2$ is semialgebraic in the coordinates of any adapted chart of $J^r(M,\R^D)^s$ (see below), its intersection with the domain of any such chart is Whitney stratifiable. We are somewhat convinced that this should easily imply that $\tilde{W}_2$ is globally Whitney stratifiable (on the other hand, we cannot in general expect $\tilde{W}_2$ to be globally semialgebraic), but we could not find any such statement in the literature. Consequently, we have to consider each chart separately,\footnote{In fact, we think that a variant of the argument exposed below could be used to show that $\tilde{W}_2$ as a whole is Whitney stratifiable, though it is not required for our purposes.} as follows.

For any  $(m_1,\ldots,m_s,x_1,\ldots, x_s, [f_1], \ldots, [f_s]) \in J^r(M,\R^D)^s$, we can find charts $\phi_i : U_i \rightarrow \R^m$ of $M$ with domains $U_i \ni m_i$ for $i=1,\ldots,s$, and from them get an adapted chart
\[\phi : U\longrightarrow \phi_1(U_1)\times \ldots \times \phi_s(U_s) \times (\R^D)^s \times (B^r_{m,D})^s \]
with domain $ U := \{(m_1,\ldots,m_s,x_1,\ldots, x_s, [f_1], \ldots, [f_s]) \in J^r(M,\R^D)^s :\ m_i \in U_i \text{ for }i\in [s]\} $ (we do not need charts around the targets $x_i\in \R^D$, as $\R^D$ is already an Euclidean space).
For any such chart $U$, let us choose $\eps_U>0$ small enough that for $i\in [s]$ we have
\[ \overline{B}(\phi_i(m_i), \eps_U) \subset \phi_i(U_i) \subset \R^m.\]
Using the compactness of $M^s$, we can find a finite collection of such charts $(\phi_\alpha, U_\alpha)_{\alpha \in A}$ (where $A$ is a finite collection of indices) such that 
\[\bigcup_{\alpha \in A} \phi_{1,\alpha}^{-1}(\overline{B}(\phi_{1,\alpha}(m_{1,\alpha}), \eps_{U_\alpha})) \times \ldots \times \phi_{s,\alpha}^{-1}(\overline{B}(\phi_{s,\alpha}(m_{s,\alpha}), \eps_{U_\alpha})) = M^s,\]
which is equivalent to the charts $(\phi_\alpha, U_\alpha)_{\alpha \in A}$ corestricted to 
\[   \phi_{1,\alpha}^{-1}(\overline{B}(\phi_{1,\alpha}(m_{1,\alpha}), \eps_{U_\alpha})) \times \ldots \times \phi_{s,\alpha}^{-1}(\overline{B}(\phi_{s,\alpha}(m_{s,\alpha}), \eps_{U_\alpha})) \times (\R^D)^s \times (B^r_{m,D})^s\]
covering  $J^r(M,\R^D)^s$.

In the coordinates of any such chart, the set $\tilde{W}_2$ (intersected with the domain of the chart) is semialgebraic. Indeed, a consequence of the Tarski-Seidenberg theorem (see e.g.~\cite[Theorem 2.3.4]{RislerBenedettiAlgAndSemialg}) is that any set defined as 
\[\{X  \in \R^p :\ \exists Z  \in \R^m \text{ s.t. } (P_1(X,Z) \text{ op}_1\; 0) \land \ldots \land (P_q(X,Z) \text{ op}_q \; 0)\}\]
is semialgebraic, where $P_i$ is a polynomial and $\text{op}_i \in \{>,\geq, =\}$ for $i=1,\ldots, q$ and  $\land$ is the $\texttt{AND}$ logical operator (in fact, more general formula featuring $\texttt{OR}$ operators, negations and parentheses are allowed, but this is enough for our purposes).
In the coordinates of a chart $(\phi_\alpha, U_\alpha)$, the set $\tilde{W}_2$ is of this form. 
Indeed, its definition (see \eqref{eq:def_tildeW2}) is a combination of:
\begin{itemize}
    \item A $\exists$ operation for $z,\lambda_1,\ldots, \lambda_s$.
    \item Algebraic constraints on $z$ and the targets $x_i$,
    \item A constraint that the image of the derivatives $(df_i)_{m_i}$ 
    must be perpendicular to $z-x_i$. As noted earlier, this is equivalent to the transpose $(A_i)^\top$ of the matrix representing $(df_i)_{m_i}$ in the coordinates of the charts verifying $(A_i)^\top (z-x_i) = 0$, which is an algebraic constraint on the coordinates representing $(df_i)_{m_i}$, $z$ and $x_i$.
    \item Some more algebraic constraints on the $\lambda_i$, $x_i$ and $z$.
\end{itemize}
Furthermore, the closed ball $\overline{B}(\phi_{s,\alpha}(m_{s,\alpha}), \eps_{U_\alpha})$ is also semialgebraic in the coordinate of the chart $(\phi_\alpha, U_\alpha)$, as  it is defined using a  polynomial inequality.
Consequently, the intersection 
\[ \tilde{W}_{2,\alpha} := \tilde{W}_2 \cap (\phi_{1,\alpha}^{-1}(\overline{B}(\phi_{1,\alpha}(m_{1,\alpha}), \eps_{U_\alpha})) \times \ldots \times \phi_{s,\alpha}^{-1}(\overline{B}(\phi_{s,\alpha}(m_{s,\alpha}), \eps_{U_\alpha})) \times (\R^D)^s \times (B^r_{m,D})^s)\]
is semialgebraic in the coordinates of $(\phi_\alpha, U_\alpha)$, and as a result it is Whitney stratifiable in $J^r(M,\R^D)^s$ (the image of a Whitney stratifiable set by a diffeomorphism is also Whitney stratifiable). It is also closed. 

As Whitney's conditions a) and b) are purely local, it is trivial to check that the intersection of a Whitney stratified set with an open set of its ambient manifold is still Whitney stratified.
As $J^r_s(M,\R^D)$ is an open set of $J^r(M,\R^D)^s$, the intersection 
$$W_{2,\alpha}:=\tilde{W}_{2,\alpha} \cap J^r_s(M,\R^D)$$
is therefore still Whitney stratifiable (and closed) in $J^r_s(M,\R^D)$. We also have
$$W_2 = \tilde{W}_2 \cap J^r_s(M,\R^D) = \left( \bigcup_{\alpha\in A}  \tilde{W}_{2,\alpha}  \right) \cap J^r_s(M,\R^D)  =   \bigcup_{\alpha\in A} W_{2,\alpha}. $$

\paragraph{Dimension.}

We only have to show that the Whitney stratifiable sets $W_{2,\alpha}$ are of codimension $sm$ in $J^r_s(M,\R^D)$ to conclude. To do so, consider the following auxiliary set:
\begin{equation}\label{eq:def_W2'}
\begin{split}
   W'_2 := &\{((m_1,\ldots,m_s,x_1,\ldots, x_s, [f_1], \ldots, [f_s]),z) \in J^r(M,\R^D)^s\times \R^D : \;\\
   &\qquad   \|x_1 -z\|^2 = \ldots = \|x_s-z \|^2, \; \forall i\in [s],\; \Im((df_i)_{m_i} : T_{m_i}M \rightarrow T_{x_i}\R^D) \perp (z-x_i), \\&
   \qquad \text{rank}((df_i)_{m_i} ) = m,\; \exists \lambda_1,\ldots, \lambda_s \in ]0,1[ \text{ s.t. } \sum_{i=1}^s \lambda_i = 1,\; z= \sum_{i=1}^s \lambda_i x_i \\
   &\qquad  \text{and }\dim \text{Aff}(x_1,\ldots,x_s) = s-1\}.
   \end{split}
\end{equation}
Note the four differences compared to $\tilde{W}_2$: the point $z$ is among the coordinates rather than the object of a $\exists$ statement, we ask that the points $x_i$ be affinely independent and that the derivatives $(df_i)_{m_i}$ be of full rank, and the coefficients $\lambda_i$ cannot be trivial.
We claim that $W_2'$ is a (non-closed) submanifold of $J^r(M,\R^D)^s\times \R^D $ of codimension $D+sm$, hence of dimension $\dim J^r(M,\R^D)^s - sm$.


Assume for now that the claim is true, and consider the projection 
\[\pi: J^r(M,\R^D)^s\times \R^D \rightarrow J^r(M,\R^D)^s.\]
Restrict it to $U_\alpha\times \R^D $, where $U_\alpha$ is the domain of one of the charts defined in the previous section of the proof. In the coordinates induced by this chart, the set $W_2'$ is semialgebraic\footnote{In this paragraph, we are guilty of a small abuse of notations: when we say that $W_2'$ is semialgebraic in the coordinates of the chart, we mean that its intersection with the domain of the chart is semialgebraic in the coordinates (and same for $\pi(W_2')$ and $\overline{\pi(W_2')}$).} (for the same reasons as $\tilde{W}_2$), and the map $\pi$ is semialgebraic in the coordinates, as it is simply a projection from one vector space to another (see e.g., \cite[Section 2.3]{RislerBenedettiAlgAndSemialg}). Thus $\pi(W_2')$ is also semialgebraic in the coordinates of the chart,\footnote{In fact, the projection $\pi$ restricted to $W_2'$ is a diffeomorphism onto its image, but it is not needed for the proof.} and so is its closure $\overline{\pi(W_2')}$  (see e.g., \cite[Proposition 1.12] {Coste2002ANIT}).  Furthermore, it is easy to see that the closure $\overline{\pi(W_2')}$ of $\pi(W_2')$ in $J^r(M,\R^D)^s$ is equal to $\tilde{W}_2$.
Hence (see \cite[2.5.4]{RislerBenedettiAlgAndSemialg})
$$\dim \tilde{W}_2 = \dim(\overline{\pi(W_2')}) = \dim(\pi(W_2')) \leq \dim W_2' = \dim J^r(M,\R^D)^s - sm, $$
as the same inequalities are true when considering the intersections of those sets with the domain $U_\alpha$ (respectively $U_\alpha \times \R^D$) of any of the charts $(\phi_\alpha, U_\alpha)_\alpha$.
As $W_2$ is the intersection of $\tilde{W}_2$ with the open set $J^r_s(M,\R^D)$ of $J^r(M,\R^D)^s$, we have $\dim W_2 \leq \dim \tilde{W}_2 \leq \dim J^r(M,\R^D)^s - sm = \dim J^r_s(M,\R^D) - sm$, and each of the closed subsets $W_{2,\alpha}$ of $W_2$ is of codimension at least $sm$ in $J^r_s(M,\R^D) $.

The only thing left is to prove the claim that $W_2'$ is a submanifold of $J^r(M,\R^D)^s\times \R^D $ of codimension $D+sm$. This is a local property; in what follows, we implicitly place ourselves in the coordinates of an adapted chart $(\phi, U)$, without explicitly saying that we are considering the intersection of $W_2'$ with $U\times \R^D$ or constantly specifying "in the coordinates of...".

Let $(a,z) $ be a point of $W_2'$, where $a\in J^r(M,\R^D)^s$.  The conditions that $\dim \text{Aff}(x_1,\ldots,x_s) = s-1$ and that $  \text{rank}((df_i)_{m_i} ) = m \text{ for } i=1\in[s]$ are open, and the conditions that $\|x_1 -z\|^2 = \ldots = \|x_s-z \|^2$ and that there exist $\lambda_1,\ldots, \lambda_s \in ]0,1[$ such that $\sum_{i=1}^s \lambda_i = 1 $ and $z= \sum_{i=1}^s \lambda_i x_i$ mean that $z$ must be the (unique) center of the circumsphere of $x_1,\ldots, x_s$, and that this center must belong to the relative interior of the convex hull of $x_1,\ldots,x_s$, which is another open condition on $x_1,\ldots,x_s$. Hence we can find a small neighborhood $V$ of $a$ such that these three open conditions are satisfied on the neighborhood $V\times \R^D$ of $(a,z)$. For any point $(b,z)= ((m_1,\ldots,m_s,x_1,\ldots, x_s, [f_1], \ldots, [f_s]),z)$ in $V\times \R^D$, the point $z$ must again be the center of the circumsphere of $x_1,\ldots, x_s$, and for such a point $(b,z)\in V\times \R^D$, the coordinates of $z$ are a (smooth) function $z(x_1,\ldots,x_s)$ of the coordinates of $x_1,\ldots, x_s$ (see \cite{coxeter1930circumradius}).

 Identify the first order derivatives at $m_i$ of the  functions $f_i$ with $D \times m$ matrices $A_i$ and let $d^{\geq 2}f_i$ represent all derivatives at $m_i$ of higher order, and consider the function $F :V  \longrightarrow (\R^m)^s$ given by
 \begin{equation}\label{eq:def_F}
F(m_1,\ldots,m_s,x_1,\ldots, x_s, A_1, d^{\geq 2}f_1, \ldots, A_s, d^{\geq 2}f_s) =(A_1^\top (z-x_1), \ldots, A_s^\top (z -x_s)),
\end{equation}
where $z = z(x_1,\ldots, x_s) $. 
Then $F$ is a submersion, as the derivative of $(A_i, x_1,\ldots, x_s) \mapsto A_i^\top (z(x_1,\ldots,x_s)-x_i) $ with respect to $A_i$ is surjective (since $z-x_i \neq 0$). Consequently, the set 
\begin{align*}
F^{-1}(\{0\}) =& \{(m_1,\ldots,m_s,x_1,\ldots, x_s, A_1, d^{\geq 2}f_1, \ldots, A_s, d^{\geq 2}f_s)\in V\; :\\
&\qquad \forall i\in [s],\ A_i^\top (z(x_1,\ldots, x_s) -x_i) = 0
\}
\end{align*}
is a submanifold of codimension $sm$ in $J^r(M,\R^D)^s$, and 
$W_2'\cap (V\times \R^D) $
is simply the graph of the function
\begin{equation*}
    \begin{split}
F^{-1}(\{0\}) & \longrightarrow \R^D\\
(m_1,\ldots,m_s,x_1,\ldots, x_s, [f_1], \ldots, [f_s]) &\longmapsto  z(x_1,\ldots, x_s).
    \end{split}
\end{equation*}
Thus it is a submanifold of codimension $sm+D$ in $J^r(M,\R^D)^s\times \R^D$, as claimed.
\end{proof}

\begin{remark}
    We are quite convinced that $W_2^{(s)}$ is in fact a Whitney stratifiable subset of $J^r_s(M,\R^D)$  (rather than a union) - we know how to prove it in the case where $M$ is smooth rather than just $C^k$ using Nash-Tognoli's theorem, but it should be easy to prove in general using some theorem  along the lines of ``any subset of a manifold $X$ that is semialgebraic in the coordinates of a well-chosen atlas of $X$ is Whitney stratifiable"; however, we are no specialists of semialgebraic sets and could not find such a statement in the literature.

\end{remark}

Lemma \ref{lem:w2_structure} gives the existence of a decomposition $W_2^{(s)} = \cup_{\alpha \in A} W_{2,\alpha}^{(s)}$ of $W_2^{(s)}$ into a finite union of closed Whitney stratifiable sets of codimension at least $sm$, but this decomposition is not unique.
We choose for a given $M$ and for each  $s,r\in \mathbb{N}$ with $2\leq s\leq D+1$ and $1\leq r\leq k-1$ such a decomposition $\{W_{2,\alpha}^{(s)}\}_{\alpha \in A}$, and we simply refer to it as ``the decomposition from Lemma \ref{lem:w2_structure}" or ``the sets $\{W_{2,\alpha}^{(s)}\}_{\alpha \in A}$ from  Lemma \ref{lem:w2_structure}"  in what follows.

We now need another set that is very similar to $W_2^{(s)}$, except that it will allow us to select embeddings $i\in \Emb^k(M,\R^D)$ such that  for each critical point $z_0\in Z(i(M))$ and each $x\in \pi_{i(M)}(z_0)$, the sphere $S(z_0,d_{i(M)}(z_0))$ does not osculate $i(M)$ at $x$.
\begin{definition}
For any $s,r\in \mathbb{N}$ with $s\geq 2$ and $2\leq r\leq k$, we let $W_3^{(s)}$ be the set of multijets $(m_1,\ldots,m_s,x_1,\ldots, x_s, [f_1], \ldots, [f_s]) \in J^r_s(M,\R^D)$ that belong to $W_2^{(s)}$ and such that there exists $i\in [s]$  for which the quadratic form $v\mapsto \|d(f_i\circ \phi)_0(v)\|^2 - \dotp{z-x_i,d^2(f_i\circ \phi)_0(v,v)}$ is degenerate for any chart $\phi : (U,0) \rightarrow (M, m_i)$.
\end{definition}

\begin{remark}
The set $W_3^{(s)}$ is well-defined. Indeed, the condition on the quadratic form only depends on the derivatives of order less than $2$ of the functions $f_i$; hence it is independent of the choice of representative $f_i\in[f_i] \subset C^k(M,\R^D)_{m_i,x_i}$. Furthermore, as 
$\dotp{z-x_i,d^2(f_i\circ \phi)_0(v,v)} = \dotp{z-x_i,d^2(f_i)_{m_i}( d\phi_0(v),d\phi_0(v))}$ (because $\Im((df_i)_{m_i}) \perp (z-x_i)$),
it does not depend on the choice of the chart $\phi$ either - if it is true for one chart, it is true for all.
\end{remark}

\begin{lemma}\label{lem:w3_structure}
    For any $s,r\in \mathbb{N}$ with $2\leq s\leq D+1$ and $2\leq r\leq k-1$, the set $W_3^{(s)}$ is a finite union of closed Whitney stratifiable subsets $W_{3,\beta}^{(s)}$ of  $J^r_s(M,\R^D)$ of codimension at least $ sm + 1$.
\end{lemma}

\begin{proof}
The proof is almost identical to that of Lemma \ref{lem:w2_structure}; we omit most details and only comment on the few meaningful differences. As in the previous proof, we fix $s\in \{2,\dots,D+1\}$ and write $W_3=W_3^{(s)}$ and $W_2=W_2^{(s)}$. 
The set $W_3$ can be written as $W_3=\bigcup_{i=1}^s W_3^i$, where $W_3^i$ is the set of multijets $(m_1,\ldots,m_s,x_1,\ldots, x_s, [f_1], \ldots, [f_s])\in W_2$ for which the quadratic form  $v\mapsto \|d(f_i\circ \phi)_0(v)\|^2 - \dotp{z-x_i,d^2(f_i\circ \phi)_0(v,v)}$ is degenerate for any chart $\phi : (U,0) \rightarrow (M, m_i)$.
Our goal is to conclude using the same arguments as in the proof of Lemma \ref{lem:w2_structure} that each $W_3^i$ is a finite union of closed Whitney stratifiable subsets of  $J^r_s(M,\R^D)$ of codimension at least $ sm + 1$, hence that $W_3$ can also be written as such a union.

The condition for a point  $(m_1,\ldots,m_s,x_1,\ldots, x_s, [f_1], \ldots, [f_s]) \in J^r_s(M,\R^D)$ that $ v\mapsto \|d(f_i\circ \phi)_0(v)\|^2 - \dotp{z-x_i,d^2(f_i\circ \phi)_0(v,v)}$  be degenerate is equivalent to a closed condition on the first and  second derivative of the function $[f_i]$ (that does not depend on the choice of the representative $f_i$), and  this condition is algebraic in adapted charts of $J^r_s(M,\R^D)$; indeed, in such a system of coordinates, it is equivalent to asking that the matrix $A_i^\top A_i-B_i$ has zero determinant, where $A_i$ represents the first differential of $f_i$ and $B_i$ represents the bilinear form that applies the second differential of $f_i$ to a pair of vectors and returns the scalar product of the result with $z-x_i$ (both being represented in the chosen system of coordinates).

It remains to show that the added condition does increase the codimension by one. To that end, we would like the condition on the determinant of the matrix to be transverse to our other conditions, but sadly it is not even transverse to the ambient space: the determinant is not a submersion from the space of $m\times m$ matrices to $\R$ (in particular the set of all singular matrices is semialgebraic, but is not a manifold). However, this minor difficulty can be avoided. As the differential of the determinant at a matrix $K$ is equal to $d(\det)_K : H\mapsto \mathrm{tr}(\mathrm{adj}(K)H)$, where $\mathrm{adj}(K)$ is the adjugate of $K$ and $\mathrm{tr}$ is the trace, the determinant is indeed a submersion at any matrix $K$ of rank at least $m-1$.

In the proof of Lemma \ref{lem:w2_structure}, we defined an auxiliary set $W_2' \subset J^r(M,\R^D)^s\times \R^D $ such that $\overline{\pi(W_2')}\cap J_s^r(M,\R^D) = W_2$ and such that $\dim W_2 \leq \dim W_2'$, see \eqref{eq:def_W2'}. Similarly, we  define the set $(W^i_3)'$ of pairs $((m_1,\ldots,m_s,x_1,\ldots, x_s, [f_1], \ldots, [f_s]),z) \in W_2'$ such that the form $v\mapsto \|d(f_i\circ \phi)_0(v)\|^2 - \dotp{z-x_i,d^2(f_i\circ \phi)_0(v,v)}$ is of rank $m-1$.
We claim that the set $ (W^i_3)' $ is a submanifold of codimension  $sm+D+1$ in $J^r(M,\R^D)^s\times \R^D$ (compared to the codimension $sm+D$ of $W_2'$). If this claim is true, the rest of the proof is the same as that of Lemma \ref{lem:w2_structure} (with the additional fact that the matrices of rank $m-1$ are dense in the set of all degenerate matrices).

 To prove the claim, note that the only difference with $W_2'$ is the condition that  the form $v\mapsto \|d(f_i\circ \phi)_0(v)\|^2 - \dotp{z-x_i,d^2(f_i\circ \phi)_0(v,v)}$ be of rank $m-1$. 
This is the intersection of the condition that it be degenerate (as in the definition of $W_3$), or equivalently that the determinant of the associated matrix be $0$, and the condition that it be of rank at least $m-1$, which is an open condition.
Similarly to the proof of Lemma \ref{lem:w2_structure}, any point in $(W_3^i)'$ admits an open neighborhood $V \times \R^D \subset J^r(M,\R^D)^s\times \R^D$  such that the intersection of the set  $(W^i_3)'$ with this neighborhood is the graph by the function $(m_1,\ldots,m_s,x_1,\ldots, x_s, [f_1], \ldots, [f_s])\mapsto z(x_1,\ldots, x_s)$ of the intersection $F^{-1}(\{0\}) \cap G_i^{-1}(\{0\})$  of the zero sets of two maps. Here, the map $F$ is defined in \eqref{eq:def_F} and the map $G_i:V\to \R$ is given by
\begin{equation}
    G_i(m_1,\ldots,m_s,x_1,\ldots, x_s, A_1,d^{ 2}f_1, \ldots, A_s, d^{ 2}f_s, d^{\geq 3}f_1,\ldots,d^{\geq 3}f_s) =  \det(A_j^\top A_j -B_j)
\end{equation}
where 
the $D\times m$ matrix $A_i$ represents the first differential of $f_i$ at $m_i$ and the  $ m\times m $ matrix $B_i$ represents the bilinear form that applies the second differential $(d^{ 2}f_i)_{m_i}$ of $f_i$ (which we simply write $d^{ 2}f_i$ here) to a pair of vectors and returns the scalar product of the result with $z-x_i$.
We already know that $F$ is a submersion; so is $G_j$, because $V$ is chosen (among other open conditions) so that the matrix $A_i^\top A_i -B_i$ be of rank at least $m-1$ for any point in $V$. Furthermore, it is easy to show that the conditions ``$F=0$" and ``$G_j = 0$" are transverse to each other, essentially because ``$F=0$" does not constrain $(d^2f_i)_{m_i}$ in any way. Hence  $F^{-1}(\{0\}) \cap G_j^{-1}(\{0\}) \subset V$ is a submanifold of codimension $sm+1$ in $J^r(M,\R^D)^s$, and $(W^i_3)'\cap (V\times \R^D)$ is a submanifold of codimension $sm+D+1$ in $J^r(M,\R^D)^s\times \R^D$. This proves the claim, and thence the lemma.
\end{proof}

As for Lemma \ref{lem:w2_structure}, we choose  for a given $M$ and for each $s,r\in \mathbb{N}$ with $2\leq s\leq D+1$ and $2\leq r\leq k-1$ a decomposition $W_3^{(s)} = \cup_{\beta \in B} W_{3,\beta}^{(s)}$ of $W_3^{(s)}$ into a finite union of closed Whitney stratifiable sets of codimension at least $sm +1$ whose existence is guaranteed by Lemma \ref{lem:w3_structure}, and we simply refer to it as ``the decomposition from Lemma \ref{lem:w3_structure}" or ``the sets $\{W_{3,\beta}^{(s)}\}_{\beta \in B}$ from  Lemma \ref{lem:w3_structure}"  in what follows.

\begin{lemma}\label{lem:P2_P3}
For any $2\leq r \leq k-1$, if $i\in \Emb^k(M,\R^D)$ is such that $\M=i(M)$ satisfies condition \ref{P0} and  that for all $s=2,\dots,D+1$, $ j^r_s(i) : M^{(s)} \rightarrow  J_s^r(M,\R^D) $ is transverse to the sets $\{W_{2,\alpha}^{(s)}\}_{\alpha \in A}$ from Lemma \ref{lem:w2_structure} and to the sets  $\{W_{3,\beta}^{(s)}\}_{\beta \in B}$ from Lemma \ref{lem:w3_structure}, then $\M =i(M)$ satisfies conditions 
\begin{enumerate}[start=2, label={(P\arabic*)}]
    \item The set $Z(\M)$ is finite. 
    \item  For every $z_0\in Z(\M)$ and every $x_0\in \pi_\M(z_0)$, the sphere $S(z_0,d_\M(z_0))$ is non-osculating $\M$ at $x_0$.
    \end{enumerate}
\end{lemma}
As will become apparent in the proof of Lemma \ref{lem:P2_P3}, condition \ref{P3} is a consequence of the transversality to the sets $\{W_{3,\beta}^{(s)}\}_{\beta \in B}$ only. On the other hand, we need transversality to both collections of sets to get condition \ref{P2}: transversality to the sets $\{W_{2,\alpha}^{(s)}\}_{\alpha \in A}$ is required so that the dimension of the set of critical points $Z(i(M))$ be $0$, and transversality to the sets   $\{W_{3,\beta}^{(s)}\}_{\beta \in B}$ so that $Z(i(M))$  have no accumulation point.

\begin{proof}

Fix $2\leq r\leq k -1$ and $i\in \Emb^k(M,\R^D)$ as in the statement.
As we assume that $i(M)$ satisfies condition \ref{P0}, we know that any critical point $z$ of $i(M)$  is such that the set $\pi_{i(M)}(z)$ of its projections on $i(M)$ is of cardinality at most $D+1$.
Let us write as $Z(i(M))_s$ the points $z\in Z(i(M))$ such that $\pi_{i(M)}(z)$  is of cardinality exactly $2\leq s\leq D+1$.

Let $z\in Z(i(M))_s$  and let us write $\pi_{i(M)}(z)= \{x_1,\ldots, x_s\}$. We also call $m_1,\ldots,m_s\in M$ the points such that $i(m_j) = x_j$.
By definition $d(x_1,z) = \ldots = d(x_s,z)$, and  $T_{x_j}i(M) = \Im(di_{m_j}:T_{m_j}M \rightarrow T_{x_j}\R^D)$ must be perpendicular to $z-x_j$ for $j\in [s]$ for $x_j$ to belong to $\pi_{i(M)}(z)$.
Furthermore, the point $z$ must belong to the convex hull of $x_1,\ldots,x_s$, as it is equal to $m(\{x_1,\dots,x_s\})$, the center of the smallest enclosing ball of $\{x_1,\dots,x_s\}$ (see Section 2.1 of \cite{Chazal_CS_Lieutier_OGarticle}).

This shows that $j^r_s(i)(m_1,\ldots, m_s) = (m_1,\ldots,m_s,x_1,\ldots,x_s,[i],\ldots,[i])$ and $z$ satisfy the conditions in the definition of $W_2^{(s)}$, hence that $j^r_s(i)(m_1,\ldots, m_s)$ belongs to $W_2^{(s)}\cap \Im(j^r_s(i)) \subset J^r_s(M,\R^D)$. In other words, each point $z$ of $Z(i(M))_s$ corresponds to a point $j^r_s(i)(m_1,\ldots, m_s)$ in the intersection $W_2^{(s)}\cap \Im(j^r_s(i)) $. Note that the converse is not true: there can be points $(m_1,\ldots,m_s,x_1,\ldots,x_s,[i],\ldots,[i]) \in \Im(j^r_s(i))$ and $z\in\R^D$ that satisfy the conditions of $W_2^{(s)}$, yet such that $z$ is not a critical point of $i(M)$ (simply because the points $x_j$ are not the closest points to $z$ in $i(M)$).

Furthermore, for any point $(m_1,\ldots,m_s,x_1,\ldots,x_s,[f_1],\ldots,[f_s]) \in J^r_s(M,\R^D) $, there can be at most a single point $z$ for which the conditions of $W_2^{(s)}$ are satisfied, this point $z$ being equal to the center of the smallest enclosing ball of $\{x_1,\dots,x_s\}$ (see Section 2.1 of \cite{Chazal_CS_Lieutier_OGarticle}). 
Consequently, there is an injection $\psi : Z(i(M))_s \xhookrightarrow{} W_2^{(s)}\cap \Im(j^r_s(i)) $. 

Consider $z$ a point of $Z(i(M))_s$ with $\pi_{i(M)}(z) = \{x_1,\ldots, x_s\}$ and 
\[\psi(z) = (m_1,\ldots,m_s,x_1,\ldots,x_s,[i],\ldots,[i])\]
the associated point of  $ W_2^{(s)}\cap \Im(j^r_s(i))$. Then by definition of $W_3^{(s)}\subset W_2^{(s)}$, $\psi(z) \in W_3^{(s)}$ if and only if there exists $j\in [s]$ such that for any chart $\phi : (U,0) \rightarrow (M, m_j)$ the form  $v\mapsto \|d(i\circ \phi)_0(v)\|^2 - \dotp{z-x_j,d^2(i\circ \phi)_0(v,v)}$ is degenerate. For a fixed $j\in [s]$, this condition is equivalent to the sphere $S(z,d_{i(M)}(z))$ osculating $i(M)$ at $x_j$, as shown in the Osculation Characterization Lemma \ref{lem:characterization_osculation}.
In other words, we only need $ W_3^{(s)}\cap \Im(j^r_s(i))$  to be empty for $s=2,\ldots, D+1$ for the embedding $i$ to satisfy condition \ref{P3}.

But the strata of the Whitney stratified sets $ W_{3,\beta}^{(s)}$ have codimension at least $sm+1$ (see Lemma \ref{lem:w3_structure}), and the map  $ j^r_s(i) : M^{(s)} \rightarrow  J_s^r(M,\R^D) $ intersects them transversally by hypothesis.
As $M^{(s)}$ is of dimension $sm$, the image $\Im(j^r_s(i))$ can only have empty intersections with the strata of each $ W_{3,\beta}^{(s)}$, hence with $W_3^{(s)}$ as a whole. We have shown that $i$ verifies condition \ref{P3}.

Similarly, the strata of the Whitney stratified sets $ W_{2,\alpha}^{(s)}$ have codimension at least $sm$, hence the image $\Im(j^r_s(i))$ can only have a non-empty intersection with the strata of codimension $sm$ of  $ W_{2,\alpha}^{(s)}$, and that intersection must be of dimension $0$ and be discrete. The same holds for $ W_2^{(s)}\cap \Im(j^r_s(i))$.

Let us assume by contradiction that $i$ does not satisfy condition \ref{P2}; there must be an index $2\leq s\leq D+1$ such that $Z(i(M))_s$ contains infinitely many points. As each critical point $z$ and its projections $\pi_{i(M)}(z)$ belong to the compact set $\Conv (i(M))$, we can find a sequence of critical points $(z_n)_{n\in \mathbb{N}}$ with projections $\pi_{i(M)}(z_n) = \{x_{n,1},\ldots,x_{n,s}\}$, whose preimages by $i$ we write $\{m_{n,1},\ldots,m_{n,s}\} \subset M$,  such that $\lim_{n\rightarrow \infty} z_n = z$, $\lim_{n\rightarrow \infty} m_{n,j} =  m_j$ and $\lim_{n\rightarrow \infty} x_{n,j} =  x_j$ for some points $z \in \R^D$, $m_1,\ldots,m_s \in M$ and $x_1,\ldots,x_s \in i(M)$ (with of course $i(m_j) = x_j$).

If the points $m_j$ (or equivalently the points $x_j$) are pairwise distinct, the sequence of points $\psi(z_n) = (m_{n,1},\ldots,m_{n,s},x_{n,1},\ldots,x_{n,s},[i],\ldots, [i]) \in W_2^{(s)}\cap \Im(j^r_s(i)) $ converges to 
\[(m_{1},\ldots,m_{s},x_{1},\ldots,x_{s},[i],\ldots, [i]) \in J^r_s(M,\R^D),\] which does belong to $W_2^{(s)}\cap \Im(j^r_s(i))$ as well (as $W_2^{(s)}\cap \Im(j^r_s(i))$ is closed in $J^r_s(M,\R^D)$). Hence we have found an infinite sequence with an accumulation point in $W_2^{(s)}\cap \Im(j^r_s(i))$, contradicting its discrete nature - this is impossible.

Let us now consider the case where the points $m_j$ are not pairwise distinct. As $i(M)$ has positive reach, there cannot be critical points arbitrarily close to it; consequently $z$ does not belong to $i(M)$. Furthermore, as the norm $\|\nabla d_{i(M)} \|$ of the generalized gradient is lower semicontinuous, we have $\nabla d_{i(M)}(z) = 0$, hence $z$ is a critical point. By continuity, the limit points $x_j$ are also such that $d(x_j,z) = d(i(M),z)$, hence $\{x_1,\ldots,x_s\} \subset \pi_{i(M)} (z)$.
As the points $m_j$ are not pairwise distinct, we can assume (up to permuting some indices) that $m_1 = m_2$, hence that $\lim_{n\rightarrow \infty} x_{n,1} = \lim_{n\rightarrow \infty} x_{n,2} = x_1$. This means that for $n$ large enough, there are points $z_n$ arbitrarily close to a critical point $z$ such that they have two distinct projections $x_{n,1}$ and $x_{n,2}$ themselves arbitrarily close to a projection $x_1$ of $z$. As we have shown earlier that $i$ satisfies condition \ref{P3}, we know that the sphere $S(z,d_{i(M)}(z))$ does not osculate $i(M)$ at any of the $x_j$: this yields a contradiction with the local unicity of the projection from Lemma \ref{lem:lipschitz_projection}.
Hence $i$ satisfies condition \ref{P2} as well, as claimed.
\end{proof}
It might have been possible to handle all cases $s\in \{2,\ldots, D+1\}$ at once, and in particular the case of the sequences of projections $x_{n,1}$ and $x_{n,2}$ having the same limit, using some compactification of $J^r_{D+1}(M,\R^D)$ such as the one recently defined in \cite{letendre2023multijet}. As there is no version of Thom and Mather's transversality theorem adapted to such a compactification yet, we favoured a  more elementary approach.

\begin{remark}
A tube in $\R^3$ is a simple example (that can easily be compactified) of a manifold that satisfies neither \ref{P2} nor \ref{P3}.
\end{remark}

\subsubsection{Condition \ref{P1}, part 2}\label{subsubsec:P1''}

We will now handle the second part of Condition \ref{P1}, i.e.~the fact that each critical point $z\in Z(i(M))$ must belong to the relative interior of the convex hull of its projections on $i(M)$, in order to avoid the situation illustrated earlier in Example \ref{ex:P1''_not_satisfied}.

\begin{definition}\label{def:W1}
For any $s,r\in \mathbb{N}$ with $s\geq 3$ and $1\leq r\leq k $, we let $W_1^{(s)}$ be the set of multijets $(m_1,\ldots,m_s,x_1,\ldots, x_s, [f_1], \ldots, [f_s]) \in J^r_s(M,\R^D)$ that belong to $W_2^{(s)}$ and such that the (unique) point $z\in \R^D$ satisfying the conditions in the definition of $W_2^{(s)}$ belongs to $\Conv(x_1,\dots,x_{s-1})$.
\end{definition}

\begin{lemma}\label{lem:w1_structure}
    For any $s,r\in \mathbb{N}$ with $3 \leq s\leq D+1$ and $1\leq r\leq k-1$, the set $W_1^{(s)}$ is a finite union of closed Whitney stratifiable subsets $W_{1,\gamma}^{(s)}$ of  $J^r_s(M,\R^D)$ of codimension at least $ sm + 1$.
\end{lemma}

\begin{proof}

The proof is yet again almost identical to that of Lemma \ref{lem:w2_structure}; we omit most details and only comment on the few meaningful differences. As in the previous proofs, we fix $s\in \{3,\dots,D+1\}$ and write $W_1=W_1^{(s)}$. 
The definition of $W_1$ is the same as that of $W_2$, except that we ask that  $z\in \Conv (x_1,\ldots,x_{s-1}) $ rather than $z\in \Conv (x_1,\ldots,x_{s-1},x_{s}) $. This simply replaces one closed semialgebraic condition by another; the only non-trivial point is to show that this change does increase the codimension by one compared to $W_2$.

In the proof of Lemma \ref{lem:w2_structure}, we defined an auxiliary set $W_2' \subset J^r(M,\R^D)^s\times \R^D $ such that $\overline{\pi(W_2')}\cap J_s^r(M,\R^D) = W_2$ and such that $\dim W_2 \leq \dim W_2'$, see \eqref{eq:def_W2'}. Similarly, we can define the set
\begin{equation}\label{eq:def_W1'}
\begin{split}
   W'_1 := &\{((m_1,\ldots,m_s,x_1,\ldots, x_s, [f_1], \ldots, [f_s]),z) \in J^r(M,\R^D)^s\times \R^D : \;\\
   &\qquad   \|x_1 -z\|^2 = \ldots = \|x_s-z \|^2, \; \forall i\in [s],\; \Im((df_i)_{m_i} : T_{m_i}M \rightarrow T_{x_i}\R^D) \perp (z-x_i), \\&
   \qquad \text{rank}((df_i)_{m_i} ) = m,\; \exists \lambda_1,\ldots, \lambda_{s-1} \in ]0,1[ \text{ s.t. } \sum_{i=1}^{s-1} \lambda_i = 1,\; z= \sum_{i=1}^{s-1} \lambda_i x_i \\
   &\qquad  \text{and }\dim \text{Aff}(x_1,\ldots,x_s) = s-1\}.
   \end{split}
\end{equation}
We claim that the set $ W_1' $ is a submanifold of codimension  $sm+D+1$ in $J^r(M,\R^D)^s\times \R^D$ (compared to the codimension $sm+D$ of $W_2'$). Once this has been shown, the rest of the proof is the same as that of Lemma \ref{lem:w2_structure}.
This is a local property; in what follows, we implicitly place ourselves in the coordinates of an adapted chart $(\phi, U)$ of $ J_s^r(M,\R^D) $, without explicitly saying that we are considering the intersection of $W_1'$ with $U\times \R^D$ or constantly specifying "in the coordinates of...".

Let $(a,z) $ be a point of $W_1'$, with $a\in J^r(M,\R^D)^s$. As in the proof of Lemma \ref{lem:w2_structure}, the conditions that $\dim \text{Aff}(x_1,\ldots,x_s) = s-1$, that $  \text{rank}((df_i)_{m_i} ) = m \text{ for } i=1,\ldots,s$ and that the center of the circumsphere of $x_1,\ldots,x_{s-1}$ belongs to the relative interior of $\Conv(x_1,\ldots,x_{s-1})$ are open, and remain true in a small neighborhood $V$ of $a$, hence in the neighborhood $V\times \R^D$ around $(a,z)$. 
Likewise, for $(b,z)= ((m_1,\ldots,m_s,x_1,\ldots, x_s, [f_1], \ldots, [f_s]),z)$ in $V\times \R^D$, the conditions that $\|x_1 -z\|^2 = \ldots = \|x_{s-1}-z \|^2$ and that there exist $\lambda_1,\ldots, \lambda_{s-1} \in ]0,1[$ such that $\sum_{i=1}^{s-1} \lambda_i = 1 $ and $z= \sum_{i=1}^{s-1} \lambda_i x_i$ means that the coordinates of $z$ are a (smooth) function $z(x_1,\ldots,x_{s-1})$ of the coordinates of $x_1,\ldots, x_{s-1}$, as $z$ must be the center of the circumsphere of $x_1,\ldots, x_{s-1}$.

Similarly to the proofs of Lemmas \ref{lem:w2_structure} and \ref{lem:w2_structure}, the set $W_1'\cap (V \times \R^D) \subset J^r(M,\R^D)^s\times \R^D$ is the graph by the function $(m_1,\ldots,m_s,x_1,\ldots, x_s, [f_1], \ldots, [f_s])\mapsto z(x_1,\ldots, x_{s-1})$ of the intersection $F^{-1}(\{0\}) \cap H^{-1}(\{0\})$  of the zero sets of two maps. Here, the map $F$ is defined in \eqref{eq:def_F} and the map $H:V\to \R$ is given by
\begin{equation}
    H(m_1,\ldots,m_s,x_1,\ldots, x_s,  d^{\geq 1}f_1, \ldots, d^{\geq 1}f_s) =\|z-x_s\|^2 - \|z-x_1\|^2,
\end{equation}
where $z = z(x_1, \ldots,x_{s-1})$ and $d^{\geq 1}f_i$ represents the differentials of $f_i$ of order greater than or equal to $1$  at $m_i$.

We already know that $F$ is a submersion and so is $G$ (in particular because $z\neq x_s$). The conditions ``$F=0$" and ``$H=0$" are transverse to each other essentially because ``$H=0$" does not constrain the differentials $(df_i)_{m_i}$ (represented by the matrices $A_i$) in any way.
 Hence  $F^{-1}(\{0\}) \cap H^{-1}(\{0\}) \subset V$ is a submanifold of codimension $sm+1$ in $J^r(M,\R^D)^s$, and $W_1'\cap (V\times \R^D)$ is a submanifold of codimension $sm+D+1$ in $J^r(M,\R^D)^s\times \R^D$. This proves the claim, and thence the lemma.
\end{proof}

As for Lemmas \ref{lem:w2_structure} and \ref{lem:w3_structure}, we choose  for a given $M$ and for each $s,r\in \mathbb{N}$ with $3\leq s \leq D+1$ and $1\leq r\leq k-1$ a decomposition $W_1^{(s)} = \cup_{\gamma \in C} W_{1,\gamma}^{(s)} $ of $W_1^{(s)} $ into a finite union of closed Whitney stratifiable sets of codimension at least $sm +1$ whose existence is guaranteed by Lemma \ref{lem:w1_structure}, and we simply refer to it as ``the decomposition from Lemma \ref{lem:w1_structure}" or ``the sets $\{W_{1,\gamma}^{(s)} \}_{\gamma \in C}$ from  Lemma \ref{lem:w1_structure}"  in what follows.

\begin{lemma}\label{lem:P1}
For any $1\leq r \leq k-1$, if $i\in \Emb^k(M,\R^D)$ is such that $\M = i(M)$ satisfies condition \ref{P0} and that for all $s=3,\ldots, D+1$, $ j^r_s(i) : M^{(s)} \rightarrow  J_s^r(M,\R^D) $ is transverse to the sets $\{W_{1,\gamma}^{(s)} \}_{\gamma \in C}$ from Lemma \ref{lem:w1_structure} , then $\M$ satisfies condition 
\begin{enumerate}[start=1, label={(P\arabic*)}]
\item For every $z_0\in Z(\M)$, the projections $\pi_\M(z_0)$ are the vertices of a non-degenerate simplex of $\R^D$. In particular, $z_0$ has at most  $D+1$ projections. Moreover, the point $z_0$ belongs to the relative interior of the convex hull of $\pi_\M(z_0)$ 
\end{enumerate}
\end{lemma}

\begin{proof}
Fix $1\leq r\leq k-1$ and $i\in \Emb^k(M,\R^D)$ as in the statement.
As we assume that $i$ satisfies condition \ref{P0}, we know that any critical point $z$ of $i(M)$  is such that its projections $\pi_{i(M)}(z)$ are the vertices of a non-degenerate simplex of $\R^D$, and that in particular there can be no more than $D+1$ of them. We only have to show that each such $z$ must belong to the relative interior of this simplex for \ref{P1} to be true.

Let $z\in Z(i(M))$ with $\pi_{i(M)}(z)= \{x_1,\ldots, x_s\}$ for some $s=2,\ldots, D+1$. Let us assume, by way of contradiction, that $z$ does not belong to the relative interior of $\Conv(x_1,\ldots,x_s)$. As we have shown in the proof of Lemma \ref{lem:P2_P3} that $z$ does belong to $\Conv(x_1,\ldots,x_s)$, it must then belong to a strict face of the simplex, hence to the convex hull of a strict subset of its vertices. If $s=2$, we already find a contradiction (as $z$ cannot be equal to any of the $x_i$). Otherwise, up to reordering the vertices, we can assume that $z \in \Conv(x_1,\ldots,x_{s-1})$. As in the proof of Lemma \ref{lem:P2_P3}, we also have that the point $j^r_s(i)(m_1,\ldots, m_s) = (m_1,\ldots,m_s,x_1,\ldots,x_s,[i],\ldots,[i])$ and $z$ are such that $d(x_1,z) = \ldots = d(x_s,z)$ and that $\Im((df_i)_{m_i}) \perp (z-x_i)$ for $i=1,\ldots,s$. Thus, the point $j^r_s(i)(m_1,\ldots, m_s)$ must belong to $W_1^{(s)}\cap \Im(j^r_s(i))$.

But the strata of the Whitney stratified sets $ W_{1,\gamma}^{(s)}$ have codimension at least $sm+1$ (see Lemma \ref{lem:w3_structure}), and the map  $ j^r_s(i) : M^{(s)} \rightarrow  J_s^r(M,\R^D) $ intersects them transversally by hypothesis.
As $M^{(s)}$ is of dimension $sm$, the image $\Im(j^r_s(i))$ can only have empty intersections with the strata of each $ W_{1,\gamma}^{(s)}$, hence with $W_1^{(s)}$ as a whole. This yields a contradiction, thus we have shown that $i(M)$ verifies condition \ref{P1}.    
\end{proof}

\subsubsection{Condition \ref{P4}}\label{subsubsec:P4}

The somewhat delicate case of condition \ref{P4} is all that remains; to tackle it, we will need the results from Section \ref{sec:BSP}.
Similarly to how we proceeded earlier, we want to define a family of sets $W_4^{(s)} \subset J^r_s(M,\R^D)$ such that transversality to them will yield the Big Simplex Property \ref{BSP}, and consequently condition \ref{P4}.
 To do so, we will use the expression from Lemma \ref{lem:diff_projection_in_charts} of the differential of the projection in terms of charts of our $C^k$ compact manifold $M$, as it is compatible with the formalism of jets.

Let $m_0\in M$,  $x_0,z_0\in \R^D$ and $\phi : (V,u_0) \rightarrow (M, m_0)$ be any chart of $M$ around $m_0$. Let also   $f\in C^2(M,\R^D)$ be some function that maps $m_0$ to $x_0$ such that
$df_{m_0}$ has full rank, that $\Im(df_{m_0}) \perp z_0-x_0$, and that the quadratic form
$ v\mapsto \|d(f\circ \phi)_{u_0}(v)\|^2 - \dotp{z_0-x_0,d^2(f\circ \phi)_{u_0}(v,v)}$
is  positive definite. Note that these conditions depend only of the value and the first and second differentials of $f$ at $m_0$.
Then, Lemma \ref{lem:diff_projection_in_charts} guarantees the existence of $\delta>0$ and of a neighborhood $U\subset M$ of $m_0$ such that 
$f|_U : U \rightarrow \R^D$ is an embedding and that each $z\in B(z_0,\delta)$ admits a unique projection $p(z)$ on $\overline{f(U)}$ with $p(z)\in f(U)$. The lemma also gives an expression for the differential $d p_{z_0}:\R^D \rightarrow  \Im(df_{m_0})$ at $z_0$ of $p:B(z_0,\delta) \longrightarrow f(U)$, and this expression also only depends on the first two differentials of $f$ at $m_0$.
Hence, given $r\geq 2$, a jet $(m,x,[f])\in J^r(M,\R^D)$ such that $df_{m}$ has full rank  and a point $z\in \R^D$  such that $\Im(df_{m}) \perp z-x$, and assuming that $v\longmapsto \|d(f\circ \phi)_0(v)\|^2 - \dotp{z-x,d^2(f\circ \phi)_0(v,v)}$ is positive definite (for some chart $\phi : (V,0) \rightarrow (M, m)$), we can unambiguously define the differential at $z$ of the projection $p$ on the image of $f$, independently from the choice of representative $f\in [f]$.
\begin{definition}
    When properly defined, we write as $P_{[f],x,z}$ the differential at $z$ of the projection $p$ on the image of $f$ around $x$.
\end{definition}

Now we can formulate an analogue in the language of jets of the quadratic form $B$ from Lemma \ref{lem:formula_simplex_volume}.
\begin{definition}\label{def:Baz}
Let $r,s\geq 2$ and let $a = (m_1,\ldots,m_s,x_1,\ldots, x_s, [f_1], \ldots, [f_s]) \in J^r_s(M,\R^D)$ and $z\in \R^D$. Assume that  for $ i\in [s]$ the map $(df_i)_{m_i}$ is of rank $m$, the quadratic form $v\longmapsto \|d(f_i\circ \phi_i)_0(v)\|^2 - \dotp{z-x_i,d^2(f_i\circ \phi_i)_0(v,v)}$ is positive definite (for some chart $\phi_i : (V,0) \rightarrow (M, m_i)$) and $\Im((df_i)_{m_i})\perp (z-x_i)$.

We define the quadratic form 
\[ B_{a,z}(h) :=  \sum_{i,j =1}^{s} \det(A_{a,z,i}(h)^\top A_{a,z,j}(h)) \]
on $\R^D$, where $A_{a,z,i}(h)$ is the $D\times s$ matrix whose $j$-th column is equal to $x_j-z$ if $j\neq i$ and whose $i$-th column is $P_{[f_i],x_i,z}(h) - h $.
\end{definition}

We will typically consider the case where $z\in \Conv (x_1,\ldots,x_s)$ and $h \in \text{Vec}(x_1-z,\ldots, x_s -z)^\perp$, but this is not needed for the definition to make sense.

We can now finally formulate (the negation of) the Big Simplex Property \ref{BSP} in a way that is compatible with jets and thence define $W_4^{(s)}$.

\begin{definition}
For any $s,r\in \mathbb{N}$ with $s\geq 2$ and $2\leq r\leq k $, we let $(W_4^\circ)^{(s)}$ be the set of multijets $(m_1,\ldots,m_s,x_1,\ldots, x_s, [f_1], \ldots, [f_s])\in J^r_s(M,\R^D)$  that belong to $W_2^{(s)}$ and  such that
\begin{itemize}
    \item for all $i\in [s]$, $\mathrm{rank}((df_i)_{m_i})=m$;
    \item for all $i\in [s]$ and for any chart $\phi : (U,0) \rightarrow (M, m_i)$, the form $v\mapsto \|d(f_i\circ \phi)_0(v)\|^2 - \dotp{z-x_i,d^2(f_i\circ \phi)_0(v,v)}$ is positive definite;
    \item and the quadratic form $B_{a,z}$ is degenerate on $\mathrm{Vec}(x_1-z,\ldots,x_s-z)^\perp$.
\end{itemize}
Furthermore, we let $W_4^{(s)}$ be the closure of $(W_4^\circ)^{(s)}$ in $J^r_s(M,\R^D)$.
\end{definition}

\begin{lemma}\label{lem:w4_structure}
    For any $s,r\in \mathbb{N}$ with $2\leq s\leq  D+1$ and $2\leq r\leq k-1$, the set $W_4^{(s)}$ is a finite union of closed Whitney stratifiable subsets $W_{4,\nu}^{(s)}$ of  $J^r_s(M,\R^D)$ of codimension at least $ sm + 1$.    
\end{lemma}

\begin{proof}
The proof is yet again similar to that of Lemma \ref{lem:w2_structure}; as for Lemmas \ref{lem:w3_structure} and \ref{lem:w1_structure}, we omit details and focus on the differences of import. As in the previous proofs, we fix $s\in \{2,\dots,D+1\}$ and write $W_4=W_4^{(s)}$. 

The definition of $W_4^\circ$ is similar to that of $W_2$, but three conditions have been added:
the first is that for $(m_1,\ldots,m_s,x_1,\ldots, x_s, [f_1], \ldots, [f_s]) \in J^r_s(M,\R^D)$ the rank of $(df_i)_{m_i}$ must be $m$ for $i\in [s]$. This is an algebraic constraint in any adapted chart, and an open condition, and does not add any difficulty to the proof. The second is that for any $i\in [s]$ and any chart $\phi : (U,0) \rightarrow (M, m_i)$, the form $v\mapsto \|d(f_i\circ \phi)_0(v)\|^2 - \dotp{z-x_i,d^2(f_i\circ \phi)_0(v,v)}$ must be positive definite (where $z$, if it exists, is the only point in $ \Conv (x_1,\ldots,x_s)$ such that $ d(x_1,z) = \ldots = d(x_s,z)$). We know from the proof of Lemma \ref{lem:w3_structure} that this condition is algebraic in adapted charts, and it is also open. The third and last addition is that we ask that the form  $B_{a,z}$ be degenerate on $\text{Vec}(x_1-z,\ldots,x_s-z)^\perp $. If we prove that this condition is semialgebraic in adapted charts and that it does increase the codimension by one, the same arguments as in Lemmas \ref{lem:w2_structure}, \ref{lem:w3_structure} and \ref{lem:w1_structure} show that $W_4^\circ$ is locally semialgebraic of codimension $sm+1$, hence that so  is its closure $W_4$, which can then be decomposed as a union of finitely many closed Whitney stratified sets $\{W_{4,\nu}\}_{\nu \in N}$ of codimension at least $sm+1$ (for some finite set of indices $N$), as claimed.

Let $(m,x,[f])\in J^r(M,\R^D)$ and $z\in \R^D$ be such that the conditions of Lemma \ref{lem:diff_projection_in_charts} are satisfied by $x,f,z$. Let  $\phi : (V,0) \rightarrow (M, m)$ be a chart of $M$ around $m$, and let us place ourselves in the adapted chart of $J^r(M,\R^D)$ to which $\phi$ gives rise.
In Lemma \ref{lem:diff_projection_in_charts}, we defined the linear endomorphism $\Lambda : \Im(df_m)\rightarrow \Im(df_m)$ associated to $x,z,f$ by stating that $\Lambda(v)$ must be the only vector in $\Im(df_m)$ such that $$\dotp{z-x, d^2 (f\circ \phi)_{0}(d(f\circ \phi)_{0}^{-1}(v),d(f\circ \phi)_{0}^{-1}(w)) } = \dotp{\Lambda(v),w} $$
for any $w\in \Im(df_{m_0})$.
Let $A $ be the $D\times m$ matrix that represents $d(f\circ \phi)_0 : \R^m \rightarrow \R^D$, and let $B$ be the $m\times m$ matrix that represents the bilinear form $\R^m\times \R^m \rightarrow \R,  (v,w)\mapsto \dotp{z-x, d^2 (f\circ \phi)_{0}(v,w) }$. Then $\tilde{\Lambda} := d(f\circ \phi)_0^{-1}\circ \Lambda \circ d(f\circ \phi)_0 : \R^m \rightarrow \R^m$ is such that 
\[\dotp{z-x, d^2 (f\circ \phi)_{0}(v,w) } =  \dotp{d (f\circ \phi)_{0}(\tilde{\Lambda}(v)),d(f\circ \phi)_{0}(w)} \]
for any $v,w\in \R^m$, which in terms of matrices means that 
\[ w^\top B v = w^\top A^\top A \tilde{\Lambda}(v)\]
for all $v,w\in \R^m$, or equivalently
\[ (A^\top A)^{-1} B v = \tilde{\Lambda}(v)\]
(where we use the same notation for the map $\tilde{\Lambda}$ and the matrix that represents it in the chosen system of coordinates). 
We also see that the map $d (f\circ \phi)_{0}^{-1} \circ \pi_{\Im(d f_{m})} : \R^D \rightarrow \R^m$ is represented by the pseudo-inverse matrix
\[(A^\top A)^{-1}A^\top.\]
This expression is well-defined because $A$ is of full rank, and the coefficients of the resulting matrix are polynomial fractions in the coefficients of $A$.

Finally, remember that we defined $P_{[f],x,z}$  as the differential at $z$ on the image of $f$, which was shown in Lemma \ref{lem:diff_projection_in_charts} to be equal to
\[ P_{[f],x,z} = (\id- \Lambda)^{-1} \circ  \pi_{\Im(d f_{m})} : \R^D \rightarrow \Im(df_m).\]
This can be rewritten as
\begin{align*}
   P_{[f],x,z} & = (\id- \Lambda)^{-1} \circ  \pi_{\Im(d f_{m})}  =
   d (f\circ \phi)_{0}d (f\circ \phi)_{0}^{-1}(\id- \Lambda)^{-1}d (f\circ \phi)_{0}d (f\circ \phi)_{0}^{-1}\pi_{\Im(d f_{m})} \\&
   = d (f\circ \phi)_{0} (\id- d (f\circ \phi)_{0}^{-1}\Lambda d (f\circ \phi)_{0})^{-1}d (f\circ \phi)_{0}^{-1}\pi_{\Im(d f_{m})} \\&
   =d (f\circ \phi)_{0} (\id- \tilde{\Lambda})^{-1}d (f\circ \phi)_{0}^{-1}\pi_{\Im(d f_{m})} 
\end{align*}
which is thus represented by the $D\times D$ matrix 
\[A(I_m - (A^\top A)^{-1} B)^{-1}(A^\top A)^{-1}A^\top\]
whose coefficients are polynomial fractions in the coordinates of the adapted chart.

For $a =(m_1,\ldots,m_s,x_1,\ldots, x_s, [f_1], \ldots, [f_s]) \in J^r_s(M,\R^D)$ and $z\in \R^D$, we defined $B_{a,z}$ as
\[B_{a,z}(h) :=  \sum_{i,j =1}^{s} \det(A_{a,z,i}(h)^\top A_{a,z,j}(h))\]
for $h\in \R^D$, where $A_{a,z,i}(h)$ is the $D\times s$ matrix whose $j$-th column is equal to $x_j-z$ if $j\neq i$ and whose $i$-th column is $P_{[f_i],x_i,z}(h) - h $.
Sums, scalar products and determinants are all polynomial operations, and we have shown that the coefficients of the matrices representing the maps $P_{[f_i],x_i,z}$ are polynomial fractions in the coordinates, hence so are the coefficients of the $D\times D$ matrix $O$ representing $B_{a,z}$.
One can locally define a linear map $\iota : \R^{D-m} \rightarrow \R^D$ such that the image of $\iota$ is $E^\perp = \text{Vec}(x_2-x_1, \ldots, x_s-x_2)^\perp$ and such that it is represented by a matrix $J$ whose coefficients are polynomial fractions in the coordinates of $x_1,\ldots, x_s$; then $B_{a,z}$ is degenerate when restricted to $E^\perp$ if and only if the determinant of $J^\top O J$ is zero, and this determinant is a polynomial fraction in the coordinates of the adapted chart (possibly restricted to a smaller domain).
This proves  that asking that $B_{a,z}$ be degenerate is a semialgebraic constraint in the adapted charts.

We now turn to the codimension of $W_4^\circ$. In what follows, we restrict ourselves to an adapted chart of sufficiently small domain $U$ of $J^r_s(M,\R^D)$  that 
there exists a polynomial $P$ in the coordinates of the chart and in those of $z\in \R^D$ such that wherever $B_{a,z}$ is well-defined, it is degenerate on $E^\perp$ if and only if $P(a,z) = 0$
 (to simplify notations, we write $P(a,z)$ for the value of $P$ evaluated in the coordinates corresponding to $a$ and $z$).
Such a polynomial exists because we have shown that the determinant of $B_{a,z}$ can be locally expressed as a polynomial fraction, which is $0$ if and only if its numerator is $0$.

We only need to show that the intersection of $W_4^\circ$ with the domain of such a chart is of codimension at least $sm+1$: then the same arguments as in the proofs of Lemmas \ref{lem:w2_structure}, \ref{lem:w3_structure} and \ref{lem:w1_structure} are enough to conclude.
 In what follows, we do not always specify "in the coordinates of the chart" or "restricted to the domain of the chart" and write $f_i$ for the precomposition of $f_i$ with the coordinates of the chart.

We face two challenges: the first is geometric in nature. Indeed, it is not trivial that asking that the form $B_{a,z}$ be degenerate on $E^\perp$ truly increases the codimension by one; after all, the form $B_{a,z}$ \textit{is} degenerate on $\R^D$. Nonetheless, we will show that it is true. The second difficulty is that as in the proof of Lemma  \ref{lem:w2_structure}, the polynomial constraint $P(a,z)=0$ is not a priori everywhere transverse to the other constraints, as $P$ might have singularities. A few tricks are enough to circumvent this mostly technical problem.

Write $a\in U$ as $a=(a_{\leq 1},a_{\geq 2})$, where
\begin{align*}
    a_{\leq 1} &:=  (m_1,\ldots,m_s,x_1,\ldots, x_s, (df_1)_{m_1},  \ldots, (df_s)_{m_s} )\\
    a_{\geq 2} &= ((d^{2}f_1)_{m_1}, \ldots, (d^{2}f_s)_{m_s}, (d^{\geq 3}f_1)_{m_1},\ldots,(d^{\geq 3}f_s)_{m_s}),
\end{align*}
(up to some reordering of the coordinates). Let $a$ be such that for some (unique) $z\in \R^D$, we have $\dim \text{Aff}(x_1,\ldots,x_s) = s-1$, $ \|x_1 -z\|^2 = \ldots = \|x_s-z \|^2$, $\Im((df_i)_{m_i}) \perp (z-x_i)$ and $\text{rank}((df_i)_{m_i} ) = m$ for $i=1,\ldots,s$.
Let us consider all $a'\in U$ such that the sources $m_i'$, the targets $x_i'$ and the first differential $(df_i')_{m_i'}$ of $a'$ are equal to those of $a$, i.e.~such that $a'_{\leq 1} = a_{\leq 1}$ (we let the $a'_{\geq 2}$ vary freely). All the conditions listed above are true for all such $a'$.

The key observation is the following:  we can make the eigenvalues associated to the form 
$v\mapsto  \dotp{z-x_i,d^2(f_i')_{m_i}(v,v)}$ arbitrarily smaller than $0$ (and doing so will ensure, as a side effect, that  $v\mapsto \|d(f_i')_{m_i}(v)\|^2 - \dotp{z-x_i,d^2(f_i')_{m_i}(v,v)}$ be positive definite). To summarize what follows, this makes the differential $P_{[f_i'],x_i,z}$ at $z$ of the projection on the image of $f_i'$  tend to $0$, which in turn makes $B_{a',z}(h)$  approximately equal to the square of the $s$-volume of the simplex $\Conv (x_1,\ldots,x_s,z+h)$ for $h\in E^\perp$, which is non-zero. This shows that the polynomial that associates the differentials of degree $2$ of $a'$ to $P(a',z)$ is non-trivial, thus that this polynomial is non-zero for most $a'$ and defines a codimension $1$ constraint. Now let us formalize this argument.

Remember that we defined the linear endomorphism $\Lambda : \Im(d(f_i')_{m_i})\rightarrow \Im(d(f_i')_{m_i})$ associated to $x_i,z,f_i'$ by stating that $\Lambda(v)$ must be the only vector in $\Im(d(f_i')_{m_i})$ such that (in the coordinates) 
\[\dotp{z-x, d^2 (f_i')_{m_i}((d (f_i')_{m_i})^{-1}(v),(d (f_i')_{m_i})^{-1}(w)) } = \dotp{\Lambda(v),w} \]
for any $w\in \Im(d(f_i')_{m_i})$.
As the eigenvalues associated to the form 
$v\mapsto  \dotp{z-x_i,d^2(f_i')_{m_i}(v,v)}$ go to minus infinity, so must the eigenvalues of $\Lambda$, hence the linear map
\[ P_{[f_i'],x_i,z} = (\id- \Lambda)^{-1} \circ  \pi_{\Im(d (f_i')_{m_i})} : \R^D \rightarrow \Im(d (f_i')_{m_i})\]
tends to $0$.
Then $B_{a',z}$ tends to
\[h\mapsto \sum_{i,j =1}^{s} \det(\tilde{A}_{a,z,i}(h)^\top \tilde{A}_{a,z,j}(h)),\]
where $ \tilde{A}_{a,z,i}(h)$ is the $D\times s$ matrix whose $j$-th column is equal to $x_j-z$ if $j\neq i$ and whose $i$-th column is $ - h $.
Now the same arguments as in the proof of Lemma \ref{lem:formula_simplex_volume} show that 
\begin{align*}
\sum_{i,j =1}^{s} \det(\tilde{A}_{a,z,i}(h)^\top \tilde{A}_{a,z,j}(h))
= (s!)^2\Vol _{s}(\Conv (x_1,\ldots, x_s, z+h))^2 + o(\|h\|^2).
\end{align*}
But as $\dim \text{Aff}(x_1,\ldots,x_s) = s-1$, if $h\in E^\perp =  \text{Vec}(x_1-z,\ldots,x_s-z)^\perp$ then $\Vol _{s}(\Conv (x_1,\ldots, x_s, z+h)) \geq C\|h\|$ for some constant $C>0$ that depends on the positions of $x_1,\ldots, x_s\in \R^D$. This shows that $B_{a',z}$ is non-degenerate if the eigenvalues associated to $v\mapsto  \dotp{z-x_i,d^2(f_i')_{m_i}(v,v)}$  are negative enough.


For the same fixed $a=(a_{\leq 1}, a_{\geq 2}) \in U$ and associated $z\in \R^D$, the map 
\[ P_{a_{\leq 1}}:a'_{\geq 2} \mapsto P((a_{\leq 1}, a'_{\geq 2}),z)\]
is polynomial in the coordinates. As $B_{a',z}$ is non-degenerate for an appropriate choice of $a'_{\geq 2}$, the polynomial $P_{a_{\leq 1}}$ is non-trivially $0$, hence 
\[ S_{a_{\leq 1}}:=\{a'_{\geq 2} :\; B_{(a_{\leq 1}, a'_{\geq 2}),z} \text{ is well-defined and } P_{a_{\leq 1}}(a'_{\geq 2} ) = 0\}\]
is an algebraic hypersurface (possibly with singularities) in the open set 
\[ \{a'_{\geq 2} :\; B_{(a_{\leq 1}, a'_{\geq 2}),z} \text{ is well-defined}\},\]
and in particular a semialgebraic set of codimension (at least) $1$.

Consider the projection 
$$\pi_{\leq 1} : (a_{\leq 1},a_{\geq 2}) \mapsto a_{\leq 1}$$
defined on the domain $U$ of the adapted chart. The image of the restriction of $\pi_{\leq 1}$ to $W_4^\circ \cap U$ is included
in the set $W_2 = W_2^{(s)}$ of Definition \ref{def:W2} for $r=1$. We know from Lemma \ref{lem:w2_structure} that $W_2$ is a finite union of Whitney stratified sets of codimension $sm$ in  $J^1_s(M,\R^D)$, and we can in fact assume that the image of $\pi_{\leq 1}|_{W_4^\circ \cap U}$ is included in one of those Whitney stratified sets by making $U$ small enough.
An easy consequence of Theorem 2.7.1 from \cite{RislerBenedettiAlgAndSemialg} is that if $f:A\rightarrow B$ is a semialgebraic map between $A$ and $B$ two semialgebraic sets, the dimension of $A$ is less than or equal to the dimension of $B$ plus the maximum of the dimensions of the fibers $f^{-1}(b)$ (for $b\in B$).
Now $\pi_{\leq 1} : W_4^\circ \cap U \rightarrow W_2 $ is a semialgebraic map, and we have just shown that the fibers $(\pi_{\leq 1}|_{W_4^\circ \cap U})^{-1}(a_{\leq 1})$ for $a_{\leq 1}\in \Im(\pi_{\leq 1}|_{W_4^\circ \cap U})$  have codimension at least $1$. Hence $W_4^\circ \cap U$ has codimension at least $sm+1$, which is enough to conclude.
\end{proof}

As for Lemmas \ref{lem:w2_structure}, \ref{lem:w3_structure} and \ref{lem:w1_structure}, we choose  for a given $M$ and for each $s,r\in \mathbb{N}$ with $2\leq s\leq D+1$ and $2\leq r\leq k-1$ a decomposition $W_4^{(s)} = \cup_{\nu\in N} W_{4,\nu}^{(s)}$ of $W_4^{(s)}$ into a finite union of closed Whitney stratifiable sets of codimension at least $sm +1$ whose existence is guaranteed by Lemma \ref{lem:w4_structure}, and we simply refer to it as ``the decomposition from Lemma \ref{lem:w4_structure}" or ``the sets $\{W_{4,\nu}\}_{\nu \in N}$ from  Lemma \ref{lem:w4_structure}"  in what follows.

\begin{lemma}\label{lem:P4}
For any $2\leq r \leq k-1$, if $i\in \Emb^k(M,\R^D)$ is such that $\M = i(M) $ satisfies condition \ref{P1}, \ref{P2} and \ref{P3} and is such that for all $s=2,\ldots, D+1$, $ j^r_s(i) : M^{(s)} \rightarrow  J_s^r(M,\R^D) $ is transverse to the sets $\{W_{4,\nu}^{(s)}\}_{\nu\in N}$ from Lemma \ref{lem:w4_structure}, then $\M$ satisfies condition 
\begin{enumerate}[start=4, label={(P\arabic*)}]
\item There exist constants $C>0$ and $\mu_0\in (0,1)$ such that for every $\mu\in [0,\mu_0)$, the set $Z_\mu(\M)$ is included in a tubular neighborhood of size $C\mu$ of $Z(\M)$, that is every point of $Z_\mu(\M)$ is at distance less than $C\mu$ from $Z(\M)$.
\end{enumerate}

\end{lemma}

\begin{proof}
    Fix $2\leq r\leq k\-1$ and $i\in \Emb^k(M,\R^D)$ as in the statement.
    Our goal is to show that each critical point of $i(M)$ satisfies the criterion in Corollary \ref{cor:BSP_equivalent_BI_non_degenerate}; then the corollary will guarantee that $i(M)$ satisfies Condition \ref{P4}.

    Let $z_0\in Z(\M)$ with $\pi_\M(z_0) = \{x_1,\ldots,x_s\}$ for some $s\in \{2,\ldots, D+1\}$, and $E=\mathrm{Vec}(x_j-x_l:\ j,l\in [s]\}$. Let $\delta>0$ be small enough that there exist maps $p_j: B(z_0,\delta) \rightarrow B(x_j,\delta')\cap \M$ that satisfy the conclusions of Lemma \ref{lem:local_projections_well_defined} for $j=1,\ldots,s$ (and for some $\delta'>0$). We have to prove that there exist constants $\delta>\tilde{\delta}>0$ and $L>0$ such that for any $h\in E^\perp$ with $\|h\|< \tilde{\delta}$, 
the $s$-volume of the  $s$-simplex $\Delta(h):= \Conv (\{z_0+h\}\cup \{p_j(z_0+h) :\; j\in [s]\})$ satisfies
\begin{equation}
\Vol _{s}(\Delta(h))\geq L \|h\|.
\end{equation}
As in the proof of Lemma \ref{lem:P2_P3}, the point  $a = (m_1,\ldots,m_s,x_1,\ldots, x_s, [i], \ldots, [i]) \in J^r_s(M,\R^D)$ (where $i(m_j) = x_j$) is such that 
$d(x_1,z_0) = \ldots = d(x_s,z_0)$, $z_0\in \Conv (x_1,\ldots,x_s)$, the rank of $di_{m_j}$  is $m$, and for any chart $\phi : (U,0) \rightarrow (M, m_j)$ the quadratic form $\R^m\rightarrow\R, v\mapsto \|d(i\circ \phi)_0(v)\|^2 - \dotp{z_0-x_j,d^2(i\circ \phi)_0(v,v)}$ is positive definite (for $j=1,\ldots,s$).
The last point is true because we have assumed that $i$ satisfies Condition \ref{P3}, which is equivalent to this property (as was shown in Lemma \ref{lem:characterization_osculation}).
If the form $B_{a,z_0}$ from Definition \ref{def:Baz} was degenerate on $E^\perp$, the point $a$ would satisfy all the conditions of the definition of $(W_4^\circ)^{(s)}$ and hence would belong to $W_4^{(s)}\cap \Im(j^r_s(i))\subset 
J^r_s(M,\R^D)$. But by hypothesis $ j_s^r(i) : M^{(s)} \rightarrow  J_s^r(M,\R^D) $ is transverse to the Whitney stratifiable sets $\{W_{4,\nu}^{(s)}\}_{\nu \in N}$ and those are of codimension at least $sm+1$; hence the intersection of $\Im(j^r_s(i))$ with $W_{4,\nu}^{(s)}$ is trivial for any $\nu\in N$, and so is its intersection with $W_4^{(s)}$. Consequently, the form $B_{a,z_0}$ is non-degenerate on $E^\perp$.

Remember that by definition, the form $B_{a,z_0}$ is defined as
$$B_{a,z_0}(h) :=  \sum_{j_1,j_2 =1}^{s} \det(A_{a,z_0,j_1}(h)^\top A_{a,z_0,j_2}(h))$$
for $h\in \R^D$, where $A_{a,z_0,j}(h)$ is the $D\times s$ matrix whose $l$-th column is equal to $x_l-z_0$ if $l\neq j$ and whose $j$-th column is $P_{[i],x_j,z_0}(h) - h $, with $P_{[i],x_j,z_0}$ being the differential at $z_0$ of the projection on a small neighborhood of $x_j$ in $\M$ (as in the statement of Lemma \ref{lem:diff_projection_in_charts}).
This projection  coincides with $p_j: B(z_0,\delta) \rightarrow B(x_j,\delta')\cap \M$,
hence by definition of the form  $B$ from Equation \ref{eq:first_form} in Lemma \ref{lem:formula_simplex_volume} we have $B(h) = B_{a,z_0}(h)$ for all $h\in E^\perp$ with $\|h\|<\overline{\delta}$ for some $0<\overline{\delta}<\delta$.
By applying Corollary \ref{cor:BSP_equivalent_BI_non_degenerate} and using the same reasoning for each critical point of $\M$, this is enough to conclude that $\M$ satisfies Condition \ref{P4}. 
\end{proof}

\subsection{Density of the transversality conditions}\label{subsection:density}

In the previous subsection, we  defined five sets in the space of multijets that can be written as finite unions of closed Whitney stratified sets, and shown that the transversality to those sets of the map $j^r_s(i) :M^{(s)} \rightarrow  J^r_s(M,\R^D)$ (respectively the map $ (j^r_s)_M\rho_i : M^{(s)}\times \R^D \rightarrow  J^r_s(M,\R) $) induced by an embedding $i:M\rightarrow \R^D$ implies that $i(M)$ satisfies the conditions \ref{P1}-\ref{P4} that are of interest to us.\footnote{The property of transversality must be verified for either $s=2,\ldots, D+2$, $s=2,\ldots, D+1$ or $s=3,\ldots, D+1$ depending on the set considered.}
It remains to show that the set of embeddings for which these properties of transversality are satisfied is dense in the Whitney topology; with this, we will have proved the second half of the Genericity Theorem \ref{thm:generic}, as well as Remark \ref{rmk:C1_also_ok}.
To that end, the transversality theorems presented in Subsection \ref{subsection:transversality_thms} are the key ingredients.

\begin{thm}[The Density Theorem]\label{thm:genericity_thm_density}
     Let $M$ be a compact $C^k$ manifold for some $k\in \N\cup\{\infty\}$ with $k\geq 1$. 
    Then, the set of $C^k$ embeddings $i:M\to \R^D$ such that $\M = i(M)$ satisfies  conditions \ref{P1}, \ref{P2}, \ref{P3} and \ref{P4} is dense in $\text{Emb}^k(M,\R^D)$ for the Whitney $C^k$-topology.
\end{thm}
\begin{proof}

 The only difficulty here is that the transversality theorems \ref{thm:Damon_Transversality} and \ref{thm:distance_function_transversality} only state  results regarding the density of $C^{r+1}$ functions whose multijets are transverse to sets in $J^r_s(M,\R)$ and $J^r_s(M,\R^D)$, as opposed to  $C^r$ functions, and relatedly, we had to assume throughout Subsection \ref{subsection:sets_Wi} that $M$ was a $C^k$ manifold for $k\geq 3$.
 Hence we have to be somewhat cautious to account for the cases $k=1,2$, as proceeding naively would force us to ask that $M$ be at least $C^3$. The key ingredient here will be that $C^{r_1}$ maps are dense in $C^{r_2}$ maps for $r_1\geq r_2$ and for the Whitney $C^{r_2}$-topology.

 Throughout the proof, the sets $W_0^{(s)}, W_1^{(s)}, W_2^{(s)}, W_3^{(s)}$ and $W_4^{(s)}$ are as defined in Subsection \ref{subsection:sets_Wi}.
 Let us start by assuming that $ k\geq 3$. 
 For $r=2$ and a given $s=2,\ldots, D+2$, the set $W_0^{(s)} \subset J^2_s(M)$ is a closed submanifold (hence in particular a closed Whitney stratifiable set). As $k\geq r+1 = 3$, we can apply the transversality theorem adapted to the case of distance functions (\Cref{thm:distance_function_transversality}) to $W_0^{(s)}$ (simply with $A = M^{(s)}\times \R^D$, in the notations of the theorem) to get that the set 
 \[F_{W_0^{(s)}} := \{f \in C^{k}(M,\R^D) :\; (j^2_s)_{M}\rho_f \text{ is transverse to }W_0^{(s)} \text{ in } J_s^2(M,\R)\}\]
 is residual in $ C^{k}(M,\R^D)$ (note that we have not yet restricted ourselves to embeddings).
 Similarly, we have shown in Lemmas \ref{lem:w2_structure}, \ref{lem:w3_structure}, \ref{lem:w1_structure} and \ref{lem:w4_structure} that $W_1^{(s)}, W_2^{(s)},W_3^{(s)},W_4^{(s)} \subset J^2_s(M,\R^D)$ are finite unions of closed Whitney stratifiable subsets  $\{W_{2,\alpha}^{(s)}\}_{\alpha \in A}$, $\{W_{3,\beta}^{(s)}\}_{\beta \in B}$, $\{W_{1,\gamma}^{(s)}\}_{\gamma \in C}$ and $\{W_{4,\nu}^{(s)}\}_{\nu \in N}$ respectively.
 The transversality theorem \ref{thm:Damon_Transversality} (with $A = M^{(s)}$) then yields that for any $s=2,\ldots, D+1$ and $\alpha \in A$, the set 
 $$F_{W_{2,\alpha}^{(s)}} := \{f \in C^{k}(M,\R^D) :\; j^2_s(f) \text{ is transverse to }W_{2,\alpha}^{(s)} \text{ in } J_s^2(M,\R)\}$$
 is residual in $ C^{k}(M,\R^D)$, and similarly for the sets 
  $$F_{W_{3,\beta}^{(s)}} := \{f \in C^{k}(M,\R^D) :\; j^2_s(f) \text{ is transverse to }W_{3,\beta}^{(s)} \text{ in } J_s^2(M,\R)\}$$
   and 
   $$F_{W_{4,\nu}^{(s)}} := \{f \in C^{k}(M,\R^D) :\; j^2_s(f) \text{ is transverse to }W_{4,\nu}^{(s)} \text{ in } J_s^2(M_{k},\R)\}$$
   for $\beta \in B$, $\nu \in N$ and $s=2,\ldots, D+1$ and the sets 
   $$F_{W_{1,\gamma}^{(s)}} := \{f \in C^{k}(M,\R^D) :\; j^2_s(f) \text{ is transverse to }W_{1,\gamma}^{(s)} \text{ in } J_s^2(M,\R)\}$$
   for $\gamma \in C$ and $s=3,\ldots, D+1$. 
 As $ C^{k}(M,\R^D)$ is a Baire space (see e.g.~\cite[Thm 4.4]{HirschDifferentialTopology}), the (finite) intersection
$$ F:= \bigcap_{s=2}^{D+2} F_{W_0^{(s)}}
\cap 
\bigcap_{s=3}^{D+1}\left(\bigcap_{\gamma \in C}F_{W_{1,\gamma}^{(s)}} \right) 
\cap 
\bigcap_{s=2}^{D+1}\left(\bigcap_{\alpha \in A}F_{W_{2,\alpha}^{(s)}} \cap \bigcap_{\beta \in B} F_{W_{3,\beta}^{(s)}} \cap \bigcap_{\nu\in N} F_{W_{4,\nu}^{(s)}} \right)$$
 of residual sets is yet again residual, and in particular dense in  $ C^{k}(M,\R^D)$.
As the set of embeddings $\text{Emb}^{k}(M,\R^D)$ is open in $ C^{k}(M,\R^D)$ (see \cite[Prop 5.3]{michor1980manifolds}), the set $E:= F \cap \text{Emb}^{k}(M,\R^D)$ is dense in $\text{Emb}^{k}(M,\R^D)$.

Let $i\in E$; then $(j^2_s)_{M}\rho_i$ is transverse to $W_0^{(s)}$ for $s=2,\ldots,D+2$, hence we know from Yomdin's Lemma \ref{lem:P1} that $i(M)$ satisfies Condition \ref{P0}. From Lemma \ref{lem:P2_P3} (which we can apply, as we have assumed that $i$ is $C^k$ with $k\geq 3 = r+1$), and as $j^2_s(i) $ is transverse to the sets  $\{W_{2,\alpha}^{(s)}\}_{\alpha \in A}$ and  $\{W_{3,\beta}^{(s)}\}_{\beta \in B}$ for $s=2,\ldots, D+1$, we get that $i(M)$ satisfies conditions \ref{P2} and \ref{P3}.
This, in addition to the transversality of $j^2_s(i) $  to $\{W_{1,\gamma}^{(s)}\}_{\gamma \in C}$ for $s=3,\ldots,D+1$ and to $\{W_{4,\nu}^{(s)}\}_{\nu \in N}$ for $s=2,\ldots,D+1$, allows us to apply Lemmas \ref{lem:P1} and \ref{lem:P4} (again, $i$ is at least $C^3$) to get that $i(M)$ also fulfills Conditions \ref{P1} and \ref{P4}.
Hence we have shown that the $C^k$ embeddings that satisfy conditions \ref{P1}-\ref{P4} are dense in $\text{Emb}^{k}(M,\R^D) $ for the Whitney $C^k$-topology, which completes the proof when $k\geq 3$.

Assume now that $k\in \{1,2\}$.
Using Theorem 2.9 from \cite[Chapter 2]{HirschDifferentialTopology}, we can equip $M$ with a $C^{3}$ differential structure that is compatible with its given $C^k$ differential structure. To make a clear distinction, we write $M_{3}$ to refer to $M$ equipped with this structure.
Let $U\subset \text{Emb}^{k}(M,\R^D) = \text{Emb}^{k}(M_{3},\R^D)  $ be any non-empty open set of the Whitney $C^k$-topology.
As the $C^{3}$ embeddings $\text{Emb}^{3}(M_{3},\R^D)$ seen as a subset of $\text{Emb}^{k}(M_{3},\R^D)$ are dense in $\text{Emb}^{k}(M_{3},\R^D)$ for the Whitney $C^k$-topology (see e.g.~\cite[Section 2, Thm 2.7]{HirschDifferentialTopology}), there is a $C^3$ embedding $i \in U\cap \text{Emb}^{3}(M_{3},\R^D)$.
As we have shown above that the $C^3$ embeddings $i'$ such that $i'(M)$ satisfies Conditions \Pall~are dense in $\text{Emb}^{3}(M_{3},\R^D)$ for the Whitney $C^3$-topology, we can find such embeddings  arbitrarily close to $i$ for the $C^3$-norm, hence also for the coarser $C^k$-norm. Such an embedding $i'$ is in particular a $C^k$ embedding for the initial structure on $M$, and if $\|i-i'\|_{C^k}$ is small enough, then $i'$ must belong to $U$.
Hence we can find $i' \in U$ such that $i'(M)$ satisfies Conditions \Pall, which is enough to conclude.
\end{proof}

\begin{remark}\label{rmk:openness_alternative_proof}
The observant reader might wonder why we did not use the ``openness" part of the transversality theorems \ref{thm:Damon_Transversality} and \ref{thm:distance_function_transversality} to directly get that the set of $C^k$ embeddings $i:M\to \R^D$ such that $ i(M)$ satisfies  conditions \ref{P1}, \ref{P2}, \ref{P3} and \ref{P4} is open in $\text{Emb}^k(M,\R^D)$ for the Whitney $C^k$-topology.
The even more observant reader remembers that Theorem \ref{thm:Damon_Transversality} guarantees that if $k\geq r+1$, the set of $C^k$ embeddings $i$ such that $j^r_s(i)$ is transverse on $A\subset M^{(s)}$ to some closed Whitney stratified set $W \subset J_s^r(M,\R^D)$ is open in  the Whitney $C^{r+1}$-topology under the condition that $A$ be compact (and similarly for Theorem \ref{thm:distance_function_transversality}).
Here we face two problems: the first is that we need transversality on $A =M^{(s)}$, which is not compact.
This hurdle can in fact be avoided by controlling the distance between the projections of critical points, though some work is required.
The second problem is that for our sets $W \subset J_s^r(M,\R^D)$ to be defined, we need $r\geq 2$, hence $k\geq 3$, and we could not find a way around this limitation;
hence, though we could have shown that the set of suitable embeddings is open in the Whitney $C^3$-topology using the two transversality theorems, we needed the more concrete and geometric proof of the $C^2$ Stability Theorem \ref{thm:C2_stab} to show that this set is also open in the Whitney $C^2$-topology.
\end{remark}

As a corollary of the Density Theorem and the $C^2$ Stability Theorem \ref{thm:C2_stab}, we finally get the Genericity Theorem \ref{thm:generic}.

\begin{proof}[Proof of the Genericity Theorem \ref{thm:generic}]
 The density part of the statement comes directly from the Density Theorem \ref{thm:genericity_thm_density}.
 From the $C^2$ Stability Theorem \ref{thm:C2_stab}, we get that the set of embeddings $i\in \Emb^2(M,\R^D)$ such that $i(M)$ satisfies \ref{P1}, \ref{P2}, \ref{P3} and \ref{P4} is open in $\Emb^2(M,\R^D)$ for the Whitney $C^2$-topology. As the Whitney $C^2$-topology on $\Emb^k(M,\R^D) \subset \Emb^2(M,\R^D)$ is the subspace topology inherited from  the Whitney $C^2$-topology  on $\Emb^2(M,\R^D)$, the set of $C^k $ embeddings such that $i(M)$ satisfies \ref{P1}, \ref{P2}, \ref{P3} and \ref{P4} is also open in $\Emb^k(M,\R^D)$.
\end{proof}

\bibliography{biblio_new}

\end{document}